\documentclass[11pt]{amsart}

\renewcommand{\Re}{\operatorname{Re}}
\renewcommand{\Im}{\operatorname{Im}}
\newcommand{\sech}{\operatorname{sech}}

\usepackage{amssymb}
\usepackage{amsthm}
\usepackage{amsxtra}
\usepackage[active]{srcltx}
\usepackage{graphicx}
\usepackage{color}
\usepackage{amsmath}
\usepackage{listings} 
\usepackage{float}
\usepackage{cases} 
\usepackage{threeparttable}
\usepackage{dcolumn}
\usepackage{multirow}
\usepackage{booktabs}
\usepackage{color}

\newcommand{\R}{\mathbb R}

\newtheorem{theorem}{Theorem}

\newtheorem{proposition}{Proposition}[section]
\newtheorem{lemma}[proposition]{Lemma}

\newtheorem{conjecture}{Conjecture}

\theoremstyle{remark}
\newtheorem{remark}[proposition]{Remark}

\numberwithin{equation}{section}

\title[Critical and supercritical SNLS]{
Behavior of solutions to the 1D focusing stochastic 
$L^2$-critical and supercritical
nonlinear Schr\"odinger equation with space-time white noise}

\author[A. Millet]{Annie Millet}
\address{SAMM (EA 4543), Universit\'e Paris 1, 90 Rue de Tolbiac, 75013 Paris Cedex (and LPSM)}
\curraddr{}
\email{annie.millet@univ-paris1.fr}

\author[S. Roudenko]{Svetlana Roudenko}
\address{Department of Mathematics \& Statistics\\Florida International University,  
Miami, FL 33199, USA}
\curraddr{}
\email{sroudenko@fiu.edu}

\author[K. Yang]{Kai Yang}
\address{Department of Mathematics  \& Statistics\\Florida International University,  Miami, FL 33199, USA}
\curraddr{}
\email{yangk@fiu.edu}

\subjclass[2010]{60H15,35R60, 65C30, 65M06} 

\keywords{stochastic NLS, space-time white noise, additive noise, multiplicative noise, blow-up dynamics, mass-conservative numerical schemes}

\date{}

\begin{document}
\begin{abstract}
We study the focusing stochastic nonlinear Schr\"odinger equation in 1D in the $L^2$-critical and supercritical cases with an additive or multiplicative perturbation driven by space-time white noise. Unlike the deterministic case, the Hamiltonian (or energy) is not conserved in the stochastic setting, nor is the mass (or the $L^2$-norm) conserved in the additive case. Therefore, we investigate the time evolution of these quantities. After that we study the influence of noise on the global behavior of solutions. In particular, we show that the noise may induce blow-up, thus, ceasing the global existence of the solution, which otherwise would be global in the deterministic setting. 
Furthermore, we study the effect of the noise on the blow-up dynamics in both multiplicative and additive noise settings and obtain profiles and rates of the blow-up solutions. Our findings conclude that the blow-up parameters (rate and profile) are insensitive to the type or strength of the noise: if blow-up happens, it has the same dynamics as in the deterministic setting, however, there is a (random) shift of the blow-up center, which can be described as a random variable normally distributed. 
\end{abstract}

\maketitle
\tableofcontents

\section{Introduction}
We consider the 1D focusing stochastic nonlinear Schr\"odinger (SNLS) 
equation, that is, the NLS equation subject to a random perturbation $f$ 
\begin{align}\label{E:NLS}
\begin{cases} 
iu_t+u_{xx}+|u|^{2\sigma}u= \epsilon  f(u), 
\quad (x,t)\in    [0,\infty) \times {\mathbb R}, \\ 
u(0,x)=u_0(x).
\end{cases}
\end{align}
Here, the term $f(u)$ stands for a stochastic perturbation driven by a space-time white noise $W(dt,dx)$ (described in Section \ref{S:hat}) and 
$u_0\in H^1(\mathbb R)$
is the deterministic initial condition.
In this paper we study the SNLS equation \eqref{E:NLS} with either an additive 
or a multiplicative perturbation  driven by  space-time white noise: 
its effect on the mass ($L^2$ norm) and energy (Hamiltonian), the influence of the noise on the global behavior of solutions 
and, in particular, its effect on the blow-up dynamics. In the deterministic setting the mass and the energy are typically conserved,
 however, these quantities may behave differently under stochastic perturbations, which might significantly change global behavior of solutions. 

The focusing stochastic NLS equation appears in physical models that involve random media, inhomogeneities or noisy sources. For example, the influence of the additive noise on the soliton propagation is studied in \cite{FKLT2001}, multiplicative noise in the context of Scheibe aggregates is discussed in \cite{RGBC1995}, the NLS studies in random media (via the inverse scattering transform) are discussed in \cite{G98}, \cite{AG2005} (and references therein). Relevant analytical studies of the SNLS \eqref{E:NLS} have been done by de Bouard \& Debussche in a series of papers \cite{dBD2001}, \cite{dBD2002c}, \cite{dBD2003}, \cite{dBD2005}, and numerical investigations by Debussche \& Di Menza and collaborators, can be found in \cite{DM2002a}, \cite{DM2002b}, \cite{BDM2005}.   

We consider two cases of the stochastic perturbation $f(u)$ in \eqref{E:NLS}: 
\begin{align}\label{D:f}
f(u)=
\begin{cases} 
 u(x,t) \circ {W}(dt,dx),  \quad \mathrm{multiplicative \,\, case,} \\
{W}(dt,dx), \qquad \mathrm{additive \,\, case}.
\end{cases}
\end{align}
The notation $ u(x,t) \circ W(dt,dx)$ stands for the Stratonovich integral, which makes sense when the noise is more regular (for example, when $W$ 
is replaced by its approximation $W_N$). This integral can be related to the It\^o integral (using the Stratonovich-It\^o correction term);  
for more details we refer the reader to \cite[p.99-100]{dBD2003}. 
The reason for the Stratonovich integral is the mass conservation, which we discuss next while recalling the properties of the deterministic NLS equation. 
  
The deterministic case of \eqref{E:NLS}, corresponding to $\epsilon=0$, has been intensively studied in the last several decades. 
The local wellposedness in $H^1$ goes back to the works of Ginibre and Velo \cite{GV1979}, \cite{GV1985}; see also \cite{K1987}, \cite{T1987}, \cite{CW1988},
 and the book \cite{Ca2003} for further details. During their lifespans, solutions to the deterministic equation \eqref{E:NLS} conserve several quantities, 
 which include the mass $M(u)$ and the energy (or Hamiltonian) $H(u)$ defined as 
$$
M(u(t))=\|u(t)\|_{L^2}^2 \equiv M(u_0) \quad \mbox{\rm and } \quad H(u(t))=\frac{1}{2} \| \nabla u(t)\|_{L^2}^2 - \frac{1}{2\sigma +2} \|u(t)\|_{L^{2\sigma +2}}^{2\sigma +2}\equiv H(u_0).
$$
The deterministic equation is invariant under the scaling: if $u(t,x)$ is a solution to \eqref{E:NLS} with $\epsilon=0$, then so is 
$u_\lambda(t,x) = \lambda^{1/\sigma} \, u(\lambda^2 t, \lambda x)$. This scaling makes the Sobolev $\dot{H}^{s}$ norm of the solution invariant 
with the scaling index $s$ defined as 
\begin{equation}\label{E:s}
s= \frac12-\frac1{\sigma}. 
\end{equation}
Thus, the 1D quintic ($\sigma=2$) NLS is called the $L^2$-critical equation ($s=0$). 
The nonlinearities higher than quintic (or $\sigma >2$) make the NLS equation $L^2$-supercritical ($s>0$)\footnote{The range of the critical index in 1D is $0<s<\frac12$.}; when $\sigma<2$, the equation is $L^2$-subcritical. 

In this work we mostly study the $L^2$-critical and supercritical SNLS equation \eqref{E:NLS} with quintic or higher powers of nonlinearity. In these cases, it is known that $H^1$ solutions may not exist globally in time (and thus, blow up in finite time), which can be shown by a well-known convexity argument on a finite variance (\cite{VPT1970}, \cite{Za1972}, for a review see \cite{SS1999}). 
We next recall the notion of standing waves, that is, solutions to the deterministic NLS of the form $u(t,x) = e^{it} Q(x)$. Here, $Q$ is a smooth positive decaying at infinity solution to 
\begin{equation}\label{E:Q}
-Q+Q^{\prime\prime}+Q^{2\sigma+1} = 0.  
\end{equation}
This solution is unique and is called the ground state (see \cite{We1983} and references therein). In 1D this solution is explicit: $Q(x) = (1+\sigma)^{\frac1{2\sigma}} \, \sech^{\frac1{\sigma}}(\sigma x)$.

In the $L^2$-critical case ($\sigma=2$) the threshold for globally existing vs. finite time existing solutions was first obtained by Weinstein \cite{We1983}, 
showing that if $M(u_0) < M(Q)$, then the solution $u(t)$ exists globally in time\footnote{and scatters to a linear solution in $L^2$, see \cite{D2015} and references therein.}; 
otherwise, if $M(u_0) \geq M(Q)$, the solution $u(t)$ may blow up in finite time. The minimal mass blow-up solutions (with mass equal to $M(Q)$) 
would be nothing else but the pseudoconformal transformations of the ground state solution $e^{it}Q$ by the result of Merle \cite{M1993}.
While these blow-up solutions are explicit, they are unstable under perturbations. 
The known stable blow-up dynamics is available for solutions with the initial mass larger than that of the ground state $Q$, and has a rich history, see \cite{YRZ2018}, \cite{YRZ2020},  \cite{SS1999}, \cite{F2015} (and references therein); the key features are  recalled later.  

In the $L^2$-supercritical case ($s>0$) the known thresholds for globally existing vs. blow-up in finite time solutions depend on the scale-invariant quantities such 
as $\mathcal{ME}(u):=M(u)^{1-s} H(u)^{s}$ and $\|u\|_{L^2}^{1-s} \|\nabla u(t)\|_{L^2}^s$, where the former is conserved in time and the latter changes 
the $L^2$-norm of the gradient. The original dichotomy was obtained in the fundamental work  by Kenig and Merle \cite{KM2006} 
in the energy-critical case ($s=1$ in dimensions 3,4,5), where they introduced the concentration compactness and rigidity approach to show the scattering behavior (i.e., approaching a linear evolution) for the globally existing solutions under the energy threshold (i.e., $E(u_0)<E(Q)$ in the energy-critical setting). It was extended to the intercritical case $0<s<1$ in \cite{HR2007}, \cite{DHR2008}, \cite{Gu2014}, followed by many other adaptations to various evolution equations and settings. A combined result for $0\leq s \leq 1$ is the following theorem (here, $X = \{ H^1 ~ \mbox{if} ~ 0<s<1;~ L^2 ~\mbox{if}~ s=0; ~\dot{H}^1 ~ \mbox{if} ~ s=1 \}$, for simplicity stated for zero momentum).
\begin{theorem}[\cite{KM2006}, \cite{HR2007}, \cite{HR2007}, \cite{DHR2008}, \cite{HR2010b},\cite{Gu2014}, \cite{FXC2011}, \cite{D2015}]\label{T:1} 
Let $u_0\in X(\R^N)$
and $u(t)$ be the corresponding solution to the deterministic NLS equation \eqref{E:NLS} ($\epsilon=0$) with the maximal interval of existence $(T_*, T^*)$. Suppose that $M(u_0)^{1-s} E(u_0)^s  < M(Q)^{1-s} E(Q)^s $. 
\begin{itemize}
\item 
If $\|u_0\|_{L^2}^{1-s} \|\nabla u_0\|_{L^2}^s < \|Q\|_{L^2}^{1-s} \|\nabla Q\|_{L^2}^s$, then
$u(t)$ exists for all $t\in \mathbb R$ with $\|u(t)\|_{L^2}^{1-s} \|\nabla u(t)\|_{L^2}^s < \|Q\|_{L^2}^{1-s} \|\nabla Q\|_{L^2}^s$ and $u(t)$ scatters in $X$: there exist $u_{\pm}\in X$ such that $\lim\limits_{t\rightarrow\pm\infty}\|u(t)-e^{it\Delta}u_{\pm}\|_{X(\R^N)}=0$.
\item 
If $\|u_0\|_{L^2}^{1-s} \|\nabla u_0\|_{L^2}^s > \|Q\|_{L^2}^{1-s} \|\nabla Q\|_{L^2}^s$, 
then $\|u(t)\|_{L^2}^{1-s} \|\nabla u(t)\|_{L^2}^s > \|Q\|_{L^2}^{1-s} \|\nabla Q\|_{L^2}^s$
for $t\in(T_*, T^*)$. Moreover, if $|x|u_0\in L^2(\R^N)$ (finite variance) or $u_0$ is radial, then the solution blows up in finite time; if $u_0$ is of infinite variance and nonradial ($s>0$), then either the solution blows up in finite time or there exits a sequence of times $t_n\rightarrow +\infty$ (or $t_n\rightarrow -\infty$) such that $\|\nabla u(t_n)\|_{L^2(\R^N)}\rightarrow \infty$.
\end{itemize}
\end{theorem}

The focusing NLS equation subject to a stochastic perturbation has been studied in \cite{dBD2003} in the $L^2$-subcritical case, showing 
a global well-posedness for any $u_0 \in H^1$. Blow-up for $0 \leq s <1$ 
has been studied in \cite{dBD2002c} for an additive perturbation, and  \cite{dBD2005}  for a multiplicative noise. The results in \cite{dBD2005} state that in the multiplicative noise case for $s \geq 0$ initial conditions with finite (analytic) variance and sufficiently negative energy
blow up before some finite time $t>0$ with positive probability \cite[Thm 4.1]{dBD2002c}. For both additive and multiplicative noise, in the $L^2$-supercritical case the authors prove that if noise is non-degenerate and regular enough as initial conditions, then blow-up happens with positive probability before
a given fixed time $t>0$ (see further details in \cite[Thm 1.2]{dBD2002c}, which also discusses the $L^2$-critical situation in the additive case, and \cite[Thm 5.1]{dBD2005}). This differs from the deterministic setting, where no blow-up occurs for initial data strictly smaller
than the ground state (in terms of the mass).

In \cite{MR2020} an adaptation of the above Theorem \ref{T:1} is obtained  
to understand the global behavior of solutions in the stochastic setting in the $L^2$-critical and supercritical cases. 
One major difference is that mass and energy are not necessarily conserved in the stochastic setting.
In the SNLS equation with multiplicative noise (defined via Stratonovich integral) the mass is conserved a.s. (see \cite{dBD2003}), which allows to prove global 
existence of solutions in the $L^2$-critical setting with $M(u_0) < M(Q)$ (see \cite{MR2020}). (A somewhat similar situation happens in the additive noise case, 
though mass is no longer conserved and actually grows linearly in time (see \eqref{mass_add}.) 
To understand global behavior  in the $L^2$-supercritical setting  one needs to control energy (as can be seen from Theorem \ref{T:1}). While it is possible to 
obtain some upper bounds on the energy on a (random) time interval (and in the additive noise it is also necessary to localize the mass on a random set, since it is not conserved), the exact behavior of energy is not clear.  This is exactly what we investigate in this paper via discretization of both quantities (mass and energy) in various contexts, then obtaining estimates on the discrete analogs and tracking the dependence on several parameters. 
Once we track the growth (and leveling off in the multiplicative case) of energy
(and mass in the additive setting), we study the global behavior of solutions. In particular, we investigate the blow-up dynamics of solutions in both $L^2$-critical and 
supercritical settings and obtain the rates, profiles and other features such as a location of blow-up. Before we state these findings, we review the blow-up in the deterministic setting. 

A stable blow-up in deterministic setting exhibits a self-similar structure with a specific rate and profile. 
Thanks to the scaling invariance, the following rescaling of the (deterministic) equation is introduced via the new space and time coordinates
 $(\tau, \xi)$ and a scaling function $L(t)$ 
(for more details see \cite{LePSS1987}, \cite{SS1999}, \cite{YRZ2019})
\begin{align}\label{E:rescale}
u(t,r)=\frac{1}{L(t)^{\frac{1}{\sigma}}}\,v(\tau, \xi), \quad \mbox{where} \quad \xi=\frac{r}{L(t)},~~r=|x|, \quad \tau=\int_0^t\frac{ds}{L(s)^2}.
\end{align}
Then the equation \eqref{E:NLS} in the deterministic setting ($\epsilon=0$) becomes
\begin{align}\label{E:v}
iv_{\tau}+ia(\tau)\left(\xi v_{\xi}+\frac{v}{\sigma}\right)+\Delta v + |v|^{2\sigma}v=0
\end{align}
with
\begin{align}\label{E:a}
a(\tau)=-L\frac{dL}{dt} \equiv -\frac{d \ln L}{d\tau}.
\end{align}
The limiting behavior of $a(\tau)$ as $\tau \to \infty$ (from the second term in \eqref{E:v})
creates a significant difference in blowup behavior between the $L^2$-critical and $L^2$-supercritical cases. 
As $a(\tau)$ is related to $L(t)$ via \eqref{E:a}, the behavior of the rate, 
$L(t)$,
 is typically studied to understand the blow-up behavior. Separating variables
 $v(\tau, \xi)=e^{i\tau}Q(\xi)$ in \eqref{E:v} and assuming that $a(\tau)$ converges to a constant $a$,  
the following problem is studied to gain information about the blow-up profile:
\begin{align}\label{E:profile}
\begin{cases}
\Delta_{\xi} Q -Q + ia\left(\dfrac{Q}{\sigma} + \xi Q_{\xi} \right) + |Q|^{2\sigma}Q=0,\\
Q_{\xi}(0)=0,\qquad Q(0)=\mathrm{real}, \qquad Q(\infty)=0.
\end{cases}
\end{align}
Besides the conditions above, it is also required to have $|Q(\xi)|$ decrease monotonically with $\xi$, without any oscillations as $\xi \to \infty$ 
(see more on that in \cite{YRZ2019}, \cite{SS1999}, \cite{BCR1999}). In the $L^2$-critical case the above equation is simplified (due to $a$ being zero)
 to the ground state equation \eqref{E:Q}. However, even in that critical context the equation \eqref{E:profile} is still meticulously investigated
  (with nonzero $a$ but asymptotically approaching zero), since the correction in the blow-up rate $L(t)$ comes exactly from that. 
  It should be emphasized that the decay of $a(\tau)$ to zero in the critical case is extremely slow, which makes it very difficult to pin down the exact blow-up rate,
   or more precisely, the correction term in the blow-up rate, and it was quite some time until rigorous analytical proofs appeared 
   (in 1D \cite{Pe2000a}, followed by a systematic work in \cite{MR2005}-\cite{FMR2006} and references therein; see review of this in \cite[Introduction]{YRZ2019} 
   or \cite{SS1999}).  In the $L^2$-supercritical case, the convergence of $a(\tau)$ to a non-zero constant is rather fast, 
   and the rescaled solution converges to the blow-up profile fast as well. The more difficult question in this case is the profile itself, 
   since it is no longer the ground state from \eqref{E:Q}, but exactly an admissible solution (without fast oscillating decay and with an asymptotic decay of 
   $|\xi|^{-\frac1\sigma}$ as $|\xi| \to \infty$) of \eqref{E:profile}. 

Among all admissible solutions to \eqref{E:profile} there is no uniqueness as it was shown in \cite{BCR1999}, \cite{KL1995}, \cite{YRZ2019}. These solutions generate branches of so-called multi-bump profiles, that are labeled $Q_{J,K}$, indicating that the $J$th branch converges to the $J$th excited state, and $K$ is the enumeration of solutions in a branch. The solution $Q_{1,0}$, the first solution in the branch $Q_{1,K}$ (this is the branch, which converges to the $L^2$-critical ground state solution $Q$ in \eqref{E:Q} as the critical index $s \to 0$), is shown (numerically) to be the profile of stable supercritical blow-up. The second and third authors have been able to obtain the profile $Q_{1,0}$ in various NLS cases (see \cite{YRZ2019}, also an adaptation for a nonlocal Hartree-type NLS \cite{YRZ2020}), and thus, we are able to use that in this work and compare it with the stochastic case. 


In the focusing SNLS case, in \cite{DM2002b} and \cite{DM2002a} 
numerical simulations are done when the driving noise is rough, namely, it is an approximation of space-time white noise. The effect of the additive and multiplicative noise is described for the propagation of solitary waves, in particular, it was noted that the blow-up mechanism  transfers energy from the larger scales to smaller scales, thus, allowing the mesh size affect the formation of the blow-up in the multiplicative noise case (the coarse mesh allows formation of blow-up and the finer mesh prevents it or delays it). The probability of the blow-up time is also investigated and found that in the multiplicative case it is delayed on average. In the additive noise case (where noise is acting as the constant injection of energy) the blow-up seems to be amplified and happens sooner on average, for further details refer to \cite[Section 4]{DM2002a}. Other parameters' dependence (such as on the strength $\epsilon$ of the noise) is also discussed. We note that the observed behavior of solutions as noted highly depends on the discretization and numerical scheme used.   

In this paper we design three numerical schemes to study the SNLS \eqref{E:NLS} driven by the space-time white noise. 
We then use these schemes to track the time dependence of mass and energy of the stochastic Schr\"odinger flow in each multiplicative and additive noise cases. After that  we investigate the influence of the noise on the blow-up dynamics. 
In particular, we give positive confirmations to the following conjectures. 

\begin{conjecture}[$L^2$-critical case]\label{C:1}
Let $u_0 \in H^1(\mathbb R)$ and $u(t)$, $t>0$, be an evolution of the SNLS equation \eqref{E:NLS} with $\sigma=2$ and noise \eqref{D:f}. 

In the multiplicative (Stratonovich) noise case, sufficiently localized initial data with $\|u_0\|_{L^2} > \|Q\|_{L^2}$ blows up in finite positive 
(random) time with positive probability.

In the additive noise case, sufficiently localized initial data blows up in finite (random) time a.s.

If a solution blows up at a random positive time $T(\omega)>0$ for a given $\omega \in \Omega$, then the blow-up is characterized by a
 self-similar profile (same ground state profile $Q$ from \eqref{E:Q} as in the deterministic NLS) and for $t$ close to $T(\omega)$ 
\begin{equation}\label{E:loglog}
\|\nabla u(t,\cdot) \|_{L^2_x} \sim \frac1{L(t)}, \quad L(t) \sim \left( \frac{2\pi(T-t)}{\ln|\ln(T-t)|} \right)^{\frac{1}{2}} \quad \mbox{as} \quad {t \to T=T(\omega)}, 
\end{equation}
known as the {\it log-log} rate due to the double logarithmic correction in $L(t)$. 

Thus, the solution blows up in a self-similar regime with profile converging to a rescaled ground state  profile $Q$, and the core part of the solution 
$u_c(x,t)$ behaves as follows
$$ 
u_c(t,x) \sim \dfrac{1}{L(t)^{\frac{1}{2}}} Q\left(\frac{x-x(t)}{L(t)}\right) e^{i\gamma(t)} 
$$
with parameters $L(t)$ converging as in \eqref{E:loglog}, $\gamma(t) \to \gamma_0$, and  $x(t) \to x_c$, the blow-up center $x_c$. 

Furthermore, conditionally on the existence of blow-up in finite time $T(\omega)> 0$,  $x_c$ is a Gaussian random variable;  no conditioning is
necessary in the additive case. 
\end{conjecture}

\begin{conjecture}[$L^2$-supercritical case]\label{C:2}
Let $u_0 \in H^1(\mathbb R)$ and $u(t)$ be an evolution of the SNLS equation \eqref{E:NLS} with $\sigma > 2$ and noise \eqref{D:f}. 

In the multiplicative (Stratonovich) noise case, sufficiently localized initial data blows up in finite positive (random) time with positive probability.

In the additive noise case, any initial data leads to a blow up in finite (random) time a.s.


If a solution blows up at a random positive time $T(\omega)>0$ for a given $\omega \in \Omega$, then the blow-up core dynamics $u_c(x,t)$ for $t$ close to $T(\omega)$ is characterized as
\begin{equation}\label{E:blowup-super}
u_c(t,x) \sim \dfrac{1}{L(t)^{\frac1{\sigma}}} Q\left(\frac{x-x(t)}{L(t)}\right) \exp \left({i \theta(t) + \frac{i}{2a(t)}\log \frac{T}{T-t}} \right),
\end{equation}
where the blow-up profile $Q$ is the $Q_{1,0}$ solution of the equation \eqref{E:profile}, $a(t) \to a$, the specific constant corresponding to the $Q_{1,0}$ profile, 
$\theta(t) \to \theta_0$, $x(t) \to x_c$, the blow-up center, and $L(t)=(2a(T-t))^{\frac12}$.
Consequently, a direct computation yields that for $t$ close to $T(\omega)$
\begin{equation}\label{E:rate-super}
\| \nabla u( t, \cdot)\|_{L_x^2} \sim \frac1{L(t)^{1-s}} \equiv {\left(2a(T-t) \right)^{-\frac12(\frac12+\frac1{\sigma})}}.
\end{equation}
Furthermore, conditionally on the existence of blow-up in finite time $T(\omega)> 0$,  $x_c$ is a Gaussian random variable; no conditioning is
necessary in the additive case.  

Thus, the blow-up happens with a polynomial rate \eqref{E:rate-super} without correction, and with profile converging to the same blow-up profile
 as in the deterministic supercritical NLS case.
\end{conjecture}

As it was mentioned above, some parts of the above conjectures have been studied and partially confirmed in \cite{DM2002a}, \cite{DM2002b}, \cite{dBD2005}, \cite{dBD2001} under various 
conditions. 
In this work we provide confirmation to both conjectures for various initial data (see also \cite{DM2002a}). We note that this paper is the first work, where the dynamics of blow-up solutions such as profiles, rates, location, are investigated in the stochastic setting.  

The paper is organized as follows. In Section \ref{S:2} we give a description of the driving noise and recall analytical estimates for mass and energy in both 
multiplicative and additive settings. In Section \ref{S:3} we introduce three numerical schemes which are mass-conservative in deterministic and 
multiplicative noise settings, and one of them is energy-conservative in the deterministic setting. We discretize mass and energy and give theoretical upper bounds 
on those discrete analogs in \S \ref{S:3.2} and \ref{S:3.3}; this is followed by the corresponding numerical results, which track both mass and energy in various 
settings, and time dependence on the noise type and strength, and other discretization parameters (such as length of the interval, space and time step-sizes). 
In the following Section \ref{S:4} we create a mesh refinement strategy and make sure that it also conserves mass before and after the refinement, 
introducing a new mass-interpolation method. We then state our new full algorithm for the numerical study of solutions behavior for both deterministic 
and stochastic settings. We note that this algorithm is novel even in the deterministic case for studying the blow-up dynamics (typically the dynamic rescaling
 or moving mesh methods are used). The new algorithm is needed due to the stochastic setting, since noise creates rough solutions, 
 compared with the deterministic case, and thus, the previous methods are simply not applicable.  
 In Section \ref{S:solitons} we start considering global dynamics (for example, of solitons, and how noise affects the soliton solutions) 
 and compare with the previously known results in the $L^2$-subcritical case.  
Finally, in Section \ref{S: blow-up} we study the blow-up dynamics in both the $L^2$-critical ($\sigma =2$) and $L^2$-supercritical (e.g., $\sigma=3$) cases.
 We observe that once a blow-up starts to form, the noise does not seem to affect either the blow-up profile or the blow-up rate. 
 The only affect that we have observed is random shifting of the blow-up center of the rescaled ground state. With increasing number of runs, the variation in 
 the center location appears to be distributed normally (we estimate the corresponding mean, which is very close to 0, and variance). 
 Otherwise, there seems to be very little difference between the multiplicative/additive noise and deterministic settings. 
 We give a summary of our findings in the last section.

{\bf Acknowledgments.} This work was partially written while the first author visited Florida International University. She would like to thank FIU for the hospitality and the financial support. A. M.'s research has been conducted within the FP2M federation (CNRS FR 2036). 
S.R. was partially supported by the NSF grant DMS-1815873/1927258 as well as part of the K.Y.'s research and travel support to work on this project came from the above grant.
\smallskip

\section{Description of noise and its effect on mass and energy}\label{S:2}

\subsection{Description of the driving noise} \label{S:hat}
The space-time white noise is defined in terms of a  real-valued  zero-mean Gaussian random field 
\[ \{ W(B)\; : \; B \; \mbox{\rm bounded measurable subset of  }\; [0,+\infty) \times \R\}\]
defined on a 
probability space $(\Omega, {\mathcal F},P)$, with covariance given by
$$
E\big[ W(B) W(C)\big]= \int_{B\cap C} dt\, dx
$$
for  bounded  measurable  subsets $B,C$ of  $[0,\infty)\times \R$. 
For $t\geq 0$ let 
$$ 
{\mathcal F}_t:= \sigma(W(B):  \; B \; \mbox{\rm bounded measurable subset of } \; [0,t]\times \R).
$$
Given $0\leq t_1<t_2$  and a step function $h_N=\sum_{l=1}^N a_l 1_{[y_l, y_{l+1})}$,  where 
$a_l\in \R$ and $y_1<y_2< \cdots < y_{N+1}$ are real numbers, we let $\int_{t_1}^{t_2}\int_{\R} h_N(x) W(ds,dx):= 
\sum_{l=1}^N a_l W([t_1,t_2]\times [y_l,y_{l+1}))$.
Given $0\leq t_1<t_2$ and a function $h \in L^2(\R ; \R)$, we can define the stochastic Wiener integral 
$ \int_{t_1}^{t_2}  \int_{\R} h(x)  W(dt,dx) $ as the $L^2(\Omega)$ limit of $\int_{t_1}^{t_2}  \int_{\R} h_N(x)  W(dt,dx) $ for any sequence
of step functions $h_N$ converging to $h$ in $L^2(\R ;\R)$. This stochastic integral is a centered Gaussian random variable with variance
$[t_2-t_1] \int_{\R} |h(x)|^2 dx$. Furthermore, if $h_1$, $h_2$ are orthogonal functions in $L^2(\R;\R)$ with $\| h_1\|_{L^2}=\|h_2\|_{L^2}=1$, 
the processes $\{ \int_0^t h_1(x) W(dt,dx)\}_{t\geq 0}$ and $\{ \int_0^t h_2(x) W(dt,dx)\}_{t\geq 0}$ are independent Brownian motions for
the filtration $({\mathcal F}_t , t\geq 0)$. 

Let $\{ e_j\}_{j \geq 0}$ be an orthonormal basis of $L^2(\R ;\R)$ and let $\beta_j(t)=\int_0^t \int_{\R} e_j(x) W(ds,dx)$, $j \geq 0$. The processes
$\{ \beta_j\}$ are independent one-dimensional Brownian motions and we can formally write
\begin{equation}\label{E:noise1}
W(t,x,\omega) =\sum_{j\geq 0}   \beta_j(t,\omega)  e_j(x), \quad t\geq 0, \; x\in \R, \;  \omega \in \Omega.
\end{equation}
However, the above series does not converge in $L^2(\R)$ for a fixed $t > 0$. 
To obtain an $L^2(\R ;\R)$-valued 
Brownian motion, we should replace the space-time white noise $W$ by a Brownian
motion white in time and colored in space. More precisely, in the above series we should replace $e_j$ by $\phi e_j$ for some operator $\phi$ in $L^{0,0}_{2,\R}$, 
which is a Hilbert-Schmidt operator from $L^2(\R)$ to $L^2(\R)$  
with the Hilbert-Schmidt norm $\|\phi\|_{L_{2,\R}^{0,0}}$. 
This would yield $\tilde{W}(t)=\sum_{j\geq 0} \beta_j(t) \phi e_j$ for some sequence $\{\beta_j\}_{j\geq 0} $ of independent one-dimensional Brownian motions 
and some orthonormal basis $\{e_j\}_{j\geq 0}$ of $L^2(\R;\R)$. The covariance operator  ${\mathcal Q}$  
of $\tilde{W}$ is of the trace-class with ${\rm Trace}\; {\mathcal Q}=\|\phi\|_{L_{2,\R}^{0,0}}^2$.

For practical reasons, we will use approximations of the space-time white noise  $W$ (thus, a more regular noise) using finite sums 
\begin{equation}\label{WN}
W_N(t,x,\omega):= \sum_{j=0}^N \beta_j(t,\omega)  e_j(x),
\end{equation}
with functions $\{e_j \}_j$ with disjoint supports, which are normalized in $L^2(\R ; \R)$.
This finite sum gives rise to an  $L^2(\R ;\R)$-valued  Brownian motion with the covariance operator ${\mathcal Q}$ such that ${\rm Trace}\; {\mathcal Q}=N+1$. 

Unlike \cite{DM2002a}, \cite{DM2002b},  we will not suppose that the functions $\{e_j\}$ are indicator functions of disjoint intervals.
 Instead we will consider the following ``hat" functions, which  belong to $H^1$.


For fixed $N\geq 1$ we consider the hat functions 
$\{g_j\}_{0\leq j \leq N}$ defined on the space interval $[x_j, x_{j+1}]$ as follows. 
Let $x_{j+\frac{1}{2}}:=\frac{1}{2}\big[ x_j+x_{j+1}]$, $\Delta x_j:= x_{j+1}-x_j$, 
and for $j=0, \cdots N-1$, set
$$
g_j(x):= \begin{cases}
c_j (x-x_j) \quad \mbox{\rm for} \quad x \in [x_j,x_{j+\frac{1}{2}}], \\
c_j (x_{j+1} -x) \quad \mbox{\rm for} \quad x \in [x_{j+\frac{1}{2}}, x_{j+1}], 
\end{cases}
$$ 
where $c_j:= \frac{2 \sqrt{3}}{(\Delta x_j)^{3/2}}$ 
is chosen to ensure  $\|g_j\|_{L^2}=1$. 
 
Given points $x_0< x_1< \cdots <x_N$, define the functions $e_j$, $j=0, \cdots, N$, by
\begin{align} \label{def_ej}
 \begin{cases}  e_j = &g_{j-1} 1_{[x_{j-\frac{1}{2}}, x_j]} + g_j 1_{[x_j, x_{j+\frac{1}{2}}]}, \;  1\leq j\leq N-1,  \\
  e_0=& \sqrt{2}  g_0 1_{[x_0, x_{\frac{1}{2}}]} , \quad e_N= \sqrt{2}  g_{N-1} 1_{[x_{N-\frac{1}{2}}, x_N]}.
  \end{cases}
\end{align} 
Due to the symmetry of the functions $\{g_j\}$, we have $\|e_j\|_{L^2}=1$ for $j=0, \cdots, N$. Since the functions $\{e_j\}$'s  have disjoint supports, 
they are orthogonal 
in $L^2(\R ;\R)$. We can now construct an orthonormal basis 
$\{e_k\}_{k\geq 0}$ of $L^2(\R ;\R)$ containing the above $\{e_j\}_{0\leq j\leq N}$ set as  the first $N+1$ elements.
 Then the $L^{0,0}_{2,\R}$-Hilbert-Schmidt norm of the orthogonal projection $\phi_N$ 
from $L^2(\R;\R)$ to the span of  $\{e_j\}_{0\leq j\leq N}$ is equal to $N+1$. 
  
Furthermore, unlike indicator functions, the above functions $\{e_j\}$ belong to $H^1$. 
When the mesh size $\Delta x_j$ is constant (equal to $\Delta x$), an easy computation yields $\|e_j\|_{H^1}^2=1+ \frac{12}{(\Delta x)^2}$ and 
$\| \nabla e_j\|^2_{L^\infty}= \frac{12}{(\Delta x)^3}$. 
Therefore, if $\| \cdot \|_{L^{0,1}_{2,\R}}$ denotes the Hilbert-Schmidt norm from $L^2(\R;\R)$ to
$H^1(\R;\R)$, we have  
\begin{equation}\label{HS_H1norm}
\|\phi_N\|_{L^{0,1}_{2,\R}}^2 = (N+1) \Big( 1+ \frac{12}{(\Delta x)^{2}}\Big) \sim 
\frac{12\, (N+1)}{(\Delta x)^{2}} \quad \mbox{when} \quad \Delta x <<1,
\end{equation}
and  
\begin{equation}\label{E:m-phi}
m_{\phi_N} := \sup_{x\in \R} \sum_{k\geq 0} \big| \nabla (\phi_N e_k)(x)\big|^2 = \frac{12\, (N+1)}{(\Delta x)^3}.
\end{equation}

\subsection{Multiplicative noise} 
We recall that this stochastic perturbation on the right-hand side of \eqref{E:NLS} is $f(u)= u(x,t) \circ {W}(dt,dx)$, where the multiplication understood via the Stratonovich integral, which makes sense for a more regular noise. 
When the noise $\tilde{W}=\sum_{j\geq 0} \beta_j \phi e_j$  is  regular in the space variable 
(that is, colored in space by means of the operator  $\phi \in L^{0,0}_{2,\R}$), 
the equation \eqref{E:NLS} conserves mass almost surely (see \cite[Proposition 4.4]{dBD2003}), i.e., for any $t>0$
\begin{align}\label{D: mass}
M[u(t)]=
M[u_0] \quad \mbox{\rm a.s.}
\end{align} 
Using the time evolution of  energy in the multiplicative case for a regular noise $\tilde{W}$ (see  \cite[Proposition 4.5]{dBD2003}, we have
\begin{align*}
H(u(t)) = & H(u_0) - \mbox{\rm Im } \epsilon \sum_{j\geq 0} \int_0^t  \int_{\R} \bar{u}(s,x)\nabla u(s,x) \cdot (\nabla \phi e_j)(x) \, dx d\beta_j(s)   \\
& +\frac{\epsilon^2}{2} \sum_{j\geq 0} \int_0^t \int_{\R} |u(t,x)|^2 \; |\nabla (\phi e_j)|^2 \, dx ds.  
\end{align*}
Taking expected values and using the fact that $\phi$ is Radonifying from $L^2(\R;\R)$ to $\dot{W}^{1,\infty}(\R;\R)$,  
we deduce that 
\begin{equation} \label{energy_mul}
{\mathbb E}(H\big( u(t)\big) = H(u_0) + \frac{\epsilon^2}{2} \, {\mathbb E}  \sum_{j\geq 0} \int_0^t \int_{\R}  |u(s,x)|^2   \big| \nabla (\phi e_j)(x)\big|^2 
dx ds
\leq H(u_0) + \frac{\epsilon^2}{2} m_\phi M(u_0)\, t,
\end{equation}
where 
\begin{equation} 
m_\phi:=  \sup_{x\in \R} \sum_{j\geq 0} | \nabla(\phi e_j)(x)|^2<\infty. 
\end{equation} 
We also consider the expected value of the supremum in time of the energy (Hamiltonian).
 However, the upper bound differs depending on the critical or supercritical cases. For exact statements and notation we refer the reader to
  \cite[Section 2]{MR2020}, where it is shown that for any stopping time
$\tau <\tau^*(u_0)$ (here, $\tau^*(u_0)$ is the {\it random} existence time  from the local theory), one has
\begin{align}\label{E:supH}
\mathbb E \Big(\sup_{s\leq \tau} H\big( u(s)\big) \Big) 
\leq  \mathbb E\big(H(u_0)\big)+\frac{\epsilon^2}{2} m_\phi \, M(u_0) \, \mathbb E(\tau)  \nonumber 
 \; + 3 \epsilon \, \sqrt{m_\phi M(u_0)} \,
\mathbb E\Big( \sqrt{\tau} \, \sup_{s\leq \tau} \| \nabla u(s) \|_{L^2(\mathbb R^n)}\Big). 
\end{align}
Therefore, the bound on the energy depends on the growth of the last gradient term. In the $L^2$-critical case, assuming that $\|u_0\|_{L^2} < \|Q\|_{L^2}$,
it is possible to control the kinetic energy  $\|\nabla u(s)\|_{L^2}^2$ in terms of the energy $H(u)$ (see e.g. \cite[(2.15)]{MR2020}). 
Therefore, $\tau^*(u_0)=+\infty$ a.s. (see \cite[Thm 2.8]{MR2020}), and thus, for large times the upper estimate for the growth of the energy is at most linear: for any $t>0$ we have
\begin{equation}\label{E:supH-1}
\mathbb E \Big(\sup_{s\leq t} H\big( u(s)\big) \Big) 
\leq  \mathbb E\big(H(u_0)\big)+ 
c_1 \epsilon^2 \, m_\phi M(u_0)\,  t + c_2 \, \epsilon \, \sqrt{m_\phi M(u_0)} \, \sqrt{t}.
\end{equation}

In the $L^2$-supercritical case it is more delicate to control the gradient;  nevertheless, it is possible for some (random) time interval (for which we provide 
 upper and lower bounds  in \cite[Thm~2.10]{MR2020}). The length of that time interval is inversely proportional to the strength of the noise $\epsilon$, 
 the space correlation  $m_{\phi}$, and the size of the initial mass $M(u_0)$ to some power depending on $\sigma$.

\subsection{Additive noise}
The additive perturbation in \eqref{E:NLS} is $f(u)= {W}(dt,dx)$. In this case, 
mass is no longer conserved.
It is easy to see that its expected value grows linearly in time. More precisely,
the identity 
$$ 
M\big( u(t)\big) = M(u_0) + \epsilon^2 \|\phi\|_{L_{2 \mathbb R}^{0,0}}^2\;   t- 2 \epsilon \, \mbox{\rm Im } \Big( \sum_{j= 0}^N  \int_0^t \!\!\int_{\R} u(s,x)
\overline{\phi e_j(x)} dx d\beta_j(s) \Big)
$$
(see e.g. \cite[page 106]{dBD2003} or \cite[(3.2)]{MR2020} ) implies  
\begin{equation} \label{mass_add}
{\mathbb E}(M(u(t)))=M(u_0) + \epsilon^2 \| \phi\|_{L^{0,0}_{2,\R}}^2 t.
\end{equation}
For the energy bound, using \cite[(3.3)]{MR2020} (see also \cite[Prop~3.3]{dBD2003}), we have 
\begin{align*}
H(u(t)) \leq & H(u_0) + \frac{\epsilon^2}{2} \tau \|\phi\|_{L_{2, \mathbb R}^{0,1}}^2 + \mbox{\rm Im }\epsilon  \Big(\sum_{k\geq 0} \int_0^\tau \!\! \int_{\R} 
\nabla \overline{u(s,x)} \nabla (\phi e_k)(x)\, dx \, d\beta_k(s) \Big)\nonumber \\
&\;   - \mbox{\rm Im } \epsilon \Big( \sum_{k\geq 0}  \int_0^\tau \!\! \int_{\R}  
| u(s,x)|^{2\sigma} \overline{u(s,x)} (\phi e_k)(x) dx d\beta_k(s) \Big).
\end{align*}
Taking expected values, we deduce the
following {\it linear} upper bound for the time evolution of the expected (instantaneous) energy 
\begin{equation}\label{energy_add}
{\mathbb E}\big( H(u(t))\big) 
\leq H(u_0) + \frac{\epsilon^2}{2} \| \phi\|_{L^{0,1}_{2,\R}}^2 t .
\end{equation}

As in the multiplicative case, in order to have quantitative information on the expected time of the existence interval, we have to prove upper bounds on  
$\mathbb{E} \big(\sup_{s<\tau} H(u(s))\big)$. However, since in the additive noise case the mass is not conserved and grows linearly in time, 
we have to localize the energy estimate
on a (random) set, where the mass can be controlled  
(for details see \cite[Section 3]{MR2020}. 
With that localization and estimates on the time, 
the upper bound for the expected energy is linear in time; furthermore, the time existence of solutions is inversely proportional to $\epsilon^2$ 
and the correlation  $m_{\phi}$.  

Next, we would like to investigate the mass and energy quantities numerically. For that we define discretized (typically referred to as {\it discrete}) analogs
 of mass and energy, we also introduce several numerical schemes, which we use to simulate solutions, and thus, track the above quantities. 
 We first prove theoretical upper bounds on the discrete mass and energy in both multiplicative and additive noise cases, 
 and then provide the results of our numerical simulations.

\section{Numerical approach}\label{S:3}

We start with introducing our numerical schemes for the SNLS \eqref{E:NLS}. 
We present three numerical schemes that conserve the discrete mass in the deterministic and with a multiplicative stochastic perturbation. Furthermore, 
one of them also conserves the discrete energy (in the deterministic case). That mass-energy conservative (MEC) scheme  is a highly nonlinear scheme, 
which involves additional steps of Newton iterations, slowing down the computations significantly and generating numerical  errors. 
We simplify that scheme first to the Crank-Nicholson (CN) scheme, which is still nonlinear, though works slightly faster. 
Then after that we introduce a linearized extrapolation (LE) scheme, that is much faster (no Newton iterations involved) while producing tolerable errors. 
Before describing the schemes, we first define
the finite difference operators on the non-uniform mesh.

\subsection{Discretizations}\label{S:3a}
\subsubsection{Finite difference operator on the non-uniform mesh} 
We start with the description of an efficient way to approximate the space derivatives $f_x$ and $f_{xx}$. 

Let $\left\lbrace x_j \right\rbrace_{j=0}^N$ be the grid points on $[-L_c,L_c]$ (the points $x_j$'s are not necessarily  equi-distributed). 
From the Taylor expansion of $f(x_{j-1})$ and $f(x_{j+1})$ around  $x_j$, setting $f_j=f(x_j)$ and  $\Delta x_j=x_{j+1}-x_j$, one has
\begin{align}\label{E: D1 u}
f_x(x_j) \approx
\frac{-\Delta x_{j}}{\Delta x_{j-1}(\Delta x_{j-1}+\Delta x_{j})} f_{j-1}
 +\frac{\Delta x_{j}-\Delta x_{j-1}}{\Delta x_{j-1}\Delta x_{j}} f_j  
+\frac{\Delta x_{j-1}}{(\Delta x_{j-1}+\Delta x_{j})\Delta x_{j}}f_{j+1}, 
\end{align}
and 
\begin{align}\label{E: D2 u}
f_{xx}(x_j) \approx
\frac{2}{\Delta x_{j-1}(\Delta x_{j-1}+\Delta x_{j})} f_{j-1}-\frac{2}{\Delta x_{j-1}\Delta x_{j}} f_j +\frac{2}{(\Delta x_{j-1}+\Delta x_{j})\Delta x_{j}}f_{j+1}.
\end{align}

We define the 
second order finite difference operator
\begin{align}\label{D: D2}
\mathcal{D}_2 f_j := \frac{2}{\Delta x_{j-1}(\Delta x_{j-1}
+\Delta x_{j})} f_{j-1}- \frac{2}{\Delta x_{j-1}\Delta x_{j}} f_j +\frac{2}{(\Delta x_{j-1}+\Delta x_{j})\Delta x_{j}}f_{j+1}.
\end{align}

\subsubsection{Discretization of space, time and noise}\label{S:discretization}
We denote the 
full discretization in both space and time by  $u_j^m:= u(t_m,x_j)$ at the $m^{\rm th}$ time step and the $j^{\rm th}$ grid point. 
We denote the size of a time step by $\Delta t_{m-1}=t_m-t_{m-1}$. 

To consider the Stratonovitch stochastic integral, we let $x_{j+\frac{1}{2}}= \frac{1}{2} \big[x_j+x_{j+1}\big]$, and we discretize the stochastic term 
in a  way  similar to that  in \cite{DM2002a}, 
except that we use the basis $\{e_j\}_{0\leq j\leq N}$ defined in
\eqref{def_ej} instead of the indicator functions. Recall that $\{ \beta_j(t) \}_{0 \leq j \leq N}$ are the associated  independent Brownian motions for 
the approximation $W_N$ of the noise $W$ (i.e., $\beta_j(t)= \int_0^t \int_{\R} e_j(x) W(dt,dx)$). 
Following a procedure similar to that in \cite{DM2002a}, we set  
$$
\chi^{m+\frac{1}{2}}_j=\frac{1}{\sqrt{\Delta t_m }}(\beta_{j}(t_{m+1})-\beta_{j}(t_{m})), \quad 0\leq j \leq N.
$$ 
We note that the random variables $\{ \chi^{m+\frac{1}{2}}_j\}_{j,m}$ are independent Gaussian random variables  ${\mathcal N}(0,1)$. 
In our simulation, the vector $(\chi_0^{m+\frac{1}{2}}, \cdots, \chi_N^{m+\frac{1}{2}})$ is obtained by the Matlab random number generator \texttt{normrnd}.
  
When computing a solution at the end points $x_0$ and $x_{N+1}$,  we set $u^m_0=u^m_{1}$ and $u^m_{N}=u^m_{N-1}$ for all $m$. 
We also introduce the pseudo-point $x_{-1}$ satisfying $\Delta x_{-1}=\Delta x_0$, and similarly, the pseudo-point $x_{N+1}$ satisfying $\Delta x_{N-1}=\Delta x_N$.
Let 
\begin{equation} \label{f-fhat}
f^{m+\frac{1}{2}}_j = \frac{1}{2} \big( u^m_j+u^{m+1}_j\big)  \tilde{f}^{\,m+\frac{1}{2}}_j , \quad \mbox{\rm where} \quad  \tilde{f}^{\,m+\frac{1}{2}}_j := \frac{\sqrt{3}}{2}  
  \frac{\big[ \sqrt{\Delta x_{j-1}}+\sqrt{\Delta x_j} \big] } {\sqrt{\Delta t_m} 
   \big[  \Delta x_{j-1} + \Delta x_j\big] } \chi^{m+\frac{1}{2}}_j
\end{equation} 
for $j=1, \cdots N-1$. Indeed,
\begin{align}\label{tildef} 
\tilde{f}^{\, m+\frac{1}{2}}_j =& \frac{2}{\Delta t_m (\Delta x_{j-1} + \Delta x_j)}\;  \int_{t_m}^{t_{m+1}}  d\beta_j(s) \int_{\R} e_j(x)  dx \nonumber \\
  = & \frac{2}{\Delta t_m (\Delta x_{j-1} + \Delta x_j)}\; \big( \beta_j(t_{m+1}) -\beta_j(t_{m} \big) 
\Big[ \int_{x_{j-\frac{1}{2}}}^{x_j} c_{j-1}(x_j-x) dx + \int_{x_j}^{x_{j+\frac{1}{2}}} c_j (x-x_j) dx \Big] \nonumber \\
=&\frac{\sqrt{3}}{2} \frac{\beta_j(t_{m+1})-\beta_j(t_m)}{\sqrt{\Delta t_m}} 
   \frac{\big[ \sqrt{\Delta x_{j-1}}+\sqrt{\Delta x_j} \big] } { \big[  \Delta x_{j-1} + \Delta x_j\big] } \, \chi^{m+\frac{1}{2}}_j.
\end{align}
A similar computation gives 
\begin{align} \label{f-hatf0}
f^{m+\frac{1}{2}}_0 = \frac{1}{2} \big( u^m_0+u^{m+1}_0\big)  \tilde{f}^{\, m+\frac{1}{2}}_0 , \quad \mbox{\rm where} \quad 
\tilde{f}^{\, m+\frac{1}{2}}_0=\frac{\sqrt{3}}{2} \frac{1}{\sqrt{\Delta t_m \Delta x_0}} \, \chi^{m+\frac{1}{2}}_0,\\
f^{m+\frac{1}{2}}_N = \frac{1}{2} \big( u^m_N+u^{m+1}_N\big)  \tilde{f}^{m+\frac{1}{2}}_N , \quad \mbox{\rm where} \quad 
\tilde{f}^{\, m +\frac{1}{2}}_N=\frac{\sqrt{3}}{2} \frac{1}{\sqrt{\Delta t_m \Delta x_{N}}}\, \chi^{m+\frac{1}{2}}_N.
\label{f-hatfN}
\end{align}
Note that in the definition of  $f^{m+\frac{1}{2}}_j$, the factor $\frac{1}{2} \big( u^m_j+u^{m+1}_j\big)$ is related to the approximation of the Stratonovich integral, 
and that the expression of $\tilde{f}^{\, m +\frac{1}{2}}_j $ differs from that in \cite{DM2002a} and \cite{DM2002b} for two reasons.
 On one hand, we have a non-constant space mesh, and on the other hand, even if the space mesh $\Delta x_j$ is constant (equal to $\Delta x$),
  the extra factor $\frac{\sqrt{3}}{2}$ comes from the fact that we have changed the basis $\{e_j\}_{0\leq j\leq N}$.  For a constant space mesh  $h$, we have
$$
\tilde{f}^{\,m+\frac{1}{2}}_j = \frac{\sqrt{3}}{2} \frac{1}{\sqrt{\Delta t_m} \;\sqrt{ \Delta x}} \chi^{m+\frac{1}{2}}_j,\quad j=0, \cdots, N.
$$
Next, denote  $V_j^m=|u^m_j|^{2\sigma}$ and let $ {f}^{m+\frac{1}{2}}_j$ be defined by \eqref{f-fhat}, \eqref{f-hatf0} and \eqref{f-hatfN}. Note that $\{\tilde f_j^{\, m}\}$ 
then define additive noise. At the half-time step, we let 
$$
u_j^{m+\frac{1}{2}}=\frac{1}{2}(u_j^m+u_j^{m+1}) 
\quad \mbox{\rm and} \quad  
V_j^{m+\frac{1}{2}}= \big| u_j^{m+\frac{1}{2}} \big|^{2\sigma}.
$$ 

To summarize, the discrete version of noise that we consider in this work is defined as follows
\begin{align}\label{E:noise discretization}
g_j^{m+\frac12} =
\begin{cases} 
\epsilon \, f_j^{m+\frac12},  
\quad \mathrm{multiplicative \,\, case,} \\
\epsilon \, \tilde f_j^{\, m+\frac12}, 
\qquad \mathrm{additive \,\, case},
\end{cases}
\end{align}
where $\{f_j\}$'s and $\{\tilde f_j\}$'s are defined in \eqref{f-fhat}, \eqref{f-hatf0} and \eqref{f-hatfN}.

\subsection{Three schemes}
We now consider three schemes: the mass-energy conservative (MEC) scheme (also used in \cite{DM2002a})
\begin{align}		
\label{mass-energy}
i \, \dfrac{u_j^{m+1}-u_j^m}{\Delta t_m}+\mathcal{D}_2 u_j^{m+\frac{1}{2}} + \frac{1}{\sigma +1} \frac{|u^{m+1}_j|^{2(\sigma +1)}- |u^m_j|^{2(\sigma +1)}}
{|u^{m+1}_j|^2-|u^m_j|^2} \; u^{m+\frac{1}{2}}_j  = g_j^{m+\frac{1}{2}},
\end{align}
the Crank-Nicholson (CN) scheme (which is a Taylor expansion of the previous one) 
\begin{align}\label{NS: crank-nicholson}
i \, \dfrac{u_j^{m+1}-u_j^m}{\Delta t_m}+\mathcal{D}_2 u_j^{m+\frac{1}{2}}+V_j^{m+\frac{1}{2}}u^{m+\frac{1}{2}}_j= g_j^{m+\frac{1}{2}},
\end{align}
and the new linearized extrapolation (LE) scheme, which uses the extrapolation of $V_j^{m+\frac12}$ 
\begin{align}\label{NS:relaxation}
i \, \dfrac{u_j^{m+1}-u_j^m}{\Delta t_m}+\mathcal{D}_2 u_j^{m+\frac{1}{2}}+\frac{1}{2}\left(\frac{2\Delta t_{m-1}+
\Delta t_m}{\Delta t_{m-1}} V^{m}_j- \frac{\Delta t_m}{\Delta t_{m-1}} V^{m-1}_j \right) u^{m+\frac{1}{2}}_j= g_j^{m+\frac{1}{2}},
\end{align}
where $g_j^{m+\frac12}$ is defined in \eqref{E:noise discretization}.

To compare them, we note that the schemes \eqref{mass-energy} and \eqref{NS: crank-nicholson} require to solve a nonlinear system at each time step, 
where the fixed point iteration or Newton iteration is involved (see \cite{DM2002a} for details). 
To implement the scheme \eqref{NS:relaxation}, only a linear system needs to be solved  at each time step. 
Numerically, these three schemes generate similar results (for example, the discrete mass is conserved on the order of $10^{-10} - 10^{-12}$;  see Figure \ref{F:deterministic}.
The Crank-Nicholson scheme \eqref{NS: crank-nicholson} usually requires between 2 and 8 iterations at each time step, and thus, is about $3$ times slower than the scheme \eqref{NS:relaxation}, which requires no iteration. In its turn, the mass-energy conservative scheme \eqref{mass-energy} is about 2-3 times slower than the Crank-Nicholson \eqref{NS: crank-nicholson}. Thus, for the computational time, the last linearized extrapolation scheme \eqref{NS:relaxation} is the most convenient.
We remark that the scheme \eqref{NS:relaxation} is a multi-step method. The first time step $u^1$ is obtained by applying either the scheme \eqref{NS: crank-nicholson} or \eqref{mass-energy}, and then we proceed with \eqref{NS:relaxation}.

\subsubsection{Discrete mass and energy}

We define the discrete mass by
\begin{align}\label{D: Dmass}
M_{\mathrm{dis}}[u^m]= \frac{1}{2}\sum_{j=0}^{N} |u^m_j|^2(\Delta x_j+\Delta x_{j-1}).
\end{align}
For $m=0$ it is the first order approximation of the integral defining the mass $M(u_0)$  in \eqref{D: mass}. 

We also define the discrete energy (similar to \cite{DM2002a}), which is adapted to non-uniform mesh as follows
\begin{equation}\label{dis-energy}
H_{\rm dis}[u^m]:= 
\frac{1}{2} \sum_{j=0}^N  \frac{\big|u^m_{j+1}-u^m_j\big|^2}{\Delta x_j}- \frac{1}{2(\sigma+1)} \sum_{j=0}^N
\frac{\big( \Delta x_j + \Delta x_{j-1}\big)}{2} \,  \, |u^m_j|^{2(\sigma +1)}. 
\end{equation}

In order to check our numerical efficiency, we define the discrepancy of discrete mass and energy as
\begin{align}\label{E:error-Dmass}
\mathcal{E}^m_1[M]:=\max_m \left\lbrace M_{\mathrm{dis}}[u^m] \right\rbrace-\min_m \left\lbrace M_{\mathrm{dis}}[u^m] \right\rbrace.
\end{align}
And
\begin{align}\label{E:error-Denergy}
\mathcal{E}^m[H]:=\max_m \left\lbrace H_{\mathrm{dis}}[u^m] \right\rbrace-\min_m \left\lbrace H_{\mathrm{dis}}[u^m] \right\rbrace.
\end{align}

In the deterministic case all three schemes conserve mass. 
In Figure \ref{F:deterministic} we show that the linearized (LE) scheme has the smallest error in discrete mass, since unlike the other two schemes there is no nonlinear system to solve, and thus, only the floating error comes into play. In the MEC and CN schemes the error from solving the nonlinear systems accumulate at each time step. Consequently, the resulting error is accumulate slightly above ($10^{-10}$) (there, we take 
 $|u^{m+1,k+1}-u^{m+1,k}|<10^{-10}$ as the terminal condition for solving the nonlinear system in these two schemes, where $k$  is the index of the fixed point iteration for computing  $u^{m+1} =u^{m+1,\infty}$).

The MEC scheme \eqref{mass-energy} also conserves the discrete energy \eqref{dis-energy}. 
While the other two schemes do not exactly conserve energy, the error of approximation is tolerable as shown on the right of Figure \ref{F:deterministic}. 
(Again, as we set up the tolerance $|u^{m+1,k+1}-u^{m+1,k}|<10^{-10}$ in solving the resulting nonlinear system, the discrete energy error $\mathcal{E}^m[H]$ stays around $10^{-10}$. Note that $\mathcal{E}^m[H]$ is a non-decreasing function in $m$; it increases slowly as time evolves.)

\begin{figure}[ht]
\includegraphics[width=0.45\textwidth]{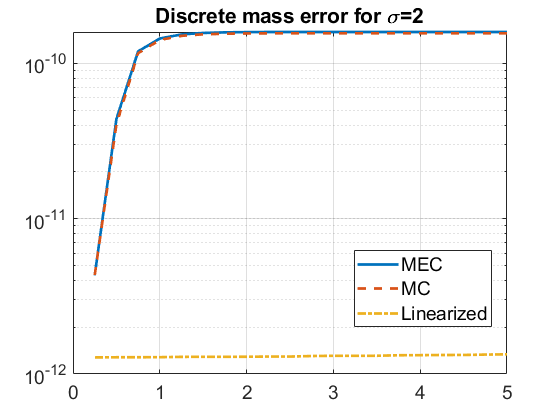}
\includegraphics[width=0.45\textwidth]{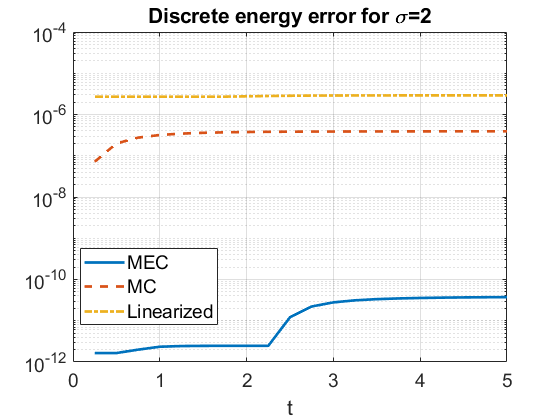}
\caption{Deterministic NLS ($\epsilon=0$), $L^2$-critical case. Comparison of errors in the three schemes: mass-energy conservative (MEC) \eqref{mass-energy}, Crank-Nicholson (CN) \eqref{NS: crank-nicholson} and linearized extrapolation (LE) \eqref{NS:relaxation}. Left: error in mass computation. Right: error in energy computation.}
\label{F:deterministic}
\end{figure}

\subsection{Discrete mass and energy for a multiplicative noise}\label{S:3.2}
We now consider a multiplicative noise, or more precisely its discrete version as defined in 
\eqref{E:noise discretization}. All three schemes conserve mass in this case. 

\begin{lemma}\label{L: numerical scheme}
The numerical schemes \eqref{mass-energy}, \eqref{NS: crank-nicholson} and  \eqref{NS:relaxation} conserve the discrete mass, that is, 
$$
M_{\mathrm{dis}}[u^{m+1}]=M_{\mathrm{dis}}[u^m], \quad m=0,1,\cdots.
$$
\end{lemma}
\begin{proof}
Multiply the equations \eqref{mass-energy}, \eqref{NS: crank-nicholson} or \eqref{NS:relaxation} by $\bar{u}^{m+\frac{1}{2}}_j (\Delta x_j + \Delta x_{j-1})$,
sum among all indexes $j$, and take the imaginary part. 
Note that we impose the Neumann BC on both sides by setting the pseudo-points $u_{-1}=u_{0}$ and $u_{N}=u_{N+1}$. Then, with a straightforward computation, one obtains
$$
M_{\mathrm{dis}}[u^{m+1}]-M_{\mathrm{dis}}[u^m]=0,
$$
which completes the proof.
\end{proof}
By Taylor's expansion, it is easy to see that the schemes \eqref{mass-energy} and \eqref{NS: crank-nicholson} are of the second order accuracy $O((\Delta t_m)^3)$ at each time step $\Delta t_m$. We say the scheme \eqref{NS:relaxation} is almost of the second order accuracy because the residue is on the order 
$O((\Delta t_m)^2 \Delta t_{m-1})$. 
(Later, to make sure that blow-up solutions do not reach the blow-up time, we take the $m$th time step $\Delta t_m=\min \left\{\Delta t_{m-1}, \frac{\Delta t_0}{\|u^m\|_{\infty}^{2\sigma}} \right\}$. Thus,  $\Delta t_m \leq \Delta t_{m-1}$.)  
Therefore, while the schemes \eqref{NS: crank-nicholson} and \eqref{mass-energy} seem to be  slightly more accurate than \eqref{NS:relaxation}, all three give the same order accuracy in their application below.

\subsubsection{Upper bounds on discrete energy}
We now study stability properties of the time evolution of the discrete energy 
\eqref{dis-energy} for the mass-energy conserving (MEC) scheme \eqref{mass-energy}.
Let $\tau^*_{\rm dis}$ denote the existence time of the discrete scheme. For simplicity we take the uniform mesh in space and time, i.e., for each $j$ and $m$, we set $\Delta x= \Delta x_j$ and $\Delta t=\Delta t_m$. 
 In that case the discrete energy is
\[ H_{\rm dis}[u^m]:= \Delta x \Big( \frac{1}{2} \sum_{j=0}^N \Big| \frac{u^m_{j+1}-u^m_j}{\Delta x}\Big|^2 - \frac{1}{2(\sigma+1)} \sum_{j=0}^N
|u^m_j|^{2(\sigma +1)}\Big). 
\] 

\begin{proposition}\label{max_H_multi}
Let $u_0\in H^1$  and $t_M<\tau^*_{\rm dis} $ be a point of the time grid.
Then  for $\Delta x \in (0,1)$
\begin{align}	
{\mathbb E}\big(  H_{\rm dis}[u^M] \big) &\leq \, H_{\rm dis}[u^0] +   \,\frac{\epsilon\, \sqrt{3}}{2\, \sqrt{2}}   \frac{\sqrt{ \ln(2\, L_c) + |\ln(\Delta x)| }}{\sqrt{\Delta x}} 
\, \frac{1}{(\Delta t)^{\frac{3}{2}}} t_M,  
\label{E_H_mult}
\\
{\mathbb E}\Big( \max_{0\leq m\leq M} H_{\rm dis}[u^m] \Big) &\leq \, H_{\rm dis}[u^0] +   \frac{\epsilon \, \sqrt{3}}{\sqrt{2}} \,
 \frac{\sqrt{ \ln(2\, L_c) + |\ln(\Delta x)| }}{\sqrt{\Delta x}} \, 
\frac{1}{(\Delta t)^{\frac{3}{2}}} t_M. \label{E_max_H_mult}
\end{align}
\end{proposition}

\begin{proof}
Multiplying equation \eqref{mass-energy} by $- \Delta x\, (\bar{u}_j^{m+1}-\bar{u}^m_j)$, adding for $m=0, ..., M-1$ and $j=0, ..., N$, and using the conservation of the
discrete energy in the deterministic case, we deduce that for some real-valued random variable $R(M,N)$, which changes from one line to the next,
\begin{align}\label{upp_H_multi}
H_{\rm dis}[u^M]=&\, H_{\rm dis}[u^0] +i R(M,N) + \epsilon \, \Delta x \,   \sum_{m=0}^{M-1} \sum_{j=0}^N  (\bar{u}^{m+1}_j-\bar{u}^m_j) 
 \, \frac{1}{2}\big( u^{m+1}_j+u^m_j\big)
 \tilde{f}^{\,m+\frac{1}{2}}_j  \nonumber \\
 =& H_{\rm dis }[u^0] + i R(M,N)   + \frac{ \epsilon\, \Delta x}{2}   \sum_{m=0}^{M-1} \sum_{j=0}^N   \, \big(  |u^{m+1}_j|^2 - |u^{m}_j|^2\big) \, 
   \tilde{f}^{\,m+\frac{1}{2}}_j, \\
 = &\, H_{\rm dis }[u^0] + i R(M,N)   - \frac{\epsilon}{2}  \,  \frac{1}{\Delta t} \, \int_0^{t_M}\!  \int_{\R} 
  |{U}(s,x)|^2 W_N(ds,dx) 
+  \frac{ \epsilon\, \Delta x}{2}   \sum_{m=0}^{M-1} \sum_{j=0}^N   \,  |u^{m+1}_j|^2 \,   \tilde{f}^{\,m+\frac{1}{2}}_j,  \label{upp_H_multi_1}
  \end{align} 
where $U(s,x)$ is the step process defined by $U(s,x)=u^m_j$ on the rectangle $[t_m,t_{m+1})\times [x_{j-\frac{1}{2}} , x_{j+\frac{1}{2}})$. 
Since the discrete mass is preserved by the scheme (Lemma \ref{L: numerical scheme}), we have 
$$ 
\frac{ \epsilon\, \Delta x}{2}   \sum_{m=0}^{M-1} \sum_{j=0}^N   \,  |u^{m+1}_j|^2 \,   \tilde{f}^{\, m+\frac{1}{2}}_j \leq
\frac{\epsilon}{2} \sum_{m=0}^{M-1} \max_{0\leq j\leq N} | \tilde{f}^{\, m+\frac{1}{2}}_j| \sum_{j=0}^N \Delta x \, |u^{m+1}_j|^2 =
\frac{\epsilon\, M_{\rm dis}[u^0]}{2} \sum_{m=0}^{M-1} \max_{0\leq j\leq N} | \tilde{f}^{\, m+\frac{1}{2}}_j|. 
$$ 
Using the definition of $\tilde{f}^{\, m+\frac{1}{2}}_j$ in \eqref{f-fhat}, \eqref{f-hatf0} and \eqref{f-hatfN}, we deduce  
$$ 
E\Big( \max_{0\leq j\leq N} | \tilde{f}^{\, m+\frac{1}{2}}_j| \Big) = \frac{\sqrt{3}}{2\, \sqrt{\Delta t} \sqrt{\Delta x}} E\Big( \max_{0\leq j\leq N} | \chi^{m+\frac{1}{2}}_j| \Big),
$$
where the random variables $\chi^{m+\frac{1}{2}}_j$ are independent standard Gaussians. 

Using Pisier's lemma (see e.g. \cite[Lemma 10.1]{Lif2012}),
one observes that if $\{G_k\}_{ k=1, ...,n}$ are independent standard Gaussians and $M_n=\max_{1\leq k\leq n} |G_k|$,  we have  for $n\geq 2$
\begin{equation}\label{E_max_G}
\mathbb{E}(M_n)  \leq  \sqrt{2 \, \ln(2\, n)} . 
\end{equation}
 We enclose the proof below for the sake of completeness. For any $\lambda >0$, using the Jensen inequality and the fact that $x\mapsto 
 e^{\lambda x}$ is increasing,   we deduce
 \begin{align*}
\exp\Big( \lambda \,   {\mathbb E}\Big[ \max_{1\leq k\leq n} |G_k|\Big]\Big) \leq &\; {\mathbb E} \Big( \exp\Big[ \lambda \max_{1\leq k\leq n}
\big|G_k\big| \Big] \Big) \leq {\mathbb E} \Big( \max_{1\leq k\leq n} \exp\big( \lambda |G_k|\big) \Big) \\
\leq & \; \sum_{k=1}^n {\mathbb E} \Big(e^{\lambda \big| G_k\big|}\Big) \leq n \, 2\,   e^{\frac{\lambda^2 }{2}}.
 \end{align*} 
 Taking logarithms, we obtain 
 \[ {\mathbb E}\Big( \max_{1\leq k\leq n} |G_k | \Big) \leq  \frac{1}{\lambda} \ln\Big(2\,  n\,  e^{\frac{\lambda^2}{2}}\Big)
 = \frac{\ln(2\, n)}{\lambda}  + \frac{\lambda}{2}, 
 \]
for every $\lambda >0$.  Choosing $\lambda = \sqrt{2 \ln(2n)}$  concludes the proof of \eqref{E_max_G}. 

Keeping the real part of  \eqref{upp_H_multi_1}, we obtain 
\begin{align*}
{\mathbb E} \big( H_{\rm dis}[u^M]\big) \leq &\; H_{\rm dis}[u^0] +  \frac{\epsilon}{2} M_{\rm dis}[u^0]   \, M\, \frac{\sqrt{3}}{2}\, 
\sqrt{2 \ln\big[ 2\, (N+1)\big] } 
  \frac{1}{\sqrt{\Delta t} \sqrt{\Delta x}} \\
\leq & \; 
H_{\rm dis}[u^0] + \frac{\epsilon\, \sqrt{3} }{2\, \sqrt{2}}   \frac{1}{\sqrt{\Delta x}} \, \sqrt{  \ln\Big( \frac{2\, L_c}{\Delta x}\Big) }\,  \frac{1}{(\Delta t)^{\frac{3}{2}}} t_M.    \
\end{align*}
This completes the proof of \eqref{E_H_mult}.

To prove \eqref{E_max_H_mult},  keeping the real part of  \eqref{upp_H_multi} and estimating from above $| u^{m+1}_j|^2 - |u^m_j|^2$ by
$| u^{m+1}_j|^2 + |u^m_j|^2$, we get 
$$ 
\max_{0\leq m\leq M} H_{\rm dis}[u^M]= H_{\rm dis}[u^0]  +  \frac{ \epsilon\, \Delta x}{2}   \sum_{m=0}^{M-1} \sum_{j=0}^N   \,  \big( |u^{m+1}_j|^2 + |u^m_j|^2\big)
 \,   |\tilde{f}^{\,m+\frac{1}{2}}_j|,  
$$
and the previous argument concludes the proof. 
\end{proof}

\begin{remark}
Note that  in  \eqref{upp_H_multi} and \eqref{upp_H_multi_1}, the upper bound depends
linearly on $\epsilon$, and for small $\epsilon<< 1$ so does the {\it leading term} of the theoretical estimate  \eqref{E:supH-1}. There is also a very small
dependence on $L_c$, and a more important one on $\Delta x$ and $\Delta t$. We remark that these are just the upper bounds, 
and to get a better idea about the growth and dependence of the energy on the various parameters, we investigate that numerically. 
\end{remark}

\subsubsection{Numerical tracking of discrete mass and energy}

Our analytical results above provide mass conservation and upper bounds on the expected values of energy. We would like to check numerically behavior of these quantities. 
We start with testing the accuracy and efficiency of our schemes, for that 
we consider initial data  $u_0 = A\,Q$, where $A>0$ and $Q$ is the ground state \eqref{E:Q}.  

\begin{figure}[ht]
\includegraphics[width=0.45\textwidth]{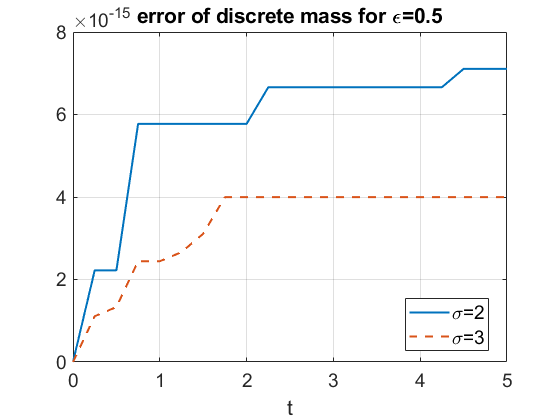}
\caption{Multiplicative noise. The error of the discrete mass computation $\mathcal{E}_1^m$ from \eqref{E:error-Dmass}, $\epsilon=0.5$, in both $L^2$-critical and supercritical cases for one trajectory.}
\label{mass 5p}
\end{figure}

For the first test we take $u_0=0.95 Q$ and $\epsilon=0.5$ in both $L^2$-critical ($\sigma=2$) and $L^2$-supercritical ($\sigma=3$) cases. 
The difference $\mathcal{E}^n_1$ in both cases is shown in Figure \ref{mass 5p}.
Observe that the error is on the order of $10^{-15}$, which is almost at the machine precision ($10^{-16}$).

Since not all of our three schemes conserve the discrete energy exactly (in the deterministic case), we study influence of the multiplicative noise onto the discrete energy 
\eqref{dis-energy}. In Figure \ref{F:E-comparison-multi} we show that in the $L^2$-critical case 
and $\epsilon=0.5$, all three schemes produce the same result for the initial data $u_0=0.9Q$, where the energy is growing and then starts leveling off around the 
time $t=15$. On the right of the same figure  we zoom  on the time interval  $[0,5]$ to see better the difference between the schemes, and we note that the linearized extrapolation (LE) scheme produces slightly lower values of the energy, even if  the overall behavior is the same. 
In our further investigations we usually use the MEC scheme if we need to track the mass and energy, and when we investigate the more 
global features such as blow-up profiles or run a lot of simulations, then we utilize the LE scheme. 

\begin{figure}[ht]
\includegraphics[width=0.45\textwidth]{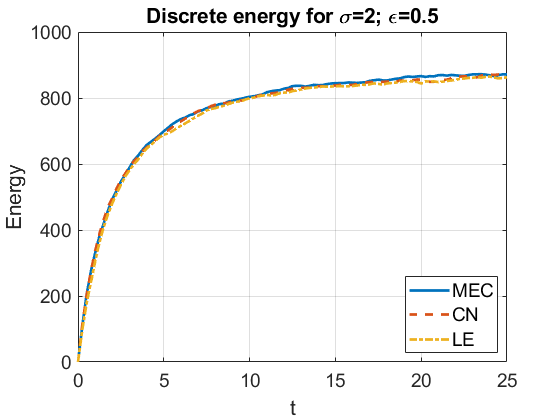}
\includegraphics[width=0.45\textwidth]{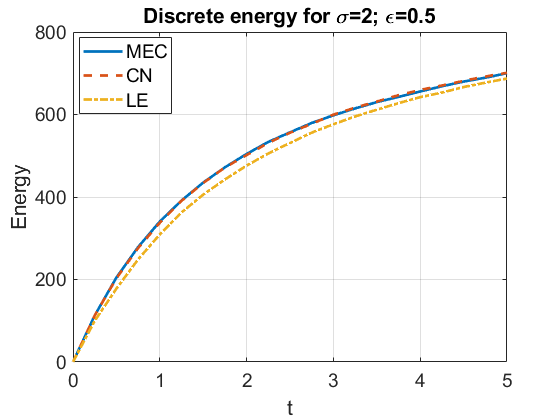}
\caption{Multiplicative noise, $\epsilon=0.5$, $L^2$-critical case. Expected energy (averaged over 100 runs) using different schemes: mass-energy conservative 
(MEC) \eqref{mass-energy}, Crank-Nicholson (CN) \eqref{NS: crank-nicholson} and linearized extrapolation (LE) \eqref{NS:relaxation}. 
Left: time $0<t<25$. Right: zoom-in for time $0<t<5$: note only a small difference with the LE scheme.}
\label{F:E-comparison-multi}
\end{figure}

We next study the growth of energy in time and the dependence on various parameters. 
In Figures \ref{F:energy-growth-1}, \ref{F:energy-growth-dx}, \ref{F:energy-growth-dt} we show the time dependence of solutions with initial data of type $u_0 = A\, Q$.

\begin{figure}[ht]
\includegraphics[width=0.45\textwidth]{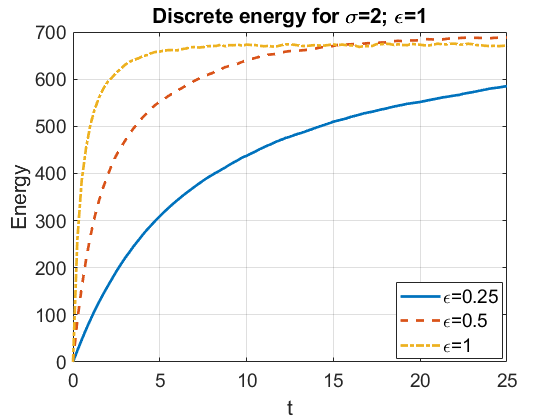}
\includegraphics[width=0.45\textwidth]{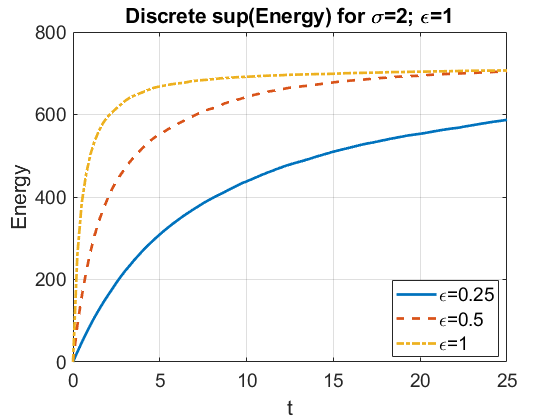}
\includegraphics[width=0.45\textwidth]{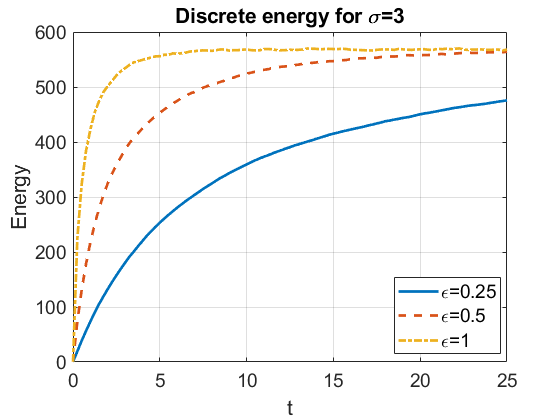}
\includegraphics[width=0.45\textwidth]{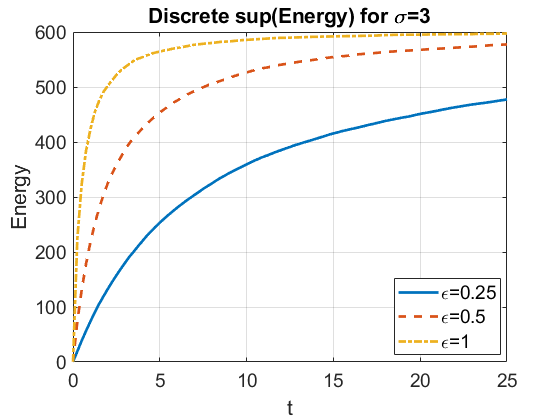}
\caption{Multiplicative noise in the $L^2$-critical case, $\sigma=2$ (top) 
and $L^2$-supercritical case, $\sigma=3$ (bottom); $u_0=0.8Q$, $\Delta x=0.05$, $\Delta t=0.005$, $L_c=20$. Time dependence of $\mathbb E (H(u(t)))$ (left) vs.
 $ \mathbb{E}( \sup_{s\leq t} H(u(s))$ (right) for various  $\epsilon$.}
\label{F:energy-growth-1}
\end{figure}

In Figure \ref{F:energy-growth-1} we track the growth of the expected values of the instantaneous energy (on the left subplots) and of the supremum of energy (on the right subplots). 
To approximate the expected value, we average over 100 runs.
Our simulations show that both start growing linearly at first (see zoom-in Figure \ref{F:zoom-in}), then start slowing down until they peak and level off to some possibly  maximum value. As expected the values of the maximal energy up to some specific time are larger. We observe that the stronger the noise is (i.e., the larger the coefficient $\epsilon$), the shorter it takes for the expected energy to start leveling off. A similar behavior is seen in Figure \ref{F:energy-growth-2} for the gaussian initial data $u_0=A \, e^{-x^2}$ and supergaussian data $u_0=A\, e^{-x^4}$ in both critical and supercritical cases. From now on we only show expectations of instantaneous energy in our figures as plots for the maximal energy  are very similar.   

\begin{figure}[ht]
\includegraphics[width=0.23\textwidth]{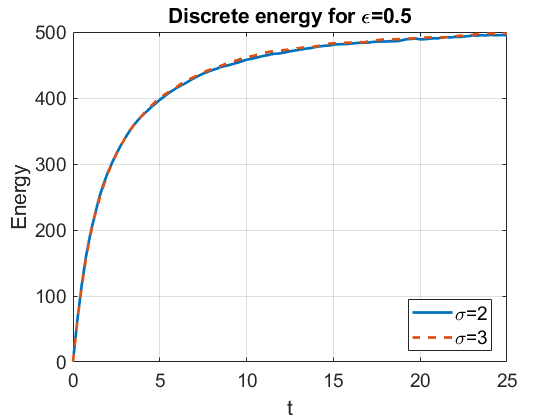}
\includegraphics[width=0.23\textwidth]{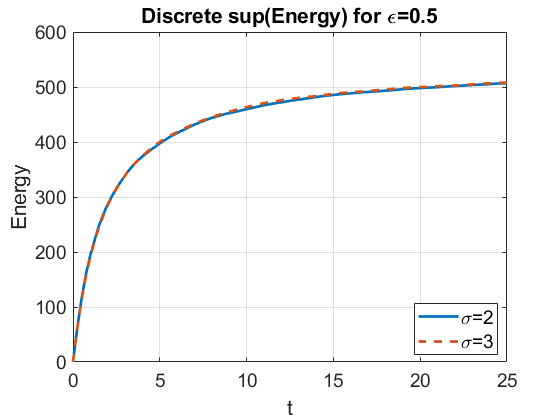}
\includegraphics[width=0.23\textwidth]{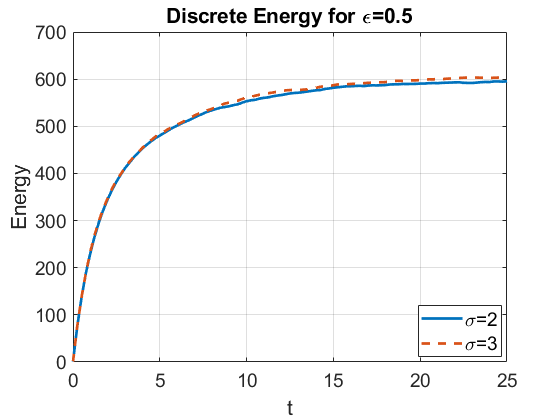}
\includegraphics[width=0.23\textwidth]{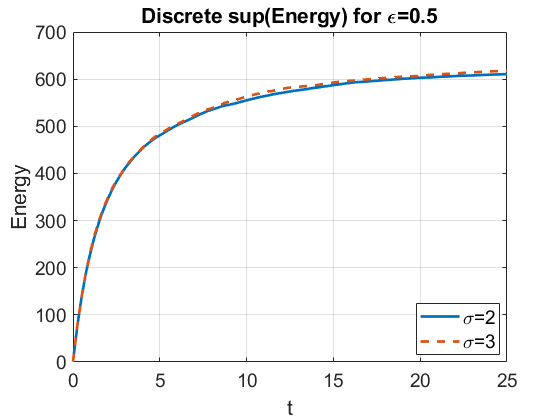}
\caption{Multiplicative noise in both $L^2$-critical and supercritical cases: gaussian (left two) $u_0=e^{-x^2}$ and supergaussian (right two) $u_0=e^{-x^4}$; $\Delta x=0.05$, $\Delta t=0.005$, $L_c=20$. The time dependence of $\mathbb E (H(u(t)))$ (left) vs. $ \mathbb{E}( \sup_{s\leq t} H(u(s))$ (right).}
\label{F:energy-growth-2}
\end{figure}
We next investigate the dependence of the discrete energy \eqref{dis-energy} on computational parameters such as the length of the interval $L_c$, 
the spatial step size $\Delta x$ and the time step $\Delta t$. 
The results are shown in Figure \ref{F:energy-growth-dx} for the expected energy values $\mathbb E(H(u(t)))$
with varying sizes of $\Delta x$ and $L_c$; in Figure \ref{F:energy-growth-dt} 
the dependence on $\Delta t$ is displayed. 

\begin{figure}[ht]
\includegraphics[width=0.45\textwidth]{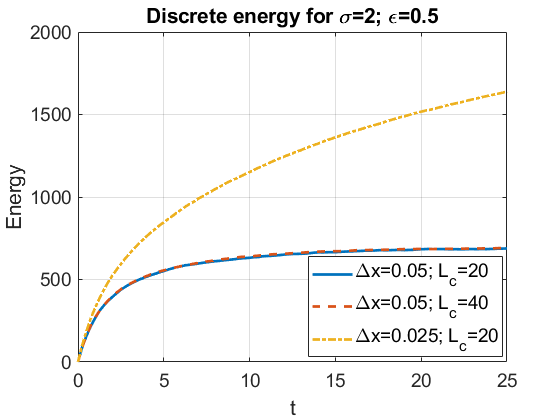}
\includegraphics[width=0.45\textwidth]{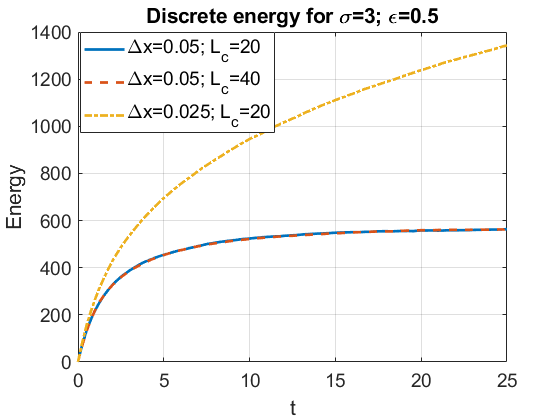}
\caption{Multiplicative noise, $u_0=0.8Q$, $\epsilon=0.5$. The growth of  expected energy depends on $\Delta x$
 but not on  $L_c$ in both $L^2$-critical (left) and supercritical (right) cases.}
\label{F:energy-growth-dx}
\end{figure}

\begin{figure}[ht]
\includegraphics[width=0.45\textwidth]{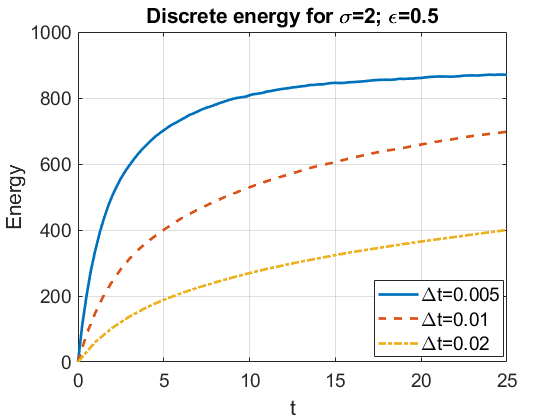}
\includegraphics[width=0.45\textwidth]{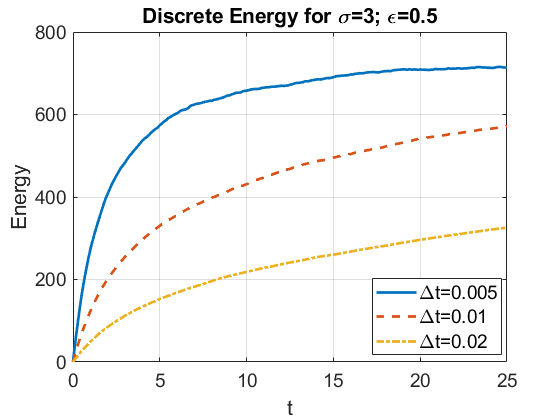}
\caption{Multiplicative noise, $u_0=0.9Q$, $\epsilon=0.5$, $L_c=20$, $\Delta x = 0.05$, $\Delta t=0.005$. The growth of the expected energy for different $\Delta t$ in both the $L^2$-critical and supercritical cases. 
}
\label{F:energy-growth-dt}
\end{figure}

\begin{figure}[ht]
\includegraphics[width=0.45\textwidth]{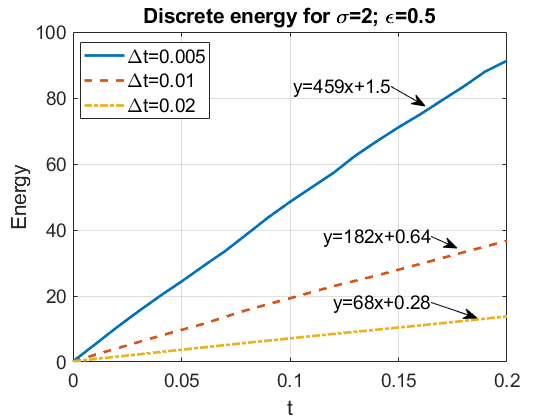}
\includegraphics[width=0.45\textwidth]{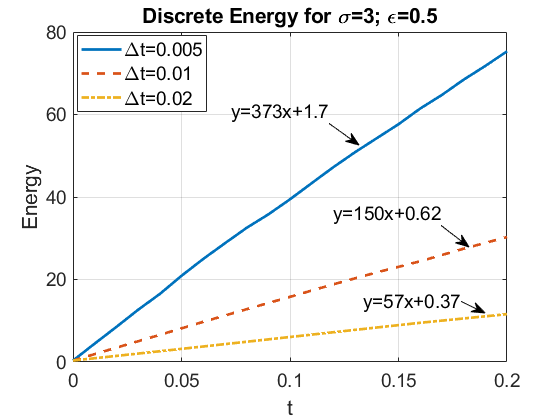}
\caption{Multiplicative noise. Zoom-in for small times (to track linear dependence): $u_0=0.9Q$, $\Delta x=0.05$, $L_c=20$, $\epsilon=0.5$. The time dependence of $\mathbb E (H(u(t)))$ 
for different values of $\Delta t$ in both $L^2$-critical and supercritical cases.}
\label{F:zoom-in}
\end{figure}

We remark that in both critical and supercritical cases, the computed values of expected 
energies (instantaneous and sup) are insensitive to the length of the computational domain $L_c$.  However, there is a dependence on the mesh size $\Delta x$: 
the smaller step size results in a larger value of energy; there is also a dependence on the time step $\Delta t$.

\subsection{Discrete mass and energy for an additive noise}\label{S:3.3}
Our next endeavor is to study the additive stochastic perturbation $f(u)=W(dt,dx)$, or its discretized version in \eqref{E:noise discretization}.
 As in the multiplicative case, we replace the space-time white noise $W$ by its approximation $W_N$ defined in \eqref{WN} in terms of the functions
  $\{e_j \}_{0 \leq j \leq {N}}$ described in \eqref{def_ej}. Then in our numerical schemes \eqref{mass-energy}, \eqref{NS: crank-nicholson} and \eqref{NS:relaxation} 
  the right-hand side is $\Big\{\tilde{f}^{\,m+\frac12}_j \Big\}$ defined in \eqref{tildef} for $j=1, ..., N-1$,  in \eqref{f-hatf0} for $j=0$, and in \eqref{f-hatfN} for $j=N$.

We show that for the schemes \eqref{mass-energy}, \eqref{NS: crank-nicholson} and \eqref{NS:relaxation} 
the time evolution of the expected value of the discrete mass on the time interval $[0,T]$  is estimated from above by an affine function $a+bt$. 
We prove that the slope $b$ is a linear function
of the length $L_c$ of the discretization interval $[-L_c,L_c]$. Therefore, our upper bounds on the discrete mass and energy  depend linearly on the total length 
$L_c$ ; they are inversely proportional to the constant time and space mesh sizes $\Delta t$ and $\Delta x$. 
We do not claim that our upper bounds are sharp; this is the first attempt to  upper estimate the discrete quantities. 

\subsubsection{Upper bounds on discrete mass and energy with additive noise}
Recall that the discrete mass of $u^m$ is defined by $M_{\rm dis}[u^m]=\Delta x \sum_{j=0}^N |u^m_j|^2$. 
Let $\tau^*_{\rm dis}$ be the maximal existence time of a discrete scheme. 
\begin{proposition}\label{dis_mass_add}
Let $u^m_j$ be the solution to the scheme \eqref{mass-energy},  
\eqref{NS: crank-nicholson} or \eqref{NS:relaxation} with $\tilde{f}^{m+\frac{1}{2}}_j$ instead of $f^{m+\frac{1}{2}}_j$ for a constant
time mesh $\Delta t$ and space mesh $\Delta x$. 

Then given $T\in (0, \tau^*_{\rm dis})$ and  any element $t_M\leq T$ of the time grid, 
we have for any $\alpha >0$ 
\begin{align}\label{affine_mass_dis}
{\mathbb E} \big(M_{\mathrm{dis}}[u^M]\big) & \leq   (1+\alpha)  M_{\rm dis}[u^0] + \frac{3\,T\, (1+\alpha)}{4\, \ln(1+\alpha)} \;  \epsilon^2\, \frac{L_c}{\Delta x} \, \frac{t_M}{\Delta t},
\quad \mbox{if} \quad \Delta t\leq \frac{T}{1+\alpha}, \\ 
{\mathbb E}\Big( \max_{0\leq m\leq M} M_{\mathrm{dis}}[u^m]\Big) &\leq \Big(1+\alpha + \frac{\alpha}{2T}\Big)  M_{\mathrm{dis}} [u^0] 
+  \frac{3 \, T\, (1+\alpha)^2}{2\, \alpha}  \; \epsilon^2 \, \frac{L_c}{\Delta x} \, \frac{ t_M}{\Delta t} .
\label{affine_max_mass_dis}
\end{align} 
\end{proposition}  
\begin{proof}
Recall that $u^{m+\frac{1}{2}}_j= \frac{1}{2} \big( u^m_j + u^{m+1}_j\big)$. Multiply the equation \eqref{mass-energy}  by 
$- 2 i  \Delta t \Delta x \, \bar{u}^{\,m+\frac{1}{2}}_j$, sum on $j$ from $j=0$ to $N$ and then sum on $m$ from $m=0$ to $M-1$. 
Then there exists a real-valued random variable $R(M,N)$ (changing from one line to the next)  such that
\begin{align}\label{mass-diff}
\Delta x \sum_{j=0}^{N} |u^{M}_j|^2 & - \Delta x \sum_{j=0}^{N} |u^{0}_j|^2  =  i R(M,N) \; -2\, \, \epsilon  i \sum_{m=0}^{M-1} \sum_{j=0}^N
\Delta t \, \Delta x \, \frac{\bar{u}^{\, m}_j + \bar{u}^{\, m+1}_j}{2}\, \tilde{f}^{\,m+\frac{1}{2}}_j  \nonumber \\
&= i R(M,N) \, - \, \epsilon \int_0^{t_M} \!\! \int_{\R} \mbox {\rm Im}\, \big( {U}(s,x) \big)\, W_N(ds,dx)  - \epsilon \, \sum_{m=0}^{M-1} \sum_{j=0}^N 
\Delta t \, \Delta x \, \mbox {\rm Im}\, ({u}^{m+1}_j) \, \tilde{f}^{\,m+\frac{1}{2}}_j, \
\end{align}
where $U$ is  the step process defined by ${U}(s,x)=  u^m_j$ on the rectangle $[t_m,t_{m+1})\times [x_{j-\frac{1}{2}},x_{j+\frac{1}{2}})$.
The Cauchy-Schwarz inequality applied to $\sum_m\, \sum_j$, the definition of $\tilde{f}^{\, m+\frac{1}{2}}_j$ in \eqref{f-fhat}, \eqref{f-hatf0} and \eqref{f-hatfN},
and Young's inequality imply that for $\delta >0$ we have
\begin{align}\label{upper_QV1}
\Big|\,  \epsilon \sum_{m=0}^{M-1} \sum_{j=0}^N  \Delta t \, \Delta x \, \mbox {\rm Im}\, ({u}^{m+1}_j) \, \tilde{f}^{\,m+\frac{1}{2}}_j\Big| \leq &\, \epsilon 
\Big\{ \sum_{m=1}^M \sum_{j=0}^N \Delta t\, \Delta x\, |u^m_j|^2 \Big\}^{\frac{1}{2}}
 \Big\{  \sum_{m=0}^{M-1} \sum_{j=0}^N \Delta t\, \Delta x\, |\tilde{f}^{\,m+\frac{1}{2}}_j|^2\Big\}^{\frac{1}{2}} \nonumber \\
 \leq &\, \delta \sum_{m=1}^M \Delta t \, M_{\rm dis}[u^m] + \frac{3\, \epsilon^2}{16\, \delta }  \sum_{m=0}^{M-1}  \sum_{j=0}^N  \frac{3}{4} | \chi^{\,m+\frac{1}{2}}_j|^2,
\end{align} 
where the random variables $\chi^{m+\frac{1}{2}}_j$ are (as before) independent standard Gaussian random variables. 

Keeping the real part in \eqref{mass-diff}, then plugging the above estimate into the \eqref{mass-diff} and taking expected values 
(note that the process $U$ is adapted), we deduce
$$ 
{\mathbb E}\big( M_{\rm dis}[u^M]\big) \leq  M_{\rm dis}[u^0] + \frac{3\, \epsilon^2}{16\, \delta } M (N+1) 
+ \delta \sum_{m=1}^M \Delta t {\mathbb E}\big( M_{\rm dis}[u^m]\big).
$$ 
Given $\beta \in (0,1)$, we suppose that $ \delta\, \Delta t \leq \beta$. Then the discrete version of the Gronwall lemma (see e.g. \cite[Lemma 1]{Ho2009})
implies 
$$
{\mathbb E}\big( M_{\rm dis}[u^M]\big) \leq  \frac{1}{1-\beta} \Big[ M_{\rm dis}[u^0] + \frac{3\, \epsilon^2}{16\, \delta } M (N+1) \Big] e^{\delta (M-1) \Delta t}.
$$
Fix $\alpha \in (0,1)$ and choose $\beta\in (0,1)$ such that $\frac{1}{1-\beta}= \sqrt{1+\alpha}$, and  choose $\delta>0$ such that $e^{\delta T} = \sqrt{1+\alpha}$.
Then $\delta = \frac{\ln(1+\alpha)}{2T}\in \big(  \frac{\alpha}{4T}, \frac{\alpha}{2T}\big)$, and  $\Delta t \leq \frac{T}{1+\alpha} \leq \frac{2T }{(\sqrt{1+\alpha}+1)
\sqrt{1+\alpha}} = \frac{2T}{\alpha} \beta$ implies 
$\delta\, \Delta t \leq \beta$. 
Furthermore, $M(N+1) \leq \frac{t_M}{\Delta t}\, \frac{2\, L_c}{\Delta x} $, and we deduce  \eqref{affine_mass_dis}.
\smallskip

We next prove \eqref{affine_max_mass_dis}. A similar computation, based on the first upper estimate in \eqref{mass-diff} and on \eqref{upper_QV1}, 
 proves  that for $\delta >0$ we have 
\begin{align}\label{mass-diff-2}
M_{\rm dis}[u^M] = &\, M_{\rm dis}[u^0]   + \epsilon \, \Big| \sum_{m=0}^{M-1} \sum_{j=0}^N 
\Delta t \, \Delta x \, \mbox {\rm Im}\, ({u}^{m}_j) \, \tilde{f}^{\,m+\frac{1}{2}}_j \Big| +  \epsilon \, \sum_{m=0}^{M-1} \sum_{j=0}^N 
\Delta t \, \Delta x \, \mbox {\rm Im}\, ({u}^{m+1}_j) \, \tilde{f}^{\,m+\frac{1}{2}}_j, \nonumber \\
\leq &\, M_{\rm dis}[u^0] + \delta  \Delta t \, M_{\rm dis}[u^0]  + 2 \delta \sum_{m=1}^M \Delta t \,  M_{\rm dis}[u^m] 
+ 2 \frac{\epsilon^2}{4\, \delta }  \sum_{m=0}^{M-1}  \sum_{j=0}^N \frac{3}{4}\,  |\chi^{\,m+\frac{1}{2}}_j|^2,
\end{align} 
where the random variables $\chi^{m+\frac{1}{2}}_j$ (as before) are independent standard Gaussian random variables. 
Taking expected values, we deduce for any $\delta >0$ 
$$ 
{\mathbb E}\Big(\max_{1\leq m\leq M}  M_{\rm dis}[u^m]\Big) \leq (1+\delta \Delta t ) M_{\rm dis}[u^0] +
2\, \delta \, t_M\, {\mathbb E}\Big(\max_{1\leq m\leq M}  M_{\rm dis}[u^m]\Big) 
+   \frac{3\, \epsilon^2}{8\, \delta } \,   M\, (N+1). 
$$ 
Given $\beta >0$, choose $\delta >0$ such that $2\, \delta T = \beta$; this yields
$${\mathbb E}\Big(\max_{1\leq m\leq M}  M_{\rm dis}[u^m]\Big) \leq \frac{1}{1-\beta} \Big[ \big( 1+\delta \Delta t\big) M_{\rm dis}[u^0] 
+ \frac{3\, \epsilon^2}{8\, \delta } \,   M\, (N+1)\Big].
$$
Given $\alpha >0$, choose $\beta\in (0,1)$ such that $\frac{1}{1-\beta}=  1+\alpha$; then $\delta = \frac{\alpha}{2T(1+\alpha)}$.  
This concludes the proof of \eqref{affine_max_mass_dis} for the mass-energy conservative  scheme.
\smallskip

A similar argument is applied to the schemes  \eqref{NS: crank-nicholson} and \eqref{NS:relaxation} (with the additive right-hand side); the only difference  is in the real-valued random variable $R(M,N)$, which varies from one scheme to the next, but is not present in the final estimate. 
\end{proof}

\begin{remark} Note that the estimates \eqref{affine_mass_dis} and \eqref{affine_max_mass_dis} of the instantaneous and maximal mass are worse than the discrete analog of \eqref{energy_add} by a factor of $\frac{1}{\Delta t}$. One might try to solve this problem in the proof, changing $ 2  \bar{u}^{\, m+\frac{1}{2}}_j \, \tilde{f}^{\,m+\frac{1}{2}}_j$ into 
$ \big( \bar{u}^{m+1}_j- \bar{u}^m_j\big)  \, \tilde{f}^{\,m+\frac{1}{2}}_j + 2\, \bar{u}^m_j   \, \tilde{f}^{\,m+\frac{1}{2}}_j $, and using again the scheme to deal
with the first term.  This would introduce an extra $\Delta t$ factor. However, if the product of the two stochastic Gaussian variables would give a discrere analog of
the inequality \eqref{energy_add}, the deterministic part of the scheme would still create  terms involving $\bar{u}^{m+1}_j  \, \tilde{f}^{\,m+\frac{1}{2}}_j $. 
The corresponding non-linear ``potential" term would yield the mass to be raised to a large power to enable the use of the discrete Gronwall or Young lemma.
\end{remark}

We next study stability properties of the time evolution of the discrete energy defined by \eqref{dis-energy} for the mass-energy conserving (MEC) scheme \eqref{mass-energy} in the additive case. 
\begin{proposition} 
Let $u^n_j$ be the solution to the scheme \eqref{mass-energy} with $\tilde{f}^{\, m+\frac{1}{2}}_j$ in \eqref{E:noise discretization} 
for a  constant time mesh $\Delta t$ and  space mesh $\Delta x$. 
Then given $T\in (0, \tau^*_{\rm dis})$ and  any element $t_M\leq T$ of the time grid, 
we have for any $\alpha >0$ 
\begin{align}\label{E-instant-Hdis_add}
{\mathbb E} \big(H_{\mathrm{dis}}[u^M]\big) & \leq   (1+\alpha)  H_{\rm dis}[u^0] + \frac{3\,T\, (1+\alpha)}{4\, \ln(1+\alpha)} \;  \epsilon^2\, \frac{L_c}{\Delta x} \, \frac{t_M}{(\Delta t)^2},
\quad \mbox{if} \quad \Delta t\leq \frac{T}{1+\alpha}, \\ 
{\mathbb E}\Big( \max_{0\leq m\leq M} H_{\mathrm{dis}}[u^m]\Big) &\leq \Big(1+\alpha + \frac{\alpha}{2T}\Big)  H_{\mathrm{dis}} [u^0] 
+  \frac{3 \, T\, (1+\alpha)^2}{2\, \alpha}  \; \epsilon^2 \, \frac{L_c}{\Delta x} \, \frac{ t_M}{(\Delta t)^2} .
\label{EMaxHdis_add}
\end{align} 
\end{proposition}  
\begin{proof}
Multiplying equation \eqref{mass-energy} by
$- (\bar{u}^{m+1}_j - \bar{u}^m_j) \Delta x $, adding for $j=0, ..., N$ and $m=0, ..., M-1$, 
and using the fact that in the deterministic case ($\epsilon =0$) the scheme \eqref{mass-energy} preserves the discrete energy,  we deduce the existence of a real-valued random variable $R(M,N)$ (changing from line to line) such that
\begin{align*}
H_{\rm dis}[u^M] = & H_{\rm dis }[u^0] + i R(M,N) -  \epsilon  \Delta x \sum_{m=0}^{M-1} \sum_{j=0}^N  (\bar{u}^{m+1}_j-\bar{u}^m_j) 
 \tilde{f}^{\, m+\frac{1}{2}}_j \\
 =  &H_{\rm dis }[u^0] + i R(M,N) + \,  \epsilon \, \frac{\Delta x}{\Delta t}  \int_0^{t_M} \mbox{\rm Re}\, (u^m_j)\, W_N(ds,dx) 
 -   \epsilon \, \Delta x\, \sum_{m=0}^{M-1} \sum_{j=0}^N \mbox{\rm Re}\, u^{m+1}_j\,  \tilde{f}^{\, m+\frac{1}{2}}_j .
\end{align*} 
Notice that the last term in the above identity is similar to the last one in \eqref{mass-diff}, except that the factor $\Delta t$ is missing. 
Thus, the  arguments used to prove Proposition \ref{dis_mass_add} conclude the proof.
\end{proof}

\subsubsection{Numerical tracking of discrete mass and energy, additive noise}
As in the multiplicative case, we start with testing the accuracy of our three numerical schemes \eqref{mass-energy}, \eqref{NS: crank-nicholson} and 
\eqref{NS:relaxation} with the additive forcing \eqref{E:noise discretization} on the right hand side and  using the initial data $u_0=A\,Q$.
 In Figure \ref{F:E-comparison-add} we show the comparison of three schemes for the initial condition $u_0=0.9Q$ with the strength of the noise $\epsilon = 0.05$ in the $L^2$-critical case. We see that for both discrete mass and energy the schemes behave similarly with very little variation from one to another. 
 
\begin{figure}[ht]
\includegraphics[width=0.45\textwidth]{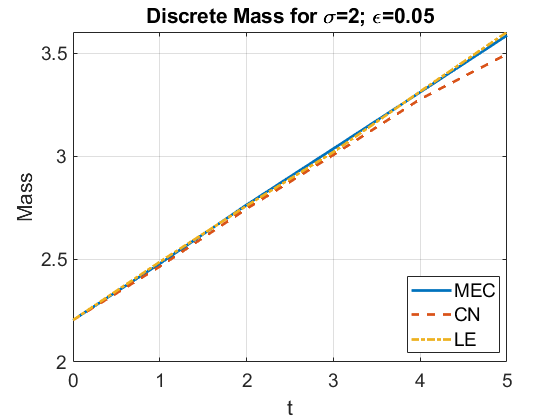}
\includegraphics[width=0.45\textwidth]{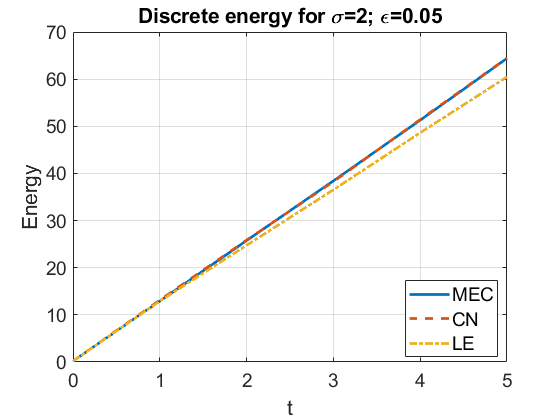}
\caption{Additive noise, $\epsilon=0.05$, $L^2$-critical case. Time evolution of discrete mass (left) and energy (right) via different schemes:
 mass-energy conservative (MEC) \eqref{mass-energy}, Crank-Nicholson (CN) \eqref{NS: crank-nicholson} and linearized extrapolation (LE) \eqref{NS:relaxation}.}
\label{F:E-comparison-add}
\end{figure}

We first investigate dependence of mass and energy on the strength of the noise $\epsilon$. 
We take the initial condition $u_0=0.9 \, Q$ and set $L_c=20$, considering $x \in [-L_c,L_c]$; we also set $\Delta x=0.05$ and $\Delta t=0.005$.
As before, we do 100 runs to approximate the expectation of either mass or energy.
Recall that the identity \eqref{mass_add} and the inequality \eqref{energy_add} give linear dependence on time and square dependence on the noise strength
 $\epsilon$, similar to that  in our upper estimates for the discrete quantities \eqref{affine_mass_dis}, \eqref{affine_max_mass_dis} (for mass) and 
 \eqref{E-instant-Hdis_add}, \eqref{EMaxHdis_add} (energy). The results 
are shown in Figure \ref{F:eps-comparison-add}, where we plot the expectation of the instantaneous quantities, $\mathbb E(M(u(t)))$ and $\mathbb E(H(u(t)))$. 
We omit figures for $\mathbb E(\sup_{s \leq t}M(u(s)))$ and $\mathbb E(\sup_{s \leq t}H(u(s)))$, since we get the same behavior as shown in 
Figure \ref{F:eps-comparison-add}, and both discrete upper estimates \eqref{E-instant-Hdis_add} and \eqref{EMaxHdis_add} give similar dependence on
 all  parameters.

\begin{figure}[ht]
\includegraphics[width=0.45\textwidth]{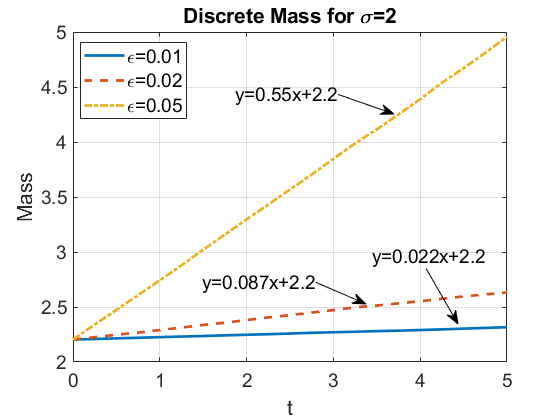}
\includegraphics[width=0.45\textwidth]{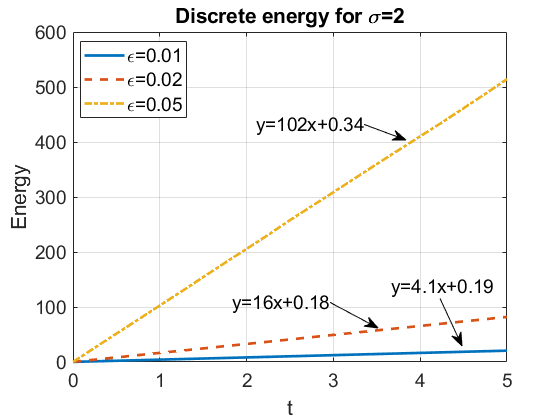}
\includegraphics[width=0.45\textwidth]{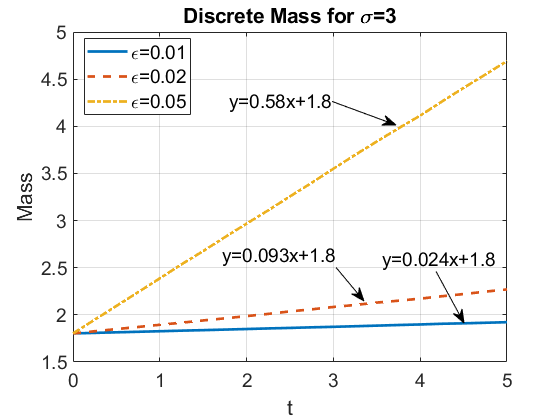}
\includegraphics[width=0.45\textwidth]{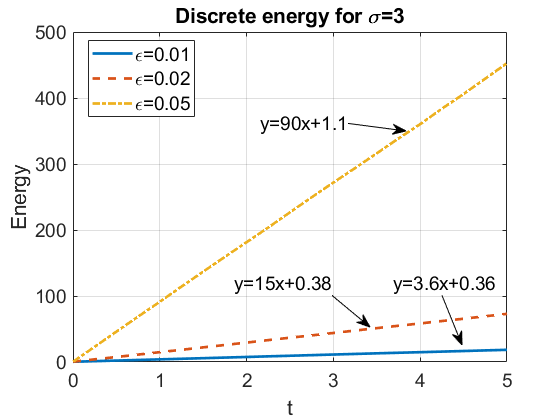}
\caption{Additive noise. Dependence of the expected value of (instantaneous) mass (left) and energy (right) on the strength of the noise $\epsilon$. Top: $L^2$-critical ($\sigma =2$), bottom: $L^2$-supercritical ($\sigma = 3$).}
\label{F:eps-comparison-add}
\end{figure}

Next, we show the dependence of the discrete mass and energy on the length of the computational interval $L_c$ and the step size $\Delta x$. 
We compare the growth of both expected mass and energy for two values of the length $L_c=20$ and $L_c=40$, see Figure \ref{NLS_Lc_add}, 
which shows the linear 
dependence for both expected values of the mass and the energy: the $L^2$-critical case ($\sigma=2$) is shown in the top row, and the $L^2$-supercritical case ($\sigma=3$) is in the bottom row. Note that the slope doubles as we double the length of the computational interval $L_c$. 

\begin{figure}[ht]
\includegraphics[width=0.45\textwidth]{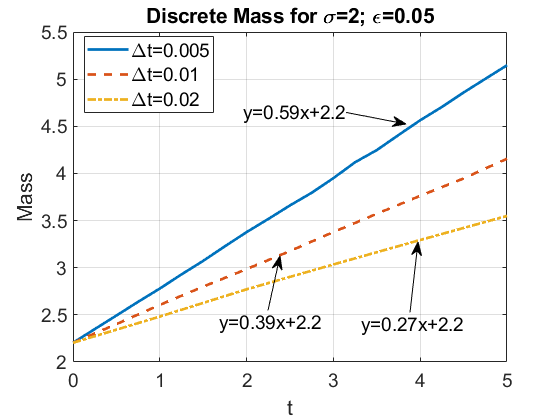}
\includegraphics[width=0.45\textwidth]{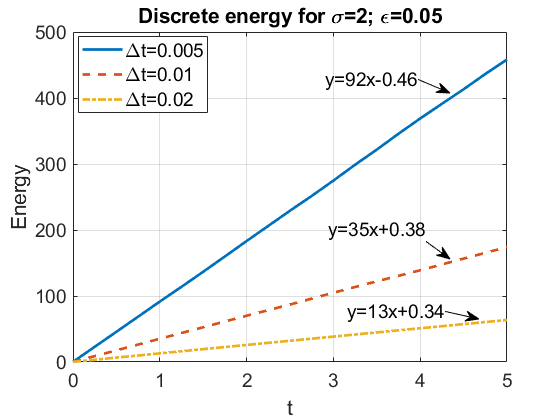}
\caption{Additive noise, $\epsilon=0.05$. Dependence of mass and energy on the time step-size $\Delta t$ in the $L^2$-critical case. }
\label{F:Deltat-add}
\end{figure}

\begin{figure}[ht]
\includegraphics[width=0.45\textwidth]{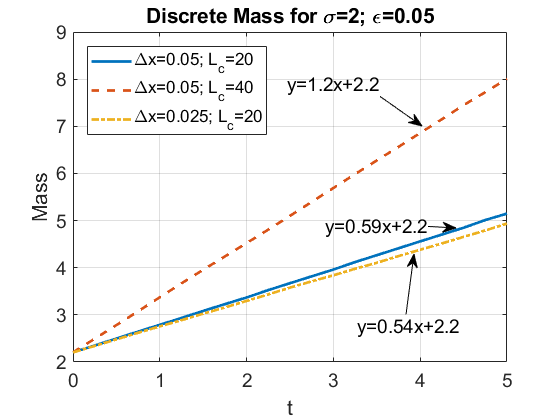}
\includegraphics[width=0.45\textwidth]{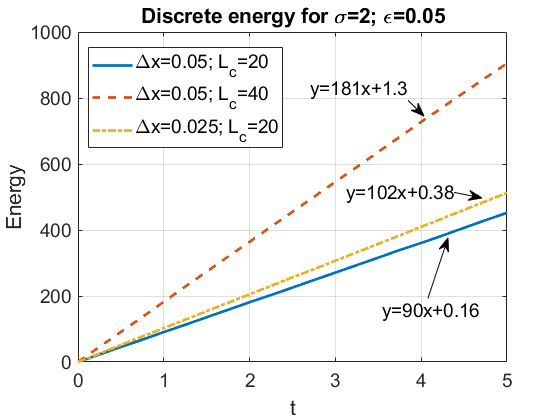}
\includegraphics[width=0.45\textwidth]{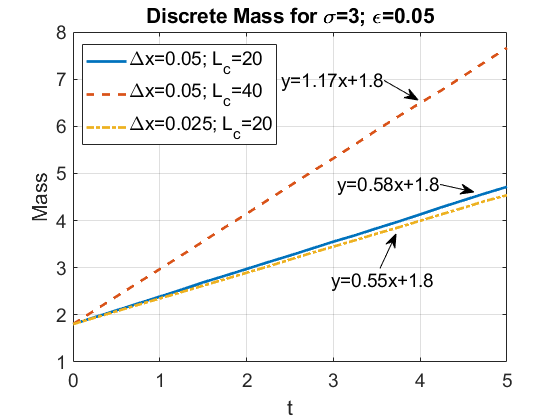}
\includegraphics[width=0.45\textwidth]{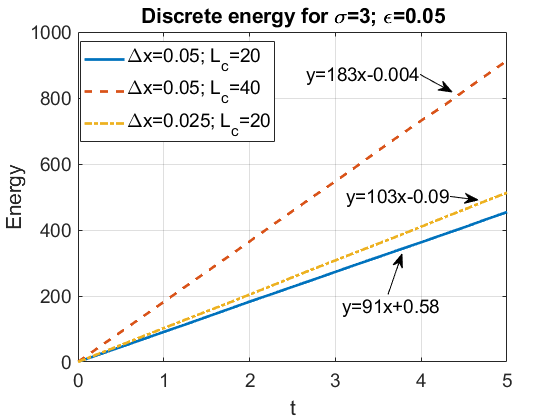}
\caption{Additive noise, $\epsilon=0.05$. Dependence of the expected value of the mass and energy on the length of the interval $L_c$ and space step-size $\Delta x$. Top: $L^2$-critical ($\sigma =2$), bottom: $L^2$-supercritical ($\sigma = 3$).}
\label{NLS_Lc_add}
\end{figure}

The dependence on the time step-size of both discrete mass and energy is shown in Figure \ref{F:Deltat-add}.
We show the dependence in the $L^2$-critical case and omit the supercritical case it is similar. 

We also mention that we studied the growth of mass and energy for other initial data, for example, gaussian $u_0 = A \, e^{-x^2}$, and obtained similar results,
see Figure \ref{F:gaussian-add}. 
\begin{figure}[ht]
\includegraphics[width=0.43\textwidth]{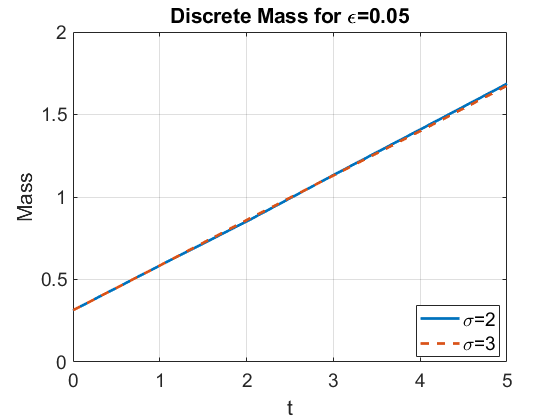}
\includegraphics[width=0.43\textwidth]{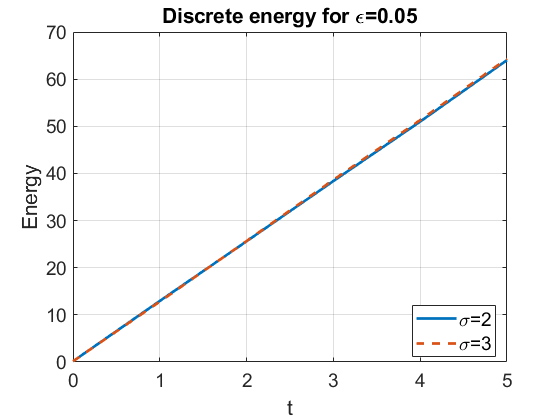}
\caption{Additive noise, $\epsilon=0.05$. Time dependence of the discrete mass (left) and energy (right) for the gaussian initial condition $u_0=0.5 \, e^{-x^2}$. 
(Here, both mass and energy coincide regardless of the nonlinearity, $\sigma = 2$ or $3$, since the only dependence is in the potential part of energy, 
which creates a very small  difference.)}
\label{F:gaussian-add}
\end{figure}
In this section, we investigated how used-to-be conserved quantities (mass and energy) in the deterministic setting behave in the stochastic case with both 
multiplicative and additive approximations of the space-time white noise. Our next goal is to look at a global picture and study how behavior of solutions is 
affected by the noise on a more global scale. We will see that in some cases the noise forces solutions to blow-up and in other instances, 
the noise will prevent blow-up formation (similar investigations were done in \cite{DM2002a} and references therein). 
We confirm some of their findings, and then investigate the blow-up dynamics (rates, profiles, etc). Before we venture into that study, 
we need to refine our numerical method, which we do in the next section.

\section{Numerical approach, refined}\label{S:4}

To study solitons and their stability numerically, it is useful to have a non-uniform mesh to capture well certain spatial features. For that we use  
a finite difference method with non-uniform mesh.
To study specific details of the evolution (such as formation of blow-up), we implement mesh refinement. However, to keep the algorithm efficient, 
the mesh refinement is applied only at certain time steps, when it is necessary. By a carefully chosen mesh-refinement strategy and a specific interpolation 
during the refinement (which we introduce below), we are able to keep the discrete mass at the same value before and after the mesh-refinement. 
Therefore, the discrete mass is exactly conserved at all times in our time evolution on $[0,T]$ (in the deterministic and multiplicative noise settings). 

We note that in the deterministic theory, solutions either exist globally in time or blow-up in finite time, and there are various results identifying thresholds 
for such a dichotomy. In the probabilistic setting blow up may hold in finite time with some (positive) probability  even for small initial data. 
Indeed in  \cite{dBD2005} it is shown that for a multiplicative stochastic perturbation (driven by non-degenerate noise with a regular enough space
 correlation)  given   any non-null initial data  there is blows up with positive probability. 
Therefore, when we study solutions of SNLS \eqref{E:NLS}, we may refer to the types of solutions as globally existing (a.s.), long-time existing 
(perhaps with some estimates on the time of existence), and blow-up in finite time (with positive probability, or a.s.) solutions. 

We also mention that as an extra bonus for a multiplicative noise, our algorithm has very small fluctuation (on the order of $10^{-12}$) 
in the difference of the actual mass
 \eqref{D: mass}, which is approximated by the composite trapezoid rule;  see \eqref{E: approxi mass}. 
The tiny difference is observed in all scenarios of solutions: globally existing, long-time existing  and blow-up in finite time (the difference is on the order of 
$10^{-12}$), which we demonstrate in Figure \ref{mass 5p}. This suggests that our algorithm is very accurate in all scenarios of solutions.

\subsection{Mesh-refinement strategy}
When a solution starts concentrating or localizing spatially,  in order to increase accuracy, 
it is necessary to put more points into that region. For example, as blow-up starts focusing  towards a singular point as $t \rightarrow T$, 
the singular region will benefit from having more grid points. In this subsection, we discuss the mesh-refinement strategy. 
The idea  comes from the scaling invariance of the NLS equation, or the dynamic rescaling method from \cite{LePSS1987}, \cite{SS1999},
 \cite{YRZ2018} and \cite{YRZ2019}.

At time 0 the computational interval $[-L_c,L_c]$ is discretized into  
$N_0+1$  grid points $\{x^0_0,\cdots,x^0_{N_0} \}$ (which may  be equi-distributed, 
since we typically begin with a uniform space mesh). When we proceed, we check at each time step if the scheme fulfills a tolerance criterion, described below.  

As we mentioned in the introduction, the stable blow-up dynamics for the deterministic NLS consists of the self-similar regime with the rescaled parameters
\begin{align}\label{E: self-similar}
u(t,x)=\frac{1}{L(t)^{1/ \sigma}} v(\tau,\xi), \quad \xi=\frac{x}{L(t)}, \quad \tau=\int_0^t \frac{1}{L^2(s)} ds,
\end{align}
where $v(\xi,\tau)$ is a globally (in $\tau$) defined self-similar solution. 
We do not rescale the equation \eqref{E:NLS} into a new equation as we do not use the dynamic rescaling method due to regularity issues. 
However, we still adopt the rescaling idea for our mesh-refinement algorithm. 
Assume $\xi$ is equi-distributed for all time steps $t_m$ and $\Delta \xi =\xi_{1}-\xi_0$. Thus, we assume that there is a mapping $L(t_m)$, which maps the point $x_j^m \rightarrow \xi_j$.
Using \eqref{E: self-similar} or $L(t)^{1/\sigma} u(t) = v(\tau)$ with our discretization, 
we get 
\begin{align}\label{E: u-v-L}
L(t_m)^{1/ \sigma} \left( u(x^m_j)- u(x^m_{j-1}) \right) = v(\xi_j)-v(\xi_{j-1}),
\end{align}
where both sides are well-behaved (since $v$ is now global), and thus, should have $O(1)$ value (referred to as the {\it moderate} value) for $j=0,1,\cdots, N_m$. 
(The rescaled solution $v(\xi)$ is well-behaved as well). Using the second relation in \eqref{E: self-similar}, we define the discretization of the mapping of $L(t_m)$ 
at each interval $[\xi_{j-1},\xi_{j}]$: 
$$
L^m_j= \frac{x^m_j-x^m_{j-1}}{\Delta \xi}.
$$
Putting this into \eqref{E: u-v-L} and using the fact that $\Delta \xi$ is a constant, we obtain that
$$
C_d:=\lbrace {x^m_j-x^m_{j-1}} \rbrace^{1/ \sigma} \left( u(x^m_j)- u(x^m_{j-1}) \right)
$$ 
remains  {\it moderate}  as time evolves for each $j=1,2,\cdots, N_m$.

Therefore, we set the tolerance to be
\begin{align}\label{Tol1}
M_{tol}^1=Tol_1 \cdot \max_j \{ (\Delta x^0_j)^{1/\sigma} \cdot  |u^0_{j+1}-u^0_j| \},
\end{align}
where $Tol_1$ is the constant we choose at $t=t_0$ (e.g., $Tol_1=2, 2.5$ or $5$). 
This criterion is focused on the size of the quantity  $u^m_{j+1}-u^m_{j}$. 
As the solution reaches higher and higher amplitudes, we refine the grid and  insert  more points, in particular, to avoid the under-resolution issue. 

In a similar way, we set  
\begin{align}\label{Tol2}
M_{tol}^2=Tol_2 \cdot \max_j\{ (\Delta x^0_j)^{1/\sigma} \cdot  |u^0_{j+1}+u^0_j| \} ,
\end{align}
where $Tol_2$ is the constant we choose at the initial time $t=t_0$ (e.g., $Tol_2=0.5$ or $1$).

At each time step $t_m$, we compute the quantities $\gamma_j^m= (\Delta x^m_j)^{1/\sigma} \cdot |u^{m}_{j+1}-u^{m}_j|$ and
 $\eta_j^m= (\Delta x^m_j)^{1/\sigma} \cdot |u^{m}_{j+1}+u^{m}_j|$ on each interval $[x_{j}^m,x_{j+1}^m]$. If at time $t=t_m$ we have
 $\gamma_j^m>M_{tol}^1$,  or $\eta_j^m>M_{tol}^2$ for some $j$'s, we divide the $j$th interval $[x_{j}^m,x_{j+1}^m]$ into two sub-intervals 
 $[x_{j}^m,x_{j+\frac{1}{2}}^m]$ and $[x_{j+\frac{1}{2}}^m,x_{j+1}^m]$. 
Then, the new value $u^m_{j+\frac{1}{2}}$ is needed. We discuss the strategy for obtaining $u^m_{j+\frac{1}{2}}$ with the mass-preserving property 
in the next subsection.
After using  this midpoint refinement, we continue our time evolution to the next time step $t_{m+1}$.

\subsection{Mass-conservative interpolation in the refinement}
Recall that when the tolerance is not satisfied at the $j^{th}$ interval, we refine the mesh by dividing that interval into two sub-intervals,
 and hence, we need an interpolation to find the new value of $u^m_{j+\frac{1}{2}}$ at the point $x^m_{j+\frac{1}{2}}= \frac{1}{2}(x^m_j+x^m_{j+1})$. 

A classical approach is to apply a linear interpolation (as, for example, in \cite{DM2002a}):  
$$
u^m_{j+\frac{1}{2}}=\frac{1}{2}(u^m_j+u^m_{j+1}).
$$  
When we add this middle point, the length of each interval $[x^m_j, x^m_{j+\frac{1}{2}}]$ and 
$[x^m_{j+\frac{1}{2}}, x^m_{j+1}]$ simply becomes $\frac{1}{2} \Delta x^m_j$. 
Unfortunately, this widely used linear interpolation does not conserve the discrete mass. Indeed, let the discrete  mass at the $j^{th}$ interval before the mesh refinement be
\begin{align}\label{E: M before}
M_j= \frac{1}{4}\left[ |u^m_j|^2( \Delta x^m_j+\Delta x^m_{j-1} )
+|u^m_{j+1}|^2( \Delta x^m_{j+1}+\Delta x^m_j ) \right],
\end{align}
and the mass after the mesh refinement be defined as
\begin{align}\label{E: M after}
\tilde{M}_j=\frac{1}{4} \left[ |u^m_j|^2\Big(\frac{1}{2} \Delta x^m_j +\Delta x^m_{j-1} \Big) + |u^m_{j+\frac{1}{2}}|^2 \Delta  x^m_j 
+|u^m_{j+1}  |^2\Big(\frac{1}{2}  \Delta x^m_j +\Delta x^m_{j+1} \Big)  \right].
\end{align} 
Then a simple computation shows that 
\begin{align}\label{E: linear interpo}
M_j - \tilde{M}_j= \frac{1}{4} |u^m_j-u^m_{j+1}|^2 \Delta x^m_j.
\end{align}
Hence, $\tilde{M}_j<M_j$ on some subset of $\Omega$ (where the random variables $u^m_j$ and $u^m_{j+1}$ differ), which is a non-empty set. 
In this linear interpolation, we suffer a loss of mass at each step of the mesh-refinement  procedure. 
In another popular interpolation, via the cubic splines, a similar analysis shows that the scheme suffers the increase of mass 
at each step of the mesh-refinement procedure. To avoid these two problems,  we proceed as follows.

We set the two quantities \eqref{E: M before} and \eqref{E: M after} to be equal to each other, i.e., $M_j=\tilde{M}_j$, by solving this equation with
the fact that 
$x^m_{j+1}-x^m_{j+\frac{1}{2}}=x^m_{j+\frac{1}{2}}-x^m_j=\frac12 \Delta x^m_j$, 
we obtain
\begin{align}\label{E: mass interpo}
| u^m_{j+\frac{1}{2}}|^2=\frac{1}{2}\left( |u^m_j|^2+|u^m_{j+1}|^2 \right).
\end{align}
\begin{figure}[ht]
\includegraphics[width=0.5\textwidth]{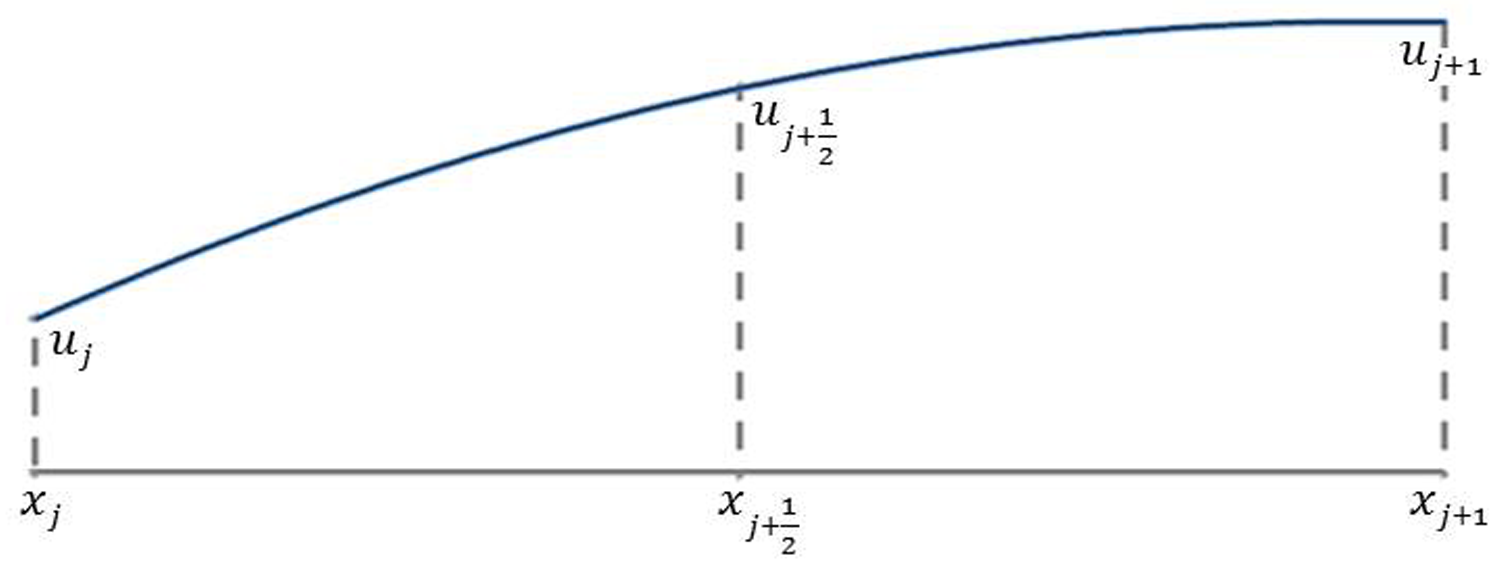}
\caption{Quadratic interpolation in \eqref{E: mass interpo} to obtain $u^m_{j+\frac12}$ (index $m$ is omitted). }
\label{F:I}
\end{figure}

To implement the condition \eqref{E: mass interpo}, one choice is to set 
\begin{align}\label{E: interpo eqn}
\begin{cases}
\Re (u^m_{j+\frac{1}{2}})=\sqrt{\frac{1}{2}\left[ \Re (u^m_j)^2+ \Re (u^m_{j+1})^2 \right] } \, \mathrm{sgn} \big(\Re (u^m_j)+\Re (u^m_{j+1}) \big), \\
\Im (u^m_{j+\frac{1}{2}})=\sqrt{\frac{1}{2}\left[ \Im  (u^m_j)^2+ \Im  (u^m_{j+1})^2 \right] } \, \mathrm{sgn} \big(\Im  (u^m_j)+\Im  (u^m_{j+1}) \big).
\end{cases}
\end{align}
This is  what we use in our simulations. We next describe  the steps of our  full numerical algorithm.

\subsection{The algorithm}\label{S:algorithm}
The full implementation of our algorithm proceeds as follows:
\begin{itemize}
\item[1.] Discretize the space in the uniform mesh and set up the values of tolerance $Tol_1$ and $Tol_2$.
\item[2.] Apply the mass-conservative numerical schemes \eqref{mass-energy}, \eqref{NS: crank-nicholson} or \eqref{NS:relaxation} for the time evolution from $u^{m-1}$ to reach $u^m$ with the time step size $\Delta t_{m-1}$.
\item[3.] At $t=t_m$, change the time step size by $\Delta t_{m}=\frac{\Delta t_0}{\| u(\cdot,t_m) \|_{\infty}^{2 \sigma}}$ for the next time evolution (thus, $t$ never reaches the blow-up time $T$, in case there is a blow-up).
\item[4.] If the solution meets the tolerance ($Tol_1$ or $Tol_2$) on some intervals $[x_j,x_{j+1}]$, we divide those intervals into two sub-intervals.
\item[5.] Apply the {\it mass-conservative interpolation} \eqref{E: mass interpo} to obtain the value of $u^m_{j+\frac{1}{2}}$.
\item[6.] Continue with the time evolution to $t=t_{m+1}$ by applying \eqref{mass-energy}, \eqref{NS: crank-nicholson} or \eqref{NS:relaxation}. 
\end{itemize}

A few remarks are due. 
First of all, this algorithm is applicable in the deterministic case. 
To our best knowledge, this is the first mesh-refinement numerical algorithm that {\it conserves} the discrete mass exactly before and after the refinement, which is especially important when simulating the finite time blow-up in the 1D focusing nonlinear Schr\"odinger equation with or without stochastic perturbation. Moreover, in the deterministic and multiplicative noise cases the discrete mass is conserved from the initial to terminal times.  
We note that in studying and simulating the blow-up solutions in the (deterministic) NLS equation, the dynamic rescaling or moving mesh methods are used (since solutions
  have some regularity); however, in the stochastic setting, those methods are simply not applicable because noise destroys regularity in the space variable.  

Secondly, its full implementation is needed for solutions that concentrate locally or blow up in finite time, where the refinement and mass-conservation 
are crucial features 
to ensure the reliability of the results. However, the algorithm is also applicable in the cases  where the solution exists globally or long enough for numerical
 simulations. Indeed, if we start with the uniform mesh and remove the steps (1), (3), (4) and (5), it becomes a widely used second order numerical scheme for 
 studying the NLS equation (in both deterministic and stochastic cases) without considering the singular solutions. 

When investigating solutions, which do not form singularities (exist globally in time or on sufficiently long  time interval), the procedures (1), (3), (4) and (5) 
are not necessary and we omit them. When studying the blow-up solutions (in Section \ref{S: blow-up}), 
we incorporate fully all steps in order to obtain  satisfactory results. 
When testing our simulations of blow-up solutions, not only the error of the discrete mass $\mathcal{E}_1^m[M]$ from \eqref{E:error-Dmass} a
is checked, but also the discrepancy of the actual mass, approximated by the composite trapezoid rule at each time step, is checked, that is,
\begin{align}\label{E; error amass}
\mathcal{E}_2^m[M]=\max_m \left\lbrace M_{\mathrm{app}}[u^m] \right\rbrace-\min_m \left\lbrace M_{\mathrm{app}}[u^m] \right\rbrace, ~~\mbox{where}
\end{align}
\begin{align}\label{E: approxi mass}
M_{\mathrm{app}}[u^m]=\frac{1}{2}|u_0^m|^2 \Delta x_0+\sum_{j=1}^{N-2} |u_j^m|  \Delta x_j+\frac{1}{2}|u_N^m|^2 \Delta x_{N-1} .
\end{align}

For this test, we choose $u_0=1.05Q$ and consider only the $L^2$-critical case ($\sigma=2$), comparing $\epsilon=0$ (deterministic case) with $\epsilon=0.1$ (multiplicative noise case). The initial spatial step-size is set to $\Delta x=0.01$, and the initial temporal step-size is set to $\Delta t_0=\Delta x/4$. We take the computational domain to be $[-L_c,L_c]$ with $L_c=5$. Figure \ref{blowup mass 5p} shows the dependence of $\mathcal{E}_1^m$ and $\mathcal{E}_2^m$ on the focusing scaling parameter $L(t)=\frac{1}{\|u(t)\|_{\infty}^{\sigma}}$. 

\begin{figure}[ht]
\includegraphics[width=0.45\textwidth]{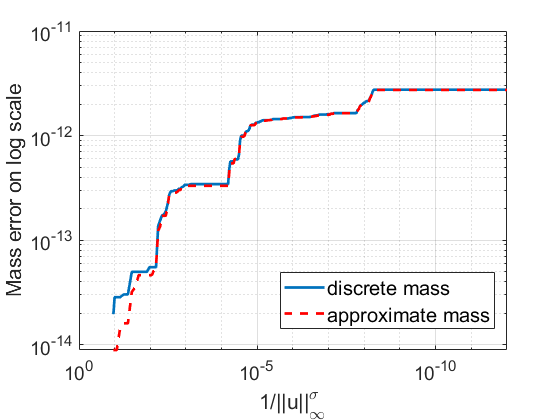}
\includegraphics[width=0.45\textwidth]{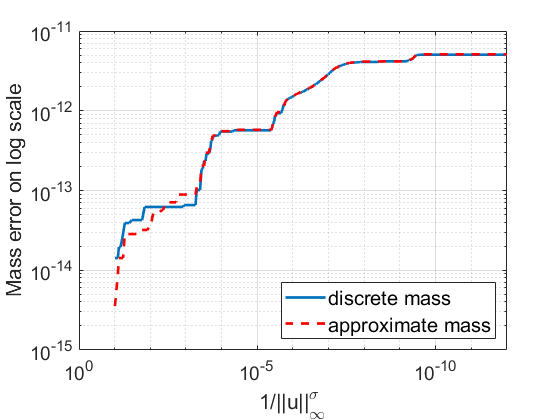}
\caption{The error of discrete and actual masses $\mathcal{E}_1^m$ and $\mathcal{E}_2^m$ for the $L^2$-critical case with or without the multiplicative noise. Left: $\epsilon=0$.  
Right: $\epsilon=0.1$.}
\label{blowup mass 5p}
\end{figure}

Observe  that both the discrete mass and approximation of the actual mass are conserved well even when the focusing parameter reaches $\sim 10^{-12}$. 
Such high precision in mass conservation  justifies well the efficiency of our schemes. 
We also tested other types of initial data (e.g., gaussian data $u_0=A \, e^{-x^2}$ ), different noise strength ($\epsilon=0.2,0.5$)   
and the supercritical power of nonlinearity ($\sigma = 3$); the precision is similar to that shown in Figure \ref{blowup mass 5p}. 

In the next two sections we discuss global behavior of solutions, showing how solitons behave for various nonlinearities (Section \ref{S:solitons}),
and then investigate the formation of blow-up (Section \ref{S: blow-up}) including our findings on profiles, rates and localization.  

\section{Numerical simulations of global behavior of solutions}\label{S:solitons} 

We again consider initial data of type $u_0 = A\, Q$, where $A>0$ and $Q$ is the ground state \eqref{E:Q}.
In the deterministic setting one would consider two cases for numerical simulations, namely, $A<1$ (which guarantees the global existence 
and $A>1$ (which could be used to study blow-up solutions). In the stochastic setting we use similar data;  however, as we will see (in Table \ref{T: blow-up percentage p5}), 
we may not know {\it a priori} if the solution is global or blows up in finite time (a.s. or with some positive probability).
 For example, the condition $A<1$ does not necessarily guarantee global existence, or even sufficiently long (for numerical simulations) time existence 
 as can be seen in 
Tables \ref{T: blow-up percentage p5} and \ref{T: blow-up percentage p7}.

We consider additive noise first. Putting sufficiently large $\epsilon$ and tracking for a sufficiently long time, we observe that small data leads to blow-up for the cases 
 $\sigma=2$ and $\sigma=3$. 
 
\begin{figure}[ht]
\includegraphics[width=1\textwidth, height=9.5cm]{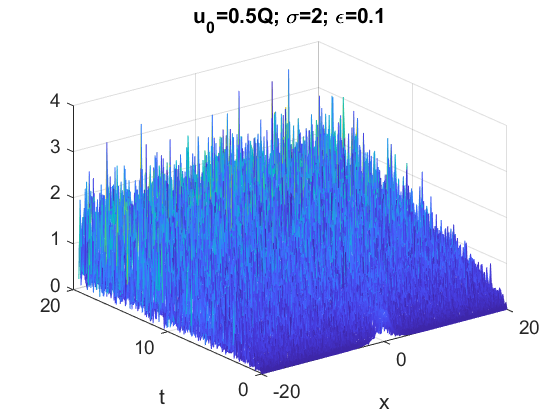}
\includegraphics[width=0.32\textwidth, height=3.5cm]{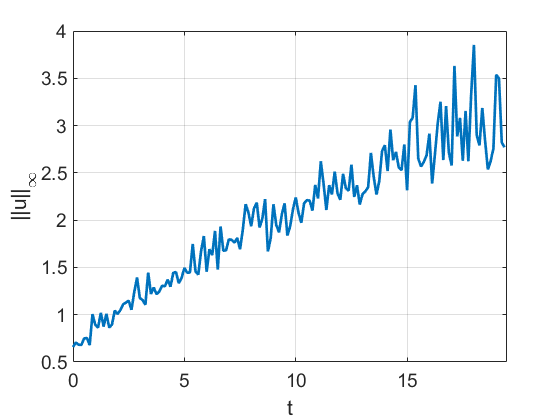}
\includegraphics[width=0.32\textwidth, height=3.5cm]{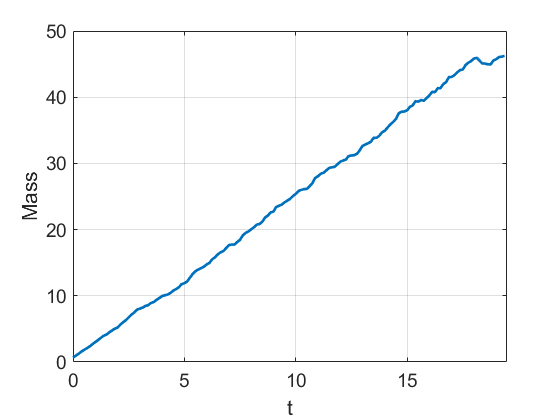}
\includegraphics[width=0.32\textwidth, height=3.5cm]{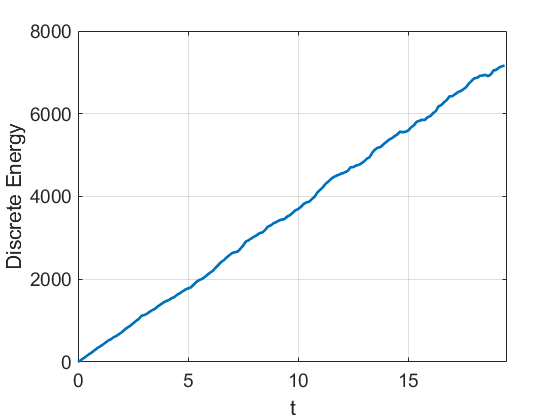}
\caption{Additive noise, $\epsilon=0.1$, $u_0=0.5Q$, $L^2$-critical case. The solution grows in time until the fixed point iteration fails. Bottom: time dependence of $\|u\|_{\infty}$, mass and energy.}
\label{F:p5 05Q prof}
\end{figure}

For example, in Figure \ref{F:p5 05Q prof}, we take $u_0=0.5Q$ (far below the deterministic threshold) with sufficiently strong noise  $\epsilon=0.1$ and run for (computationally) long time: the fixed point iteration for solving the MEC scheme \eqref{mass-energy} fails to converge after 2000  iterations at time $t \approx 19.485$, which indicates that $u^{m+1}$ is far from $u^m$ at $t_m \approx 19.485$. 
 The numerical scheme can not be run any further, and this is typically considered as the indication of the blow-up formation
  (see below comparison with the $L^2$-subcritical case).

Figure \ref{F:p5 05Q prof} shows that the additive noise can {\it create} blow-up in finite time. In other words, 
the initial data, which in the deterministic case
were to produce a globally existing scattering solution, in the additive forcing case could  evolve towards the blow-up. 
This is partially due to the fact that the additive noise makes the mass and energy grow in time; see the bottom subplots in Figure \ref{F:p5 05Q prof}, 
where both mass and energy grow linearly in time. Note that we start with a single soliton profile with a small amplitude ($0.5\,Q$) and eventually 
the noise destroys the soliton profile with the growing $L^\infty$ norm (left bottom subplot in Figure \ref{F:p5 05Q prof}). 

\begin{figure}[ht]
\includegraphics[width=1\textwidth, height=9.5cm]{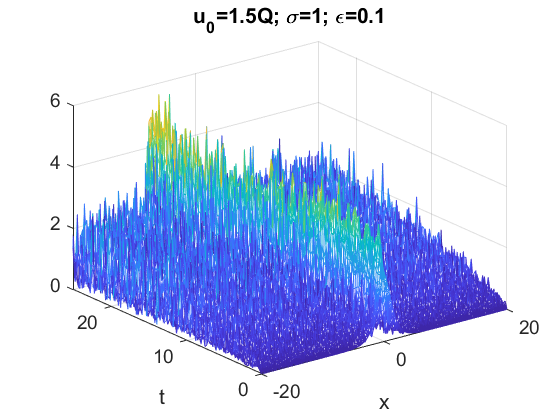}
\includegraphics[width=0.32\textwidth, height=3.5cm]{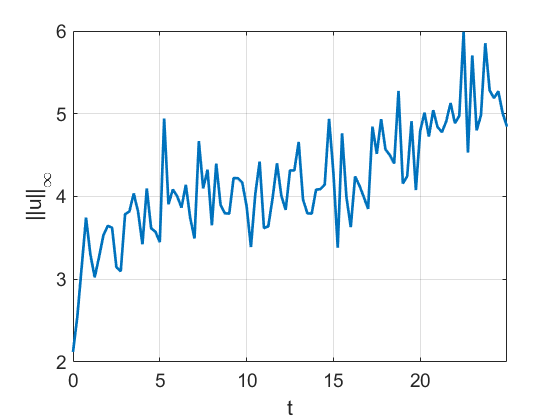}
\includegraphics[width=0.32\textwidth, height=3.5cm]{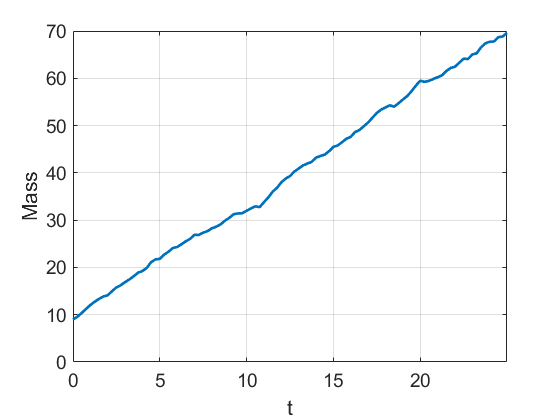}
\includegraphics[width=0.32\textwidth, height=3.5cm]{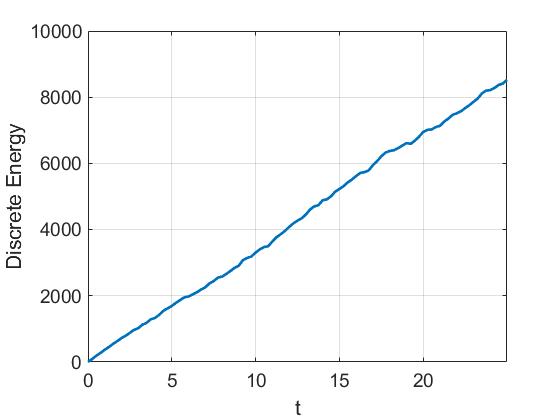}
\caption{Additive noise, $\epsilon=0.1$, $u_0=1.5Q$, $L^2$-subcritical case ($\sigma=1$). 
Top: time evolution of $|u(x,t)|$. Bottom: time dependence of $\|u(t)\|_{\infty}$, mass and energy.}
\label{F:p3 prof-1.5Q}
\end{figure}

It is also interesting to compare this behavior with the $L^2$-subcritical case ($\sigma=1$), where in the deterministic case all solutions are global (there is no blow-up for any data), see \cite{dBD2003}.
Figure \ref{F:p3 prof-1.5Q} shows time evolution of the initial condition $u_0=1.5Q$ with the strength of the additive noise $\epsilon=0.1$ 
(same as in Figure \ref{F:p5 05Q prof}). While the soliton profile is distinct for the time of the computation, it is obviously getting corrupted by noise: 
the $L^\infty$ norm is slowly increasing with some wild oscillations. One can also observe that mass and energy grow linearly to infinity (as $t \rightarrow \infty$);  
see bottom plots of Figure \ref{F:p3 prof-1.5Q}. Note that while there is growth of mass and energy, as well as the $L^\infty$ norm in this subcritical case,
 the fixed point iteration does not fail, indicating that there is no blow-up.

For comparison we also show the influence of smaller noise $\epsilon=0.05$ on a larger time scale ($0<t<50$) for the initial condition $u_0=Q$; 
see Figure \ref{F:p3 prof-1Q-smallnoise}. 
The smaller noise also seem to destroy the soliton with slow increase of the $L^\infty$ norm and linearly growing mass and energy;  however,
  the solution exists globally in time.

\begin{figure}[ht]
\includegraphics[width=1\textwidth, height=9cm]{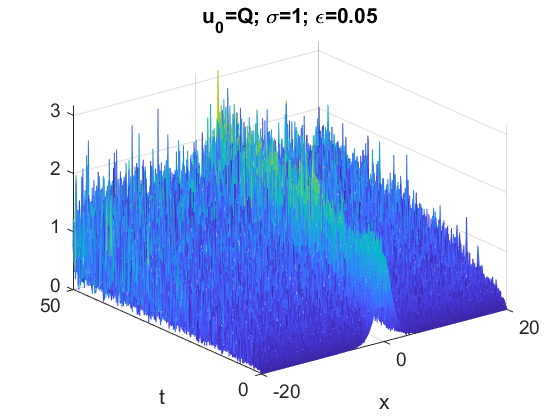}
\includegraphics[width=0.32\textwidth, height=3.5cm]{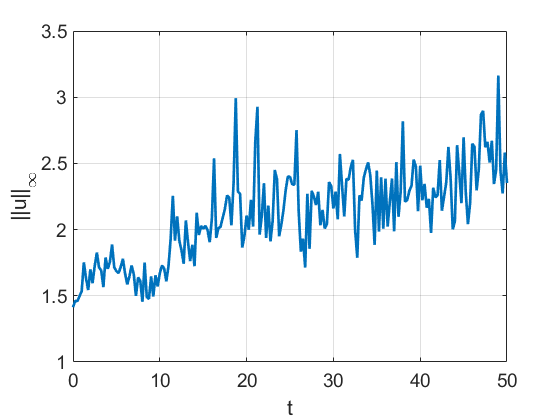}
\includegraphics[width=0.32\textwidth, height=3.5cm]{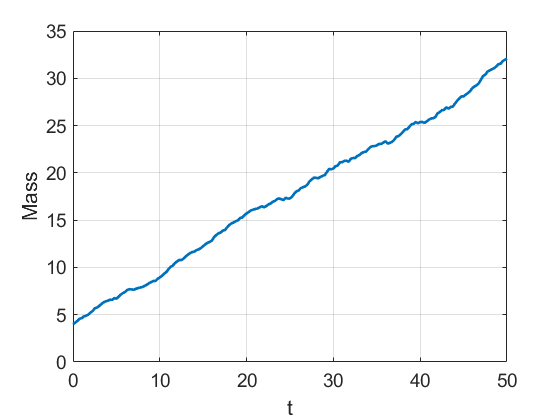}
\includegraphics[width=0.32\textwidth, height=3.5cm]{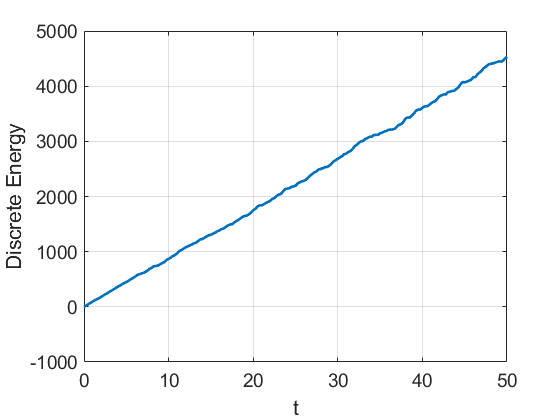}
\caption{Additive noise, $\epsilon=0.1$, $u_0=Q$, $L^2$-subcritical case ($\sigma=1$). Top: time evolution of $|u(x,t)|$. Bottom: time dependence of $\|u(t)\|_{L^\infty}$, mass and energy.}
\label{F:p3 prof-1Q-smallnoise}
\end{figure}

Returning to the $L^2$-critical and supercritical SNLS, we have seen that even small initial data can lead to blow-up. 
Therefore, we next compute the percentage of solutions that blow 
 up until some finite time (e.g., $t=5$). We run $N_t = 1000$ trials to track solutions for various values of magnitude $A$ in the initial data $u_0 = A\, Q$, 
 with $A$ close to $1$. 
In Table \ref{T: blow-up percentage p5} we show the percentage of finite time blow-up solutions in the $L^2$-critical case ($\sigma=2$) with an additive noise
 ($\epsilon = 0.01, 0.05, 0.1$): we take $A = 0.95, 1$ and $1.05$ (in the deterministic case these amplitudes would, respectively, lead to a scattering solution, a soliton, and a finite-time blow-up). Observe that blow-up occurs for $t<5$ even when $A=0.95 < 1$ with strong enough noise (e.g., when 
  $\epsilon=0.1$, we get $98.4\%$ of all solutions blow up in finite time; with $\epsilon =0.05$, we get $2.8\%$ blow-up solutions, 
  see Table \ref{T: blow-up percentage p5}). 
This is in contrast with multiplicative noise as well as with the deterministic case in the $L^2$-critical setting. 
{\small
\begin{table}[ht]
\begin{tabular}{|c|c|c|c|}
\hline
 $A$            & $0.95$ & $1$ & $1.05$   \\
 \hline
 $\epsilon=0.01$& $0$    & $0.34$ & $1$   \\
 \hline
 $\epsilon=0.05$& $0.028$ &$0.926$& $1$    \\
 \hline
 $\epsilon=0.1$ & $0.984$& $0.999$ &$0.999$   \\
 \hline
\end{tabular}
\linebreak
\caption{Additive noise. Percentage of blow-up solutions with initial data $u_0=A Q$ in the $L^2$-critical case ($\sigma =2$) with $N_t=1000$ trials and running time $0<t<5$.}
\label{T: blow-up percentage p5} 
\end{table}
}

Table \ref{T: blow-up percentage p7} shows the percentage of blow-up solutions in the $L^2$-supercritical case ($\sigma=3$) with additive noise. 
As in  the $L^2$-critical case, solutions with an amplitude below the threshold (e.g., $A=0.95$) can blow up in finite time
(here, before $t=5$)  with an additive noise of larger strength 
(for example, when $\epsilon = 0.05$, $3\%$ of our runs blow up in finite time; for $\epsilon = 0.1$ it is 98.6\%).
{\small
\begin{table}[ht]
\begin{tabular}{|c|c|c|c|}
\hline
 $A$            & $0.95$ & $1$ & $1.05$   \\
 \hline
 $\epsilon=0.01$& $0$    & $0.753$ & $1$   \\
 \hline
 $\epsilon=0.05$& $0.030$ &$0.983$& $1$    \\
 \hline
 $\epsilon=0.1$ & $0.986$& $1$ &$1$   \\
 \hline
\end{tabular}
\linebreak
\caption{Additive noise. Percentage of blow-up solutions with initial data $u_0=A Q$ in the $L^2$-supercritical case ($\sigma=3$) with $N_t=1000$ trials and running time $0<t<5$.}
\label{T: blow-up percentage p7} 
\end{table}
}
The effect of driving a time evolution into the blow-up regime (or in other words, generating a blow-up in the cases when a deterministic solution would exist 
globally and scatter) might be more obvious in the additive case, since the noise simply adds into the evolution and does not interfere with the solution.
What happens in the multiplicative case, since the noise is being multiplied by the solution, is less obvious. 
Therefore, for completeness we mention the number of blow-up solutions we observe with $A<1$ in the multiplicative case.
 We tested the $L^2$-supercritical case with  $\epsilon=0.1$ for a multiplicative perturbation, and observed the following: for $\sigma =3$, $u_0=0.99\, Q$,
  the number $N_t = 50\, 000$ trial runs produced  
$2$ blow-up trajectories. Thus, while the probability of (specific) finite time blow-up is extremely small (in this case it is 0.004\%), it is nevertheless positive. 
The positive probability of blow-up in the $L^2$-supercritical case is consistent with theoretical results of de Bouard and Debussche \cite{dBD2005}, 
which showed that in such a case any data will lead to blow-up in any given finite time with positive probability.

In the $L^2$-critical case it was shown in \cite[Theorem 2.7]{MR2020} that if $\|u_0\|_{L^2}<\|Q\|_{L^2}$, then in the multiplicative (Stratonovich) noise case, 
the solution $u(t)$ is global, thus, no blow-up occurs. We tested the initial condition $u_0=0.99\, Q$, $\epsilon=0.1$ (same as in the $L^2$-supercritical case), 
and ran again $N_t = 50\, 000$ trials. In all cases we obtained scattering behavior (or no blow-up trajectories), thus, confirming the theory. 

We next show how the blow-up solutions form and their dynamics in both cases of noise. 

\section{Blow-up dynamics}\label{S: blow-up}

In this section, we study the blow-up dynamics and how it is affected by the noise. We continue applying the numerical algorithms introduced in Section \ref{S:3}. 
We start with the $L^2$-critical case and then continue with the $L^2$-supercritical case.  
We first observe that,  as the blow-up starts forming, there is less and less effect of the noise on the blow-up profile,  and almost no effect on the  the blow-up rate. 
However, we do notice that the noise disturbs the {\it location} of the blow-up center for different trial runs.
 
In order to better understand the blow-up behavior (and track profile, rate, location),  we run $1000$ simulations and then average over all runs. 
For the location of the blow-up, we show the distribution of the location of the blow-up center shifts and its dependence on the number of runs. 
When using a very large number of trials, we obtain a normal distribution, see Figures \ref{NLS_5p_center} and \ref{NLS_7p_center}.
For more details on the blow-up dynamics in the deterministic case we refer the reader to \cite{YRZ2018}, \cite{YRZ2019}, \cite{SS1999}, \cite{F2015}. 

\subsection{The $L^2$-critical case} \label{sec-L2critical}
We first consider the quintic NLS ($\sigma=2$) and $\epsilon=0$ (deterministic case), and then $\epsilon =0.01, 0.05$ and  $0.1$ with a multiplicative noise. 
We use generic Gaussian initial data ($u_0 =A e^{-x^2}$) as well as the ground state data ($u_0 = A Q$).   
Figure  \ref{NLS_5p_profile} shows the blow-up dynamics of $u_0 = 3 e^{-x^2}$ with $\epsilon=0.1$. 
Observe that the solution slowly converges to the rescaled ground state profile $Q$. 
 
\begin{figure}[ht]
\includegraphics[width=0.48\textwidth]{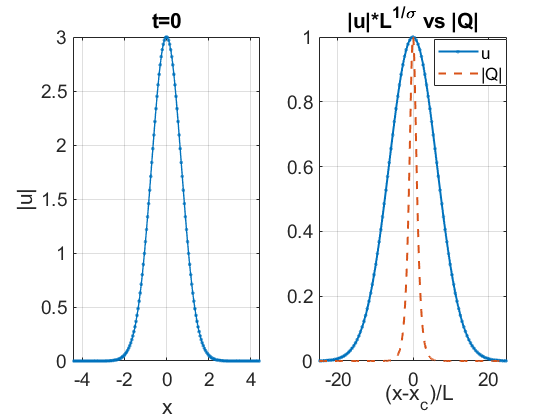}
\includegraphics[width=0.48\textwidth]{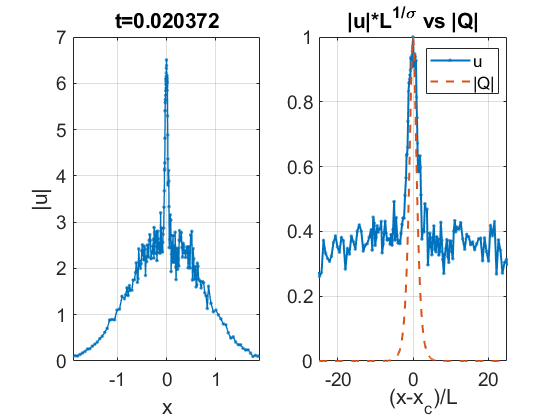}
\includegraphics[width=0.48\textwidth]{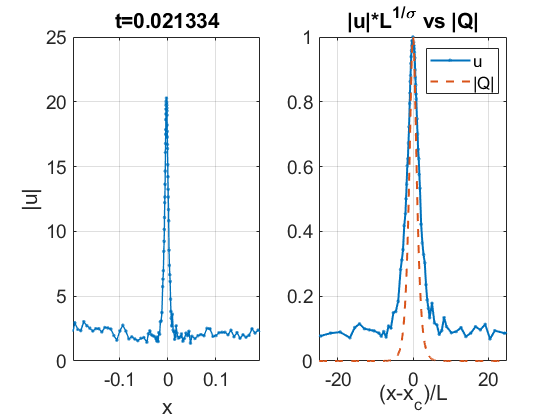}
\includegraphics[width=0.48\textwidth]{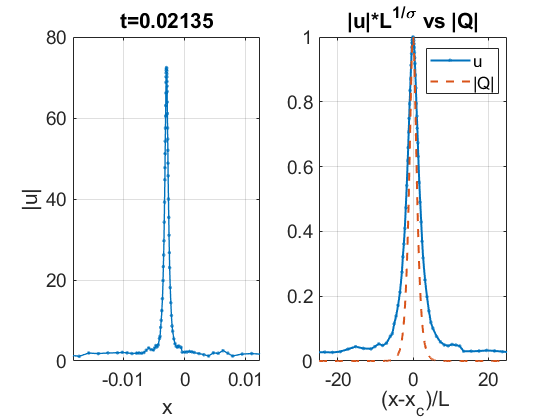}
\caption{Multiplicative noise, $\epsilon=0.1$. Formation of blow-up in the $L^2$-critical case ($\sigma = 2$): snapshots of a blow-up solution 
(given in pairs of actual and rescaled solution) at different times. Each pair of graphs shows in blue the actual solution $|u|$ and its rescaled solution 
$L^{1/\sigma}|u|$, which is compared to the absolute value of the normalized ground state solution $e^{i t}Q$ in dashed red.}
\label{NLS_5p_profile}
\end{figure}

Similar convergence of the profiles for other values of $\epsilon$ is observed (we also tested $\epsilon =0.01$ and $0.05$, 
and compared with our deterministic work $\epsilon=0$ in \cite{YRZ2018}). The last (right bottom) subplot on Figure \ref{NLS_5p_profile} shows that indeed
the profile of blow-up approaches the rescaled $Q$, however, one may notice that it converges slowly (compare this with the supercritical 
case in Figure \ref{NLS_7p_profile}).  This confirms the profile in Conjecture \ref{C:1}.

We next study the rate of the blow-up by checking the dependence of $L(t)$ on $T-t$. 
In Figure \ref{NLS_5p} we show the rate of blow-up on the logarithmic scale. 
Note that the slope in the linear fitting in each case is $\frac12$, thus, confirming the rate in Conjecture \ref{C:1}, 
$\|\nabla u(t)\|_{L^2} \sim \left( T-t \right)^{-\frac{1}{2}}$, 
possibly with some correction terms. This is similar to the deterministic $L^2$-critical case; 
see more on that in \cite{SS1999} and \cite{YRZ2018}. 


\begin{figure}[ht]
\includegraphics[width=0.45\textwidth]{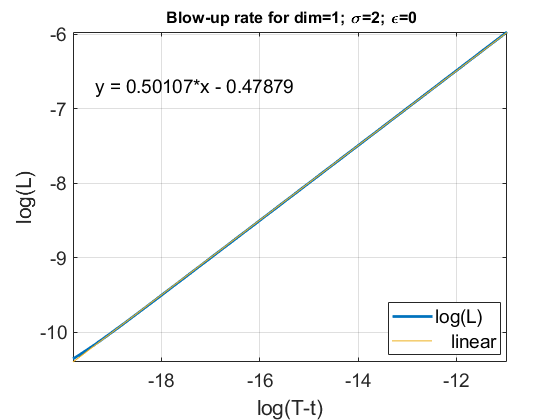}
\includegraphics[width=0.45\textwidth]{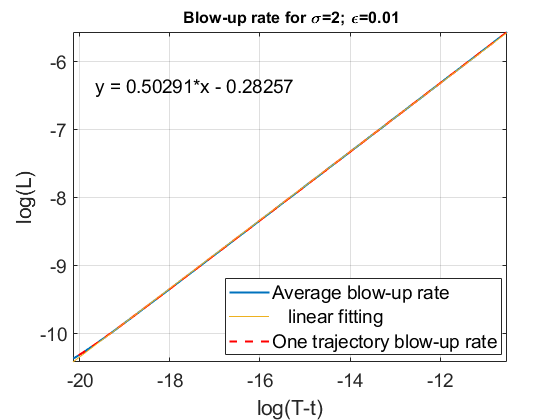}
\includegraphics[width=0.45\textwidth]{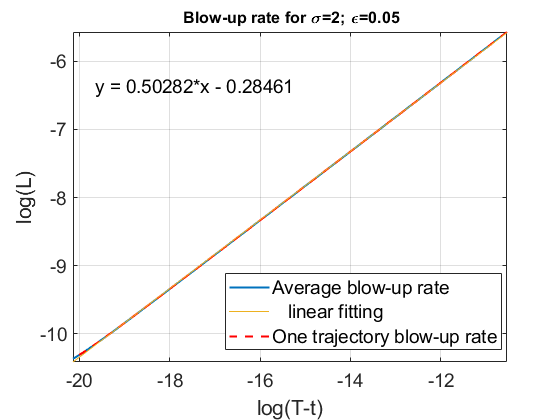}
\includegraphics[width=0.45\textwidth]{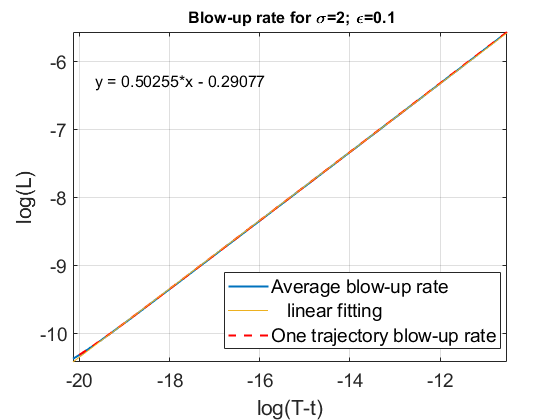}
\caption{Multiplicative noise, $L^2$-critical case. The fitting of the rate $L(t)$ 
v.s. $(T-t)$ on a log scale. The values of the noise strength $\epsilon$ are $0$ (top left), $0.01$ (top right), $0.05$ (bottom left), $0.1$ (bottom right). Observe that in all cases the linear fitting gives the slope $0.50$.}
\label{NLS_5p}
\end{figure}

To provide a justification towards the claim that the correction in the stochastic perturbation case is also of a {\it log-log} type, see \eqref{E:loglog}, 
we track similar quantities as we did in the dynamic rescaling method for the deterministic NLS-type equations; see \cite{YRZ2018}, \cite{YRZ2019}, 
\cite{YRZ2020}.  We track the quantity $a(t) =
-LL_t$, or equivalently, in the the rescaled time $\tau=\int_0^{t} \frac{1}{L(s)^2} ds$ (or $\frac{d\tau}{dt}=\frac{1}{L^2(t)}$), we have $a(\tau)=-\frac{L_{\tau}}{L}$. 
In the discrete version, by setting $\Delta \tau =\Delta t_0$, we get $\tau_m=m \cdot \Delta t_0$ as a rescaled time. Consequently, at the $m${th} step we have  
 $L(\tau_m)$, $u(\tau_m)$, and $a(\tau_m)$.
As in \cite{SS1999}, \cite{F2015}, \cite{YRZ2018}, the parameter $a$ 
can be evaluated by setting $L(t)=\left(1/\| \nabla u(t)\|_{L^2}\right)^{\frac{2}{\alpha}}$ with $\alpha=1+\frac{2}{\sigma}=2-2s$, since $s=\frac{1}{2}-\frac{1}{\sigma}$. 
Then, similar to \cite[Chapter 6]{SS1999} we get 
\begin{align}\label{E: a}
a(t) =- \frac{2}{\alpha} \frac{1}{(  \| \nabla u(t)\|_{L^2}^2)^{{\frac{2}{\alpha}+1}} } 
\int |u|^{2\sigma}\Im(u_{xx}\bar{u}) dx.
\end{align}
Here, we specifically write a more general statement in terms of the dimension $d$ and nonlinearity power $\sigma \searrow 2$, since 
the convergence of those parameters down to $d=1$ and $\sigma=2$ is crucial in determining the correction in the blow-up rate (see more in \cite{YRZ2018}),
as well as the value of $a(\tau)$ for the profile identification in the supercritical case.
The integral in \eqref{E: a} is evaluated by the composite trapezoid rule. 

Figure \ref{NLS_5p_a} shows the dependence of the parameter $a$ with respect to $\log L$ for a single trajectory (in dotted red) and for the averaged value 
over 2400 runs (in solid blue) on the left subplot (the strength of the multiplicative noise is $\epsilon=0.1$). 
Observe that a single trajectory gives a dependence with severe oscillations due to noise in the beginning,  but eventually smoothes out and converges to the average value as it approaches the blow-up time $T$. This matches our findings in Figure \ref{NLS_5p_profile}, where eventually the blow-up profile becomes smooth. 
The right subplot shows the linear fitting for $a(\tau)$ versus $1/\ln(\tau)$. One may notice small oscillations in the blue curve: perhaps with the increase of the number of runs, the blue curve could have smaller and smaller oscillations, and would eventually approach a (yellow) line). We show one trajectory dependence in dotted red, the averaged 
value in solid blue and the linear fitting in solid yellow. This gives us first confirmation that the correction term  is of logarithmic order. As in the deterministic case, we suspect that the correction is a double logarithm;  however, this will require further investigations, which are highly nontrivial (even in the deterministic case). The above confirms Conjecture \ref{C:1} up to one logarithmic correction. 

\begin{figure}[ht]
\includegraphics[width=0.48\textwidth]{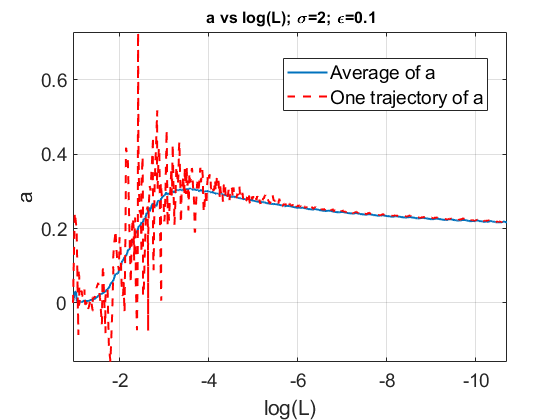}
\includegraphics[width=0.48\textwidth]{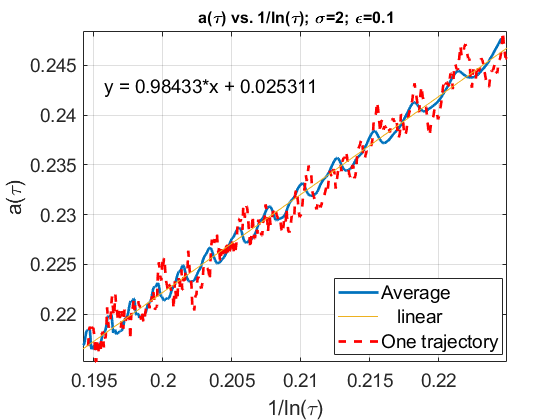}
\caption{Multiplicative noise, $\epsilon=0.1$, $L^2$-critical case. Left: $a$ vs. $\log(L)$. Right: linear fitting for $a(\tau)$ vs. $1/\ln(\tau)$.} 
\label{NLS_5p_a}
\end{figure}

\subsubsection{Blow-up location}
So far we exhibited similarities in the blow-up dynamics between the multiplicative noise case and the deterministic case. 
A feature, which we find different, is the location of blow-up. 
We observe that the blow-up core, to be precise the spatial location $x_c$ of the blow-up {\it center}, shifts away from the zero (or rather wonders around it) 
for different runs. 
We record the values $x_c$ of shifts and plot their distribution in Figure \ref{NLS_5p_center} for various values of $\epsilon$ and for different number of trials 
$N_t$ to track the dependence. Our first observation is that the center shifts further away from zero when the strength of noise  $\epsilon$ increases. 
Secondly, we observe that the shifting has a normal distribution (see the right bottom subplot with the maximal number of trials in Figure \ref{NLS_5p_center}). 
The mean of this distribution approaches $0$ when the number of runs $N_t$ increases.
We record the variance of the shifts for different $\epsilon$'s and $\sigma$'s in Table \ref{T: x_c var}. The variance seems to be an increasing function 
of the strength of the noise, which confirms our first observation above. In the same Table, we also record the $L^2$-supercritical case that is discussed later.   
\begin{figure}[ht]
\includegraphics[width=0.48\textwidth]{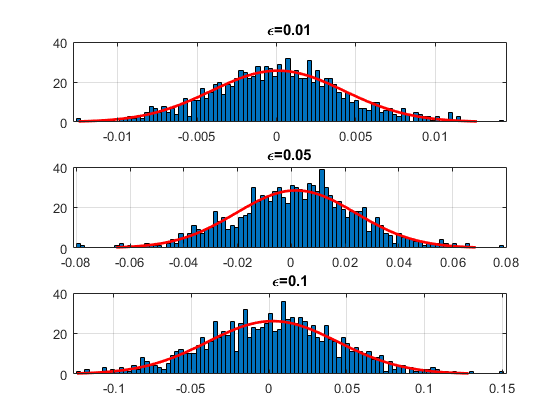}
\includegraphics[width=0.48\textwidth]{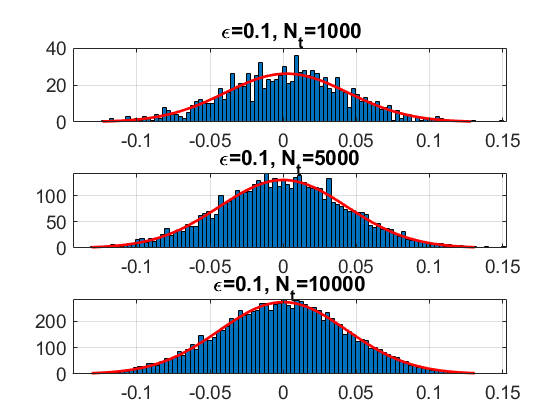}
\caption {Multiplicative noise, $\epsilon=0.1$, $L^2$-critical case. Left: shifts $x_c$ of the blow-up center for different noise strength $\epsilon$ 
with the fixed  $N_t=1000$ number of runs. Right: dependence of shifts on the number of runs $N_t$ for the same $\epsilon=0.1$; 
observe that it approaches the normal distribution as the number of runs increases.}
\label{NLS_5p_center}
\end{figure}
{\small
\begin{table}[ht]
\begin{tabular}{|c|c|c|c|}
\hline
 $\epsilon$         & $0.01$ & $0.05$ & $0.1$   \\
 \hline
 $\sigma=2$& $1.3e-4$    & $7.1e-4$ & $0.0013$   \\
 \hline
 $\sigma=3$& $0.0016$ &$0.0021$& $0.0024$    \\
 \hline
\end{tabular}
\linebreak
\caption{Multiplicative noise. The variance of the blow-up center shifts $x_c$ in $N_t=1000$ trials, see also Figure \ref{NLS_5p_center}. }
\label{T: x_c var} 
\end{table}
}

In the case of an additive noise we obtain analogous  results; for brevity we only include Figure \ref{SNLS_blow_up_add5p} to show convergence of the profiles, 
the other features remain similar and we omit them. 

\begin{figure}[ht]
\includegraphics[width=0.32\textwidth]{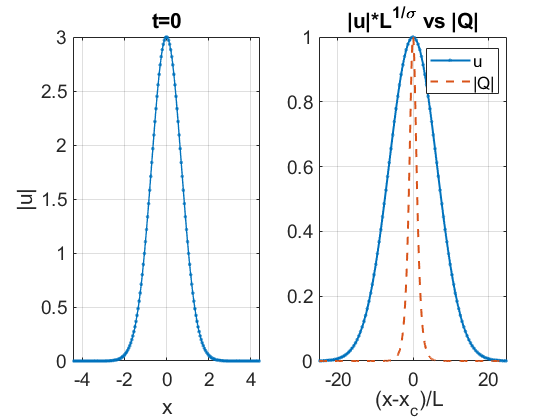}
\includegraphics[width=0.32\textwidth]{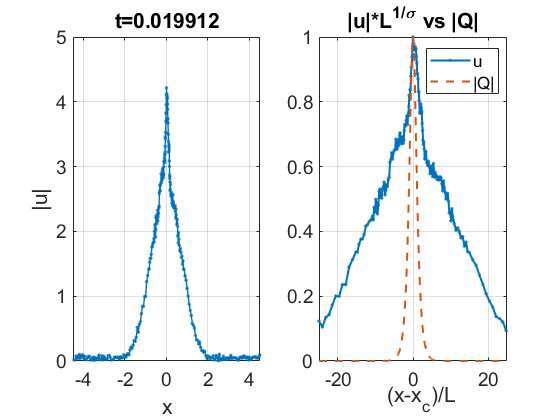}
\includegraphics[width=0.32\textwidth]{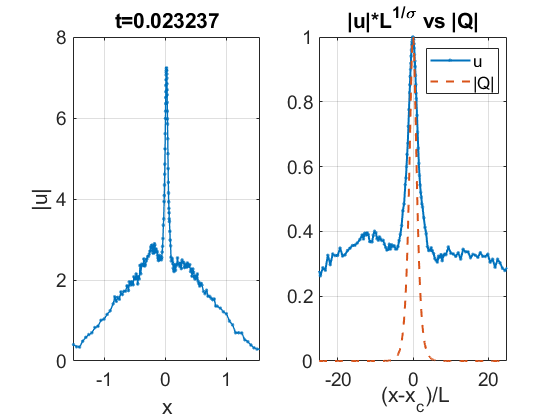}
\includegraphics[width=0.32\textwidth]{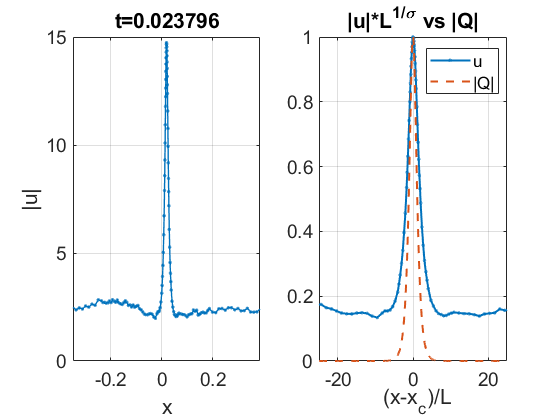}
\includegraphics[width=0.32\textwidth]{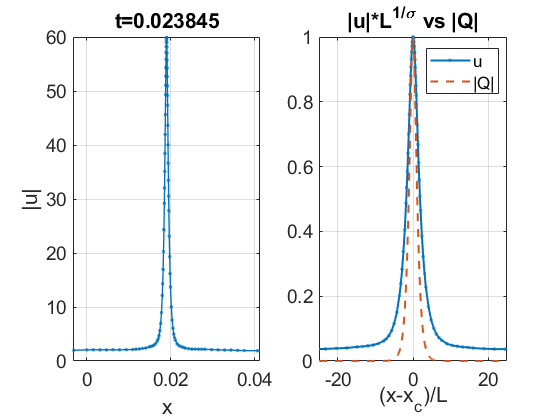}
\includegraphics[width=0.32\textwidth]{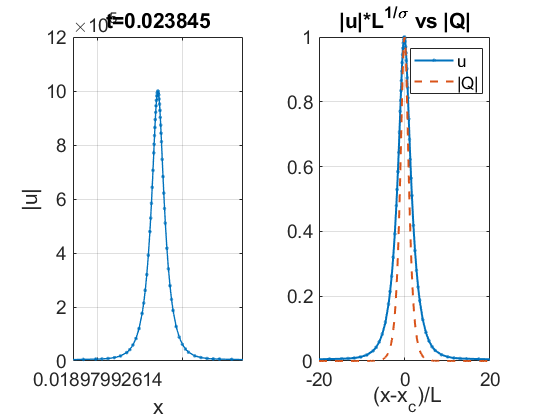}
\caption {Additive noise, $\epsilon=0.1$. 
Formation of blow-up in the $L^2$-critical case ($\sigma=2$): snapshots of a blow-up solution at different times. }
\label{SNLS_blow_up_add5p}
\end{figure}

We conclude that in the $L^2$-critical case, regardless of the type of stochastic perturbation (multiplicative or additive) and the strength 
(different values of $\epsilon$) of the noise, the solution always blows up in a self-similar regime with the rescaled profile of the ground state $Q$ 
and the square root blow-up rate with the logarithmic correction, thus, confirming Conjecture \ref{C:1}.

\subsection{The $L^2$-supercritical case}
In the $L^2$-supercritical case we consider the septic NLS equation ($\sigma=3$) as before with multiplicative or additive noise. We use either Gaussian-type 
initial data $u_0= A \,e^{-x^2}$ or a multiple of the ground state solution $u_0=AQ$, where $Q$ is the ground state solution with $\sigma =3$ in \eqref{E:Q}. 
We consider the multiplicative noise of strength $\epsilon= 0.01, 0.02$ and  $0.1$ and investigate the blow-up profile. For the initial data 
$u_0=3\,e^{-x^2}$ Figure \ref{NLS_7p_profile} shows the solution profiles at different times for $\epsilon=0.1$. The two main observations are: (i) the solution 
smoothes out faster compared to the $L^2$-critical case (see Figure \ref{NLS_5p_profile}); (ii) it converges to a self-similar profile very fast. 
To confirm this we compare the bottom right subplots in both Figure \ref{NLS_5p_profile} and Figure \ref{NLS_7p_profile}: in the supercritical case the profile 
of the rescaled solution (in solid blue) practically coincides with the absolute value of the re-normalized $Q \equiv Q_{1,0}$ (in dashed red);
 this is  similar to the deterministic case.

\begin{figure}[ht]
\includegraphics[width=0.48\textwidth]{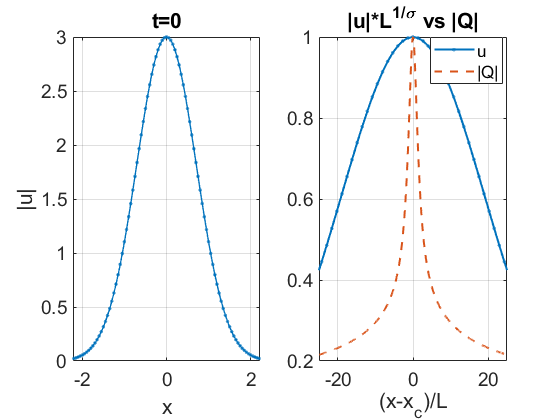}
\includegraphics[width=0.48\textwidth]{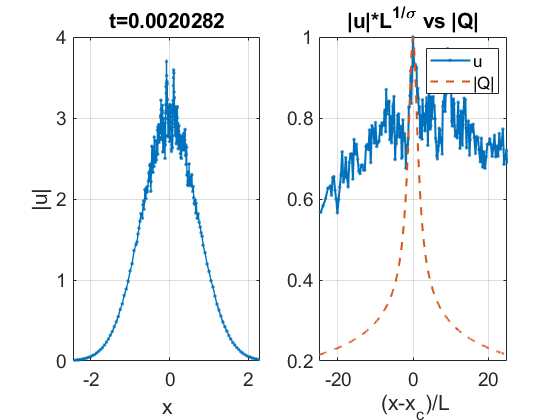}
\includegraphics[width=0.48\textwidth]{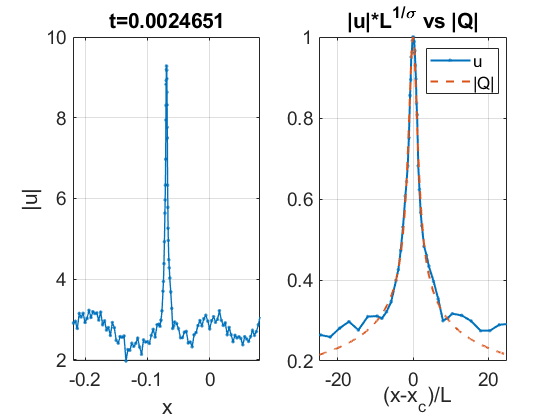}
\includegraphics[width=0.48\textwidth]{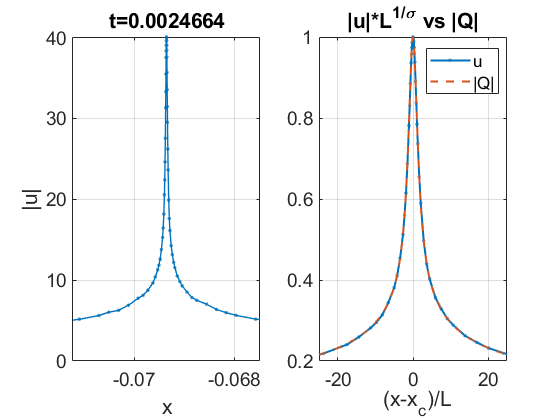}
\caption {Multiplicative noise, $\epsilon=0.1$. Formation of blow-up in the $L^2$-supercritical case ($\sigma=3$): snapshots at different times: 
the actual solution (blue) compared to the rescaled profile $Q_{1,0}$ (red). Note a visibly perfect match in the last right bottom subplot.}
\label{NLS_7p_profile}
\end{figure}

Tests of other data and various values of $\epsilon$ show that
all observed blow-up solutions converge to the profile $Q_{1,0}$.  
In Figure \ref{NLS_7p} we show  the linear fitting for the log dependence of 
$L(t)$ vs. $(T-t)$, which gives the slope $\frac12$. Note that even one trajectory fitting is very good. Further justification of the blow-up rate is done by 
checking the behavior of the quantity $a(\tau)$ from \eqref{E: a}. 
Figure \ref{NLS_7p_a} shows that  the quantity $a(\tau)$ converges to a constant very fast (comparing with the decay to zero of $a(\tau)$ in the $L^2$-critical 
case in Figure \ref{NLS_7p}). Since $a(t) \to a$, a constant, we have 
$a=-LL_t$ and 
solving this ODE (with $L(T)=0$) yields
$$
L(t)=\sqrt{2a(T-t)}.
$$
Recall that  $L(t)=\left(1/\| \nabla u(t)\|_{L^2}\right)^{\frac{2}{\alpha}}$, or equivalently, $L(t)=1/\|u(t)\|_{\infty}^{\sigma}$, thus, we have the blow-up rate \eqref{E:rate-super} 
for the super-critical case, or equivalently,
$$\|u(t)\|_{\infty}=\left(2a(T-t) \right)^{-\frac{1}{2\sigma}} \,\, \mbox{as} \,\, t\rightarrow T, $$
in the case when we evaluate the $L^{\infty}$ norm.
This indicates that solutions blows up with the pure power rate without any logarithmic correction, similar to the deterministic case (for details see \cite{BCR1999}, \cite{YRZ2019}).

\begin{figure}[ht]
\includegraphics[width=0.45\textwidth]{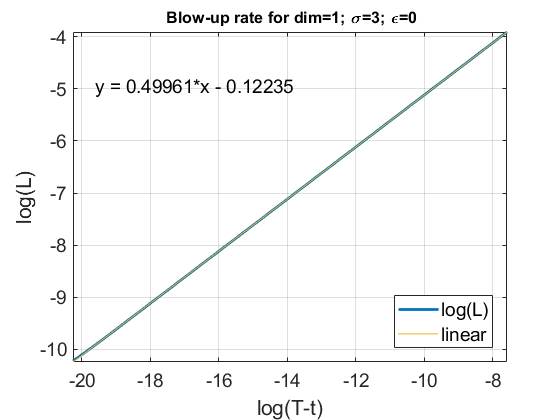}
\includegraphics[width=0.45\textwidth]{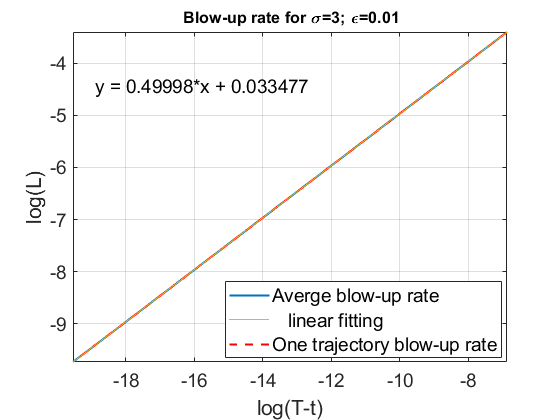}
\includegraphics[width=0.45\textwidth]{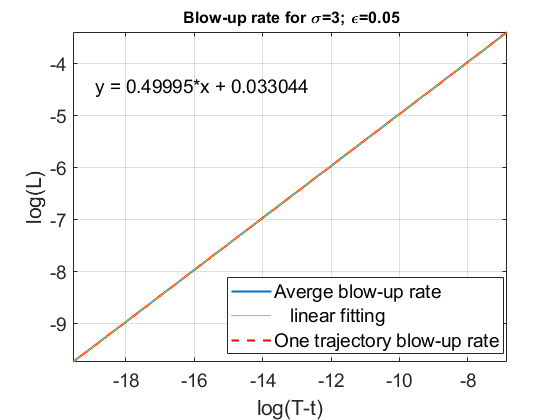}
\includegraphics[width=0.45\textwidth]{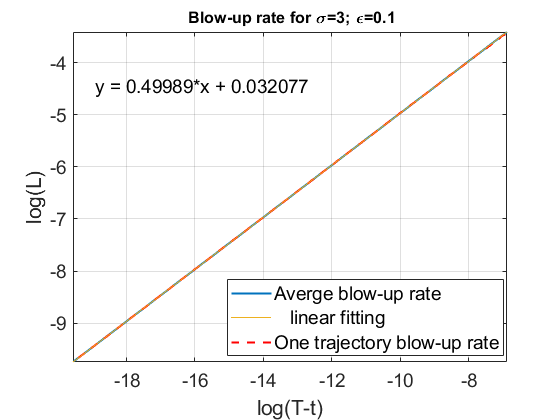}
\caption {Multiplicative noise, $L^2$-supercritical case. A linear fitting of the rate $L(t)$ 
v.s. $(T-t)$ on log scale. The values of the noise strength $\epsilon$ are $0$ (top left), $0.01$ (top right), $0.05$ (bottom left), $0.1$ (bottom right); the linear fitting gives 0.50 slope.}
\label{NLS_7p}
\end{figure}

\begin{figure}[ht]
\includegraphics[width=0.48\textwidth]{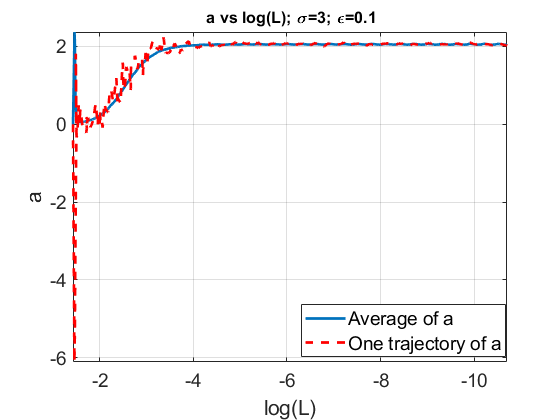}
\includegraphics[width=0.48\textwidth]{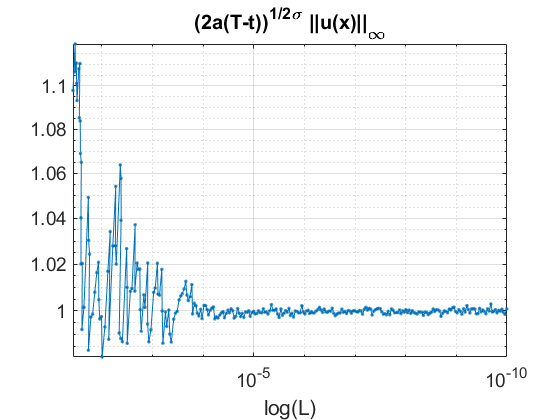}
\caption {Multiplicative noise, $L^2$-supercritical case, $\epsilon=0.1$. Left: $a$ v.s. $\log(L)$, the focusing level. Right: numerical confirmation of the blowup rate $\|u(t)\|_{\infty}=\left(2a(T-t) \right)^{-\frac{1}{2\sigma}}$ (the limit has stabilized at 1).}
\label{NLS_7p_a}
\end{figure}

In the $L^2$-supercritical case 
we also observe shifting of the blow-up center, 
show the distribution of shifts $x_c$ in the multiplicative noise; in particular, these random shifts have a normal distribution similar to the $L^2$-critical case.
The variance of shifts is shown in Table \ref{T: x_c var}.
Note that stronger noises (that is, larger values of $\epsilon$) yield a larger shift away from the origin. Furthermore, comparing Figure \ref{NLS_5p_center} 
with Figure \ref{NLS_7p_center}, we find that the $L^2$-supercritical case produces slightly larger variance of shifts. In other words, we observe that higher power 
of nonlinearity creates   a larger variance, that is the blow-up location is more spread out.

\begin{figure}[ht]
\includegraphics[width=0.48\textwidth]{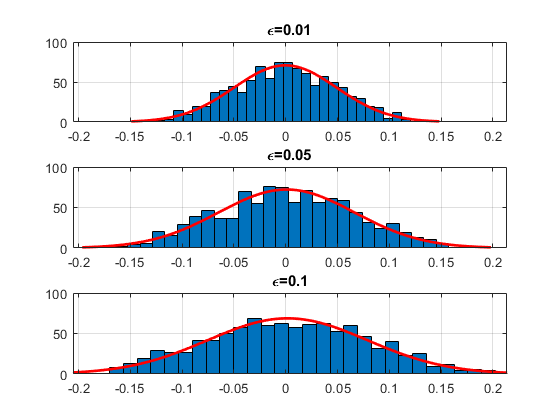}
\includegraphics[width=0.48\textwidth]{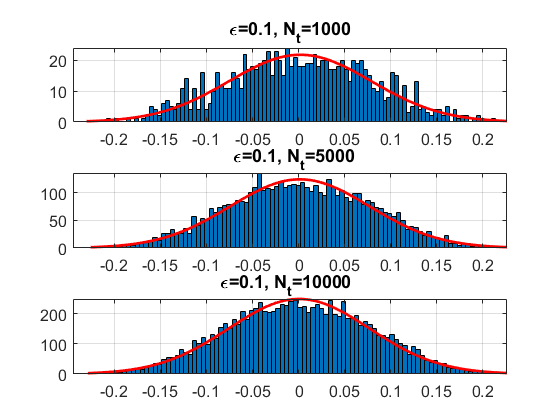}
\caption{Multiplicative noise, $L^2$-supercritical case. Left: distribution of shifts $x_c$ of the blow-up center for different $\epsilon$'s with $N_t=1000$ runs. Right: as $N_t$ increases, it becomes more evident that the spread out of the blow-up location satisfies a normal distribution.}
\label{NLS_7p_center}
\end{figure}

We obtained similar results in the additive noise: the blow up occurs in a self-similar way at the rate $ L(t) =(2a(T-t))^{\frac{1}{2}}$, and the solution profile converges to the profile $Q_{1,0}$ relatively fast, see Figure \ref{SNLS_blow_up_add5p} for profile convergence. The quantities $a(\tau)$, $L(\tau)$ also behave similar to the multiplicative noise parameters (and also to the deterministic cases). 
This confirms Conjecture \ref{C:2}.

\begin{figure}[ht]
\includegraphics[width=0.32\textwidth]{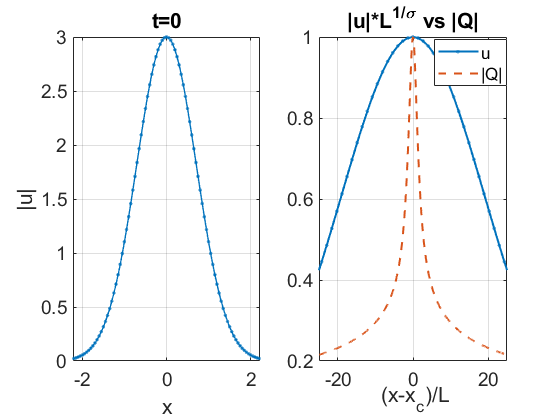}
\includegraphics[width=0.32\textwidth]{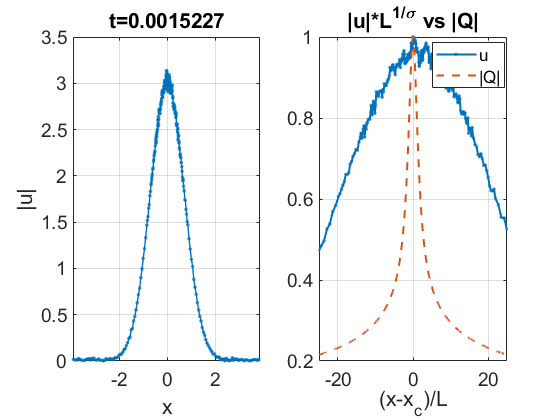}
\includegraphics[width=0.32\textwidth]{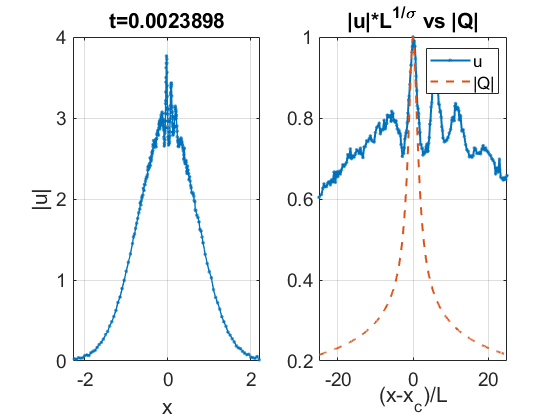}
\includegraphics[width=0.32\textwidth]{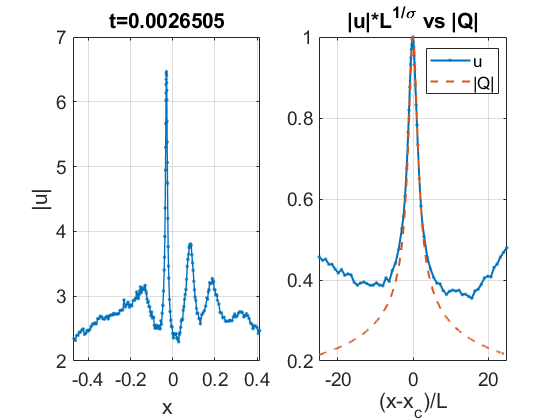}
\includegraphics[width=0.32\textwidth]{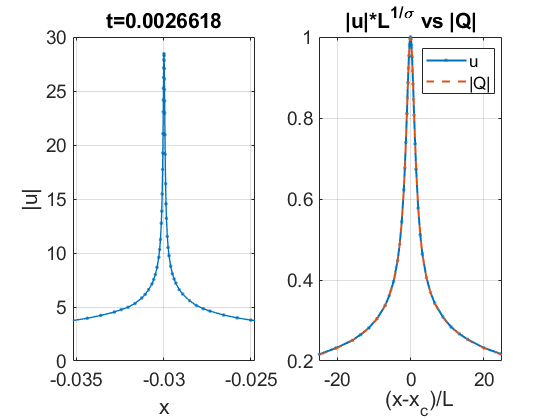}
\includegraphics[width=0.32\textwidth]{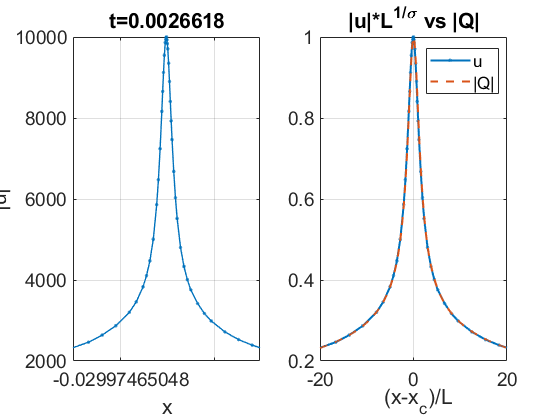}
\caption {Additive noise, $\epsilon=0.1$. Convergence of the blow-up in the $L^2$-supercritical case ($\sigma=3$); actual solution and its rescaled version (blue), the rescaled profile solution $Q_{1,0}$ (red).}
\label{SNLS_blow_up_add7p}
\end{figure}

\section{Conclusion}\label{conclusion}

In this work we investigate the behavior of solutions to the 1d focusing SNLS subject to a stochastic perturbation which is either multiplicative or additive,
and driven by  space-time white noise. In particular, we study the time dependence of the mass ($L^2$-norm) and the energy (Hamiltonian) in the $L^2$-critical 
and supercritical cases. For that we consider a discretized version of both quantities and an approximation of the actual mass or energy. 
In the deterministic case these quantities are conserved in time, however, it is not necessarily the case in the stochastic setting.
 In  the case of a multiplicative noise, which is defined in terms of the Stratonovich integral, 
the mass (both discrete and actual) is invariant. However, in the additive case the mass grows linearly. 
The energy grows in time in both stochastic settings. We give upper estimates on that time dependence 
and then track it numerically; we observe that energy levels off when the noise is multiplicative. 
We also investigate the dependence of the mass and energy on the strength of the noise, on the spatial and temporal mesh refinements and the length of the 
computational interval.  

For the above we use three different numerical schemes; all of them conserve discrete mass in the multiplicative noise setting, and one of them conserves the 
discrete energy in the deterministic setting, though that scheme involves fixed point iterations to handle the nonlinear system, thus, taking longer computational time. 
We introduce a new scheme, a linear extrapolation of the above and Crank-Nickolson discretization of the potential term, which speeds up significantly our 
computations, since the scheme is linear, and thus, avoiding extra fixed point iterations while having tolerable errors. 

We also introduce a new algorithm in order to investigate the blow-up dynamics. Typically in the deterministic setting to track the blow-up dynamics, 
the dynamic rescaling method is used. We use instead  a finite difference method with non-uniform mesh and then mesh-refinement with mass-conservative 
interpolation. 
With this algorithm we are able to track the blow-up rate, profile and we find a new feature in the blow-up dynamics, the shift of the blow-up center, which follows 
normal distribution for large number of trials. We note that our algorithm is also applicable for the deterministic NLS equation, 
in particular, it can replace the dynamic rescaling or moving mesh methods used to track blow-up.  

We confirm previous results of Debussche et al. \cite{DM2002a}, \cite{dBD2002c}, \cite{dBD2005} showing that the additive noise can amplify or create blow-up (we suspect that this happens almost surely for any data) in the $L^2$-critical and supercritical cases. In the multiplicative noise setting the blow-up seems to occur for any (sufficiently localized) data in the $L^2$-supercritical case, and above the mass threshold in the $L^2$-critical case. 
Finally, when the noise is present, a solution is likely to travel away from the initial `center',  and, 
 once the solution starts blowing up, the noise plays no role in the singularity structure, and the blow-up occurs with the rate and profile
  similar to the deterministic setting.

\bibliography{SNLS_bib}
\bibliographystyle{abbrv}

\end{document}